\documentclass[11pt,final]{article} 

\usepackage{amsfonts,amsthm,amsmath,amssymb, mathtools,bm}
\usepackage{dsfont}
\usepackage{graphicx,color}
\usepackage{hyperref}
\usepackage{a4wide}
\usepackage{multirow}
\usepackage{cleveref}
\usepackage{enumitem}
\usepackage{tikz-cd}
\usepackage{subcaption}
\usepackage{comment}
\usepackage{bbm}

\newtheorem{proposition}{Proposition}[section]
\newtheorem{theorem}[proposition]{Theorem}
\newtheorem{lemma}[proposition]{Lemma}
\newtheorem{corollary}[proposition]{Corollary}
\newtheorem{definition}[proposition]{Definition}

\newtheorem{example}[proposition]{Example}

\renewenvironment{proof}{\medskip\noindent{\textbf{Proof.}}%
  \hspace{1pt}}{\hspace{-5pt}{\nobreak\quad\nobreak\hfill\nobreak%
    $\square$\vspace{2pt}\par}\smallskip\goodbreak}

\newenvironment{proofof}[1]{\smallskip\noindent{\textbf{Proof~of~#1.}}%
  \hspace{1pt}}{\hspace{-5pt}{\nobreak\quad\nobreak\hfill\nobreak%
    $\square$\vspace{2pt}\par}\smallskip\goodbreak}

\numberwithin{equation}{section}
\numberwithin{figure}{section}
\numberwithin{table}{section}

\allowdisplaybreaks[4]

\renewcommand{\phi}{\varphi}
\renewcommand{\theta}{\vartheta}
\renewcommand{\epsilon}{\varepsilon}
\renewcommand{\L}[1]{\mathbf{L^#1}}

\renewcommand{\d}[1]{\mathinner{\mathrm{d}{#1}}}

\renewcommand{\L}[1]{\mathbf{L^#1}}
\newcommand{\dd}[1]{\mathinner{\mathrm{d}{#1}}}
\newcommand{\Lloc}[1]{\mathbf{L^{#1}_{loc}}}

\newcommand{\C}[1]{\mathbf{C^{#1}}}

\newcommand{\Cc}[1]{\mathbf{C_c^{#1}}}

\newcommand{\W}[2]{\mathbf{W^{#1,#2}}}
\newcommand{\BV}{\mathbf{BV}}
\newcommand{\cBV}[1]{\mathcal{BV}^{#1}}

\newcommand{\modulo}[1]{{\left|#1\right|}}
\newcommand{\norma}[1]{{\left\|#1\right\|}}
\newcommand{\reali}{{\mathbb{R}}}
\newcommand{\naturali}{{\mathbb{N}}}

\newcommand{\tv}{\mathop\mathrm{TV}}

\newcommand{\esssup}{\mathop\mathrm{ess~sup}}

\newcommand{\sgn}{\mathop\mathrm{sgn}}
\newcommand{\spt}{\mathop\mathrm{spt}}
\newcommand{\Id}{\mathop{\mathbf{Id}}}

\newcommand{\caratt}[1]{\chi_{\strut#1}}

\newcommand{\ghost}[1]{}

\begin{document}

\title{Coexisting Automated and Human-Driven Vehicles: \\
  Well-Posedness \\ of a Mixed Nonlocal-Local Traffic Model}

\author{Rinaldo M.~Colombo$^1$ \and Mauro Garavello$^2$ \and Claudia
  Nocita$^2$}

\maketitle

\footnotetext[1]{Unit\`a INdAM \& Dipartimento di Ingegneria
  dell'Informazione, Universit\`a di Brescia, Italy.\hfill\\
  \texttt{rinaldo.colombo@unibs.it}}

\footnotetext[2]{Dipartimento di Matematica e Applicazioni,
  Universit\`a di Milano -- Bicocca, Italy.\hfill\\
  \texttt{mauro.garavello@unimib.it} and
  \texttt{c.nocita@campus.unimib.it}}

\begin{abstract}
  \noindent We present a macroscopic traffic flow model where standard
  vehicles coexist with vehicles informed on the traffic
  distribution. The resulting mixed nonlocal-local
  integro-differential PDEs is proved to generate a locally Lipschitz
  continuous semigroup whose orbits are uniquely characterized as
  solutions to the system, according to a natural definition of
  solution. The norms and function spaces adopted are intrinsic to the
  different nature of the equations.

  \medskip

  \medskip

  \noindent\textbf{MSC~2020:} 76A30, 35L65, 35R09.
\end{abstract}

\section{Introduction}
\label{sec:introduction}

The forthcoming large scale introduction of autonomous vehicles (AVs)
leads to consider vehicles that are \emph{nonlocal}, in the sense that
they are aware of traffic conditions also at a significant distance
from their location, both in front and behind them. Analytically, this
leads to nonlocal models, i.e., to models where integrals in the space
variable allow a vehicle to react to suitable averages of the traffic
density around --- in front and behind --- its position. The current
literature offers a variety of models built on these premises,
see~\cite{aggarwal2023systems, aggarwal2023wellposedness, MR3447130,
  MR4731004, MR4238781, MR3670045}.

At the same time, the presence of standard --- or \emph{local} ---
vehicles, i.e., aware of traffic conditions only at their position,
can not be neglected. Indeed, the interplay between \emph{local} and
\emph{nonlocal} vehicles (or individuals) is being considered in the
traffic community,
see~\cite{CUThayat2023trafficsmoothingusingexplicit,
  CUTlee2024trafficcontrolconnectedautomated, Rasouli2020900,
  Schieben201969, Sun2023}, see also~\cite{MR4104950, MR4720869} for
an analytical approach.

In the model we present here, these two kinds of drivers/vehicles are
traveling simultaneously along the same road and, hence, they are
interacting. The time and space dependent density of the \emph{local}
ones is $r = r (t,x)$, while the \emph{nonlocal} ones are described by
$\rho = \rho(t,x)$, $t$ being time and $x$ the one dimensional
coordinate along the road. This results in the following natural extension of
the classical LWR model~\cite{LighthillWhitham, Richards}
\begin{equation}
  \label{eq:5}
  \left\{
    \begin{array}{l}
      \partial_t \rho
      +
      \partial_x \left(\rho \;
      v_{NL}\left((\rho + r)  *\eta\right)\right)
      =0
      \\
      \partial_t r
      + \partial_x \left(r \; v_L\left(\rho+r \right)\right)
      = 0.
    \end{array}
  \right.
\end{equation}
The functions $v_{NL}$ and $v_L$ are the \emph{speed laws} of the two
populations, while $\eta$ is a general averaging kernel. We underline
that $\eta$ is required neither to have compact support, nor to enjoy
any sort of symmetry. In particular, the forward horizon can very well
be far larger than the backward one, see for
instance~\cite[\S~3]{OperaPrima}. On the other hand, $\eta$ may not be
monotone, so that no \emph{a priori} constant bound on the nonlocal
density $\rho$ is in general available for all times, though
$\L\infty$ bounds are available on all compact time intervals. On the
contrary, in the present setting, the local density $r$ is proved to
remain for all times within the interval of the physical densities.

From the modeling point of view, \eqref{eq:5} provides an environment
not only for traffic descriptions, but also for investigating the
interplay between AVs and standard \emph{``short sighted''}
(bulk) vehicles. According to a common approach,
see~\cite{CUThayat2023trafficsmoothingusingexplicit,
  CUTlee2024trafficcontrolconnectedautomated}, $v_{NL}$ acts as the
strategy assigned to the AVs in order to control the evolution of the
whole traffic. It is then necessary to allow for the coexistence of
AVs following different behavioral rules. We are thus led to
extend~\eqref{eq:5} to the case of different classes, say $k$, of AVs:
\begin{equation}
  \label{eq:1}
  \left\{
    \begin{array}{l}
      \partial_t \rho^i
      +
      \partial_x \left(\rho^i \;
      v_{NL}^i\left(\left(\sum\limits_{j=1}^k \rho^j + r\right) * \eta^i\right)\right)
      =0
      \qquad\qquad i=1, \ldots, k
      \\
      \partial_t r
      + \partial_x \left(r \; v_L\left(\sum\limits_{j=1}^k \rho^j + r \right)\right)
      = 0.
    \end{array}
  \right.
\end{equation}

Here we focus on the analytic properties of~\eqref{eq:1}, providing
its global in time well-posedness and a full set of \emph{a priori}
estimates. The different nature of the equations in~\eqref{eq:1}
imposes the use of different techniques, relying on different
assumptions. In particular, the \emph{natural} regularities of
$x \mapsto \rho (t,x)$ and of $x \mapsto r (t,x)$ turn out to be
different, as also the norms that allow well-posedness and stability
estimates.
More precisely, the initial data $(\rho_o, r_o)$ to be assigned
to~\eqref{eq:1} as well as the solutions $x \mapsto \rho (t,x)$ and
$x \mapsto r (t,x)$ find their natural environment in the sets
\begin{equation}
  \label{eq:75}
  \cBV{1} (\reali;\reali^k)
  {\coloneqq}
  \left\{
    \rho \in \W11 (\reali; \reali^k) \colon
    \rho' \in \BV{} (\reali; \reali^k)
  \right\}
  \mbox{ and }
  \cBV{} (\reali;\reali)
  {\coloneqq}
  (\L1 \cap \BV{}) (\reali;\reali)
\end{equation}
whereas stability estimates holds in the $\W11 (\reali; \reali^k)$ and
$\L1 (\reali; \reali)$ norms. Indeed, the local Lipschitz continuous
dependence of solutions with respect to the initial datum and with
respect to time holds in the $\W11 (\reali; \reali^k)$ and
$\L1 (\reali; \reali)$ topologies, provided the norms
\begin{equation}
  \label{eq:64}
  \norma{\rho}_{\cBV{1} (\reali;\reali^k)}
  \coloneqq
  \norma{\rho}_{\W11 (\reali;\reali^k)} + \tv (\rho')
  \quad \mbox{ and } \quad
  \norma{r}_{\cBV{} (\reali;\reali)}
  \coloneqq
  \norma{r}_{\L1 (\reali;\reali)} + \tv (r)
\end{equation}
are bounded. In particular, Example~\ref{ex:1} below shows that the
use of the $\L1$ norm on $\rho$ may not ensure the continuous
dependence of the solutions to~\eqref{eq:1} on the initial datum.

The motivation for these differences lies in the need to use {Kru\v
  zkov}~\cite{MR267257} definition of solution for the local equation,
since~\eqref{eq:1} obviously comprises the case of the scalar
conservation law
$\partial_t r + \partial_x \left(r \, v_L (r)\right) = 0$. In turn,
this imposes $x \mapsto \rho (t,x)$ and $x \mapsto \rho_o (x)$ to be
at least in $\W11 (\reali; \reali^k)$. Once this regularity is
achieved, {Kru\v zkov} theory~\cite{MR267257}, as well as the
stability estimates~\cite{MR2512832}, can be exploited to deal with
the local equation in the proof of the well-posedness
of~\eqref{eq:1}. On the other hand, to deal with the nonlocal
equation, our starting point are the well-posedness result and the
stability estimates in~\cite{OperaPrima}, which we here improve to
suit to the present setting. All this leads to the asymmetry between
the assumptions required on the local and nonlocal equations and data
--- see~\ref{item:17} vs.~\ref{item:21} and
$\rho_o \in \cBV1 (\reali; \reali^k)$
vs.~$r_o \in \cBV{}(\reali;\reali)$. In particular, the present lack
of a well-posedness and stability theory for general systems of
conservation laws prevents us from allowing $r$ to be vector valued.

In this connection, we underline a natural problem arising from the
present result. The so-called \emph{nonlocal-to-local limit} recently
attracted a relevant attention in the literature, see~\cite{ bressan_entropy_2021, MR4651679, zbMATH07839635,
  zbMATH07658225, MR3961295, MR4613802, MR3944408} and the references
therein. An analogous result in the case of~\eqref{eq:5}, i.e.,
letting $\eta$ tend to Dirac delta, would yield the existence of solutions
to the LWR model with $2$ populations introduced in~\cite{MR2020123},
whose well-posedness is, to our knowledge, still an open problem. The case~\eqref{eq:1} of more than $2$ populations might then follow.

The proof below relies on careful estimates, separately, on the local
and on the nonlocal equations constituting~\eqref{eq:1}. As a first
step, more regular initial data are considered. Adapting and, where
necessary, improving the available well-posedness and stability
estimates in~\cite{OperaPrima, MR2512832} allow to ignite a recursive
procedure that yields a Cauchy sequence converging to a solution
to~\eqref{eq:1}, first locally in time and then globally. Careful
\emph{a priori} estimates ensure the (local) Lipschitz continuous
dependence of solutions on time and on the initial data. A final
bootstrap procedure allows to relax the assumptions on the initial
data and leads to the existence result presented in
Theorem~\ref{thm:main}. Theorem~\ref{thm:unique} proves uniqueness of
solutions to~\eqref{eq:1}, also making use of the basic estimates on
the $2$ separate problems.

\goodbreak

The next section presents the main results while all proofs are
deferred to Section~\ref{sec:analytical-proofs}.

\section{Main Results}
\label{sec:main-results}

Throughout, $T>0$ is fixed and $k \in \naturali \setminus\{0\}$. We
also shorten the notation for the sum $\rho^1 + \cdots+\rho^k$ setting
$\Sigma\rho \coloneqq \sum_{j=1}^k \rho^j$. We denote by
$\mathcal C{(\cdots)}$ a continuous and positive function non
decreasing in each of its arguments, the exact value of which is not
relevant.

We defer to Lemma~\ref{lem:lui} a characterization of
$\cBV{1} (\reali;\reali^k)$, as defined in~\eqref{eq:75}. Here, we
note that clearly
$\W21 (\reali;\reali^k) \subsetneq \cBV{1} (\reali;\reali^k)
\subsetneq \W11 (\reali;\reali^k)$. Throughout, we use the norm
\begin{equation}
  \label{eq:66}
  \norma{(\rho,r)}_{\W11(\reali; \reali^k)
    \times
    \L1(\reali; \reali)}
  \coloneq
  \norma{\rho}_{\W11(\reali; \reali^k)} +
  \norma{r}_{\L1(\reali; \reali)} \,.
\end{equation}

\begin{definition}
  \label{def:sol}
  By \emph{solution} to~\eqref{eq:1} with initial datum
  $\rho_o \in \W11 (\reali; \reali^k)$, $r_o \in \L1 (\reali; \reali)$
  on the interval $[0,T]$, we mean
  $ (\rho,r) \in \C0 \left([0, T];\W11 (\reali;\reali^k) \times \L1
    (\reali; \reali)\right)$ such that setting
  \begin{eqnarray*}
    v^i (t,x)
    & \coloneq
    & v_{NL}^i\left(\left(\Sigma \rho  (t) * \eta^i\right) (x)
      + \left(r (t)*\eta^i\right) (x)\right) \,,
    \\
    w (t,x,r)
    & \coloneq
    & v_L\left(\Sigma \rho (t,x)  + r\right) \,,
  \end{eqnarray*}
  $\rho^i$, for $i=1, \ldots, k$, is a distributional
  solution~\cite[Definition~4.2]{MR1816648} to
  \begin{equation}
    \label{eq:53}
    \left\{
      \begin{array}{l}
        \partial_t \rho^i + \partial_x \left(\rho^i \, v^i (t,x)\right) =0
        \\
        \rho^i (0,x) = \rho_o^i (x)
      \end{array}
    \right.
  \end{equation}
  and $r$ is a Kru\v zkov solution~\cite[Definition~1]{MR267257} to
  \begin{equation}
    \label{eq:29}
    \left\{
      \begin{array}{l}
        \partial_t r + \partial_x \left(r \, w (t,x,r)\right) =0
        \\
        r (0,x) = r_o (x) \,.
      \end{array}
    \right.
  \end{equation}
\end{definition}
Remark that the necessity for a higher regularity in $\rho$ stems from
the requirement that $r$ be a Kru\v zkov solution. Indeed, the
standard definition of solution to~\eqref{eq:29},
see~\cite[Definition~1]{MR267257}, requires that the flux
$(t,x,r) \mapsto r\, w (t,x,r)$ be differentiable with respect to the
space variable, hence that $x \mapsto \rho (t,x)$ be in
$\W11 (\reali;\reali^k)$. A further motivation of this regularity is
in~\ref{lem:conv} in the proof of Theorem~\ref{thm:main}, see~\Cref{subsec:coupled-problem}.

The higher regularity of the solution to the nonlocal problem
justifies a regularity of the nonlocal velocity $v_{NL}$ higher than
that of the local velocity $v_L$.  Below, we assume that the speed
laws $v_{NL}$ and $v_L$ in~\eqref{eq:1} and the averaging kernel
$\eta$ meet the following requirements:
\begin{enumerate}[label=\bf(vNL)]
\item \label{item:17}
  $v_{NL} \in (\C3\cap \W3\infty)(\reali;\reali^k)$.
\end{enumerate}

\begin{enumerate}[label=\bf(vL)]
\item \label{item:21} $v_L \in \C2(\reali; \reali)$ and there
  exist $V_L,R_L > 0$ such that $v_L(r) = V_L$ for
  $r \leq 0$ and $v_L(r) = 0$ for $r \geq R_L$.
\end{enumerate}

\begin{enumerate}[label=$\bm{(\eta)}$]
\item \label{item:hyp-eta}
  $\eta \in (\C3 \cap \W3\infty)(\reali;\reali^k)$.
\end{enumerate}

\noindent We first present a stability property of
Definition~\ref{def:sol}.

\begin{proposition}
  \label{lem:GoodSol}
  Assume~\ref{item:17}, \ref{item:21} and~\ref{item:hyp-eta}. Fix
  sequences
  $(\rho_{o,n}, r_{o,n}) \in \W11 (\reali;\reali^k) \times \L1
  (\reali;\reali)$,
  $(\epsilon_n,e_n) \in \C0 \left([0,T];\W11 (\reali; \reali^k) \times
    \L1 (\reali;\reali)\right)$ and a function
  $(\rho_\infty, r_\infty)$ belonging to
  $ \C0\left([0,T];\W11 (\reali;\reali^k) \times \L1
    (\reali;\reali)\right)$ such that for all $n \in \naturali$ the
  problem
  \begin{displaymath}
    \begin{cases}
      \partial_t\rho_n^i
      +
      \partial_x\left(
      \rho_n^i \,
      v^i_{NL}\left(
      \Sigma \rho_n * \eta^i
      +
      (r_n + e_n)  * \eta^i
      \right)\right)=0 \qquad i=1, \ldots, k
      \\
      \partial_t r_n
      +
      \partial_x\left(
      r_n \, v_L\left( \Sigma (\rho_n+\epsilon_n) + r_n \right)
      \right)=0
      \\
      \rho_n(0,x) = \rho_{o,n}(x)
      \\
      r_n(0,x)=r_{o,n}(x)
    \end{cases}
  \end{displaymath}
  admits a solution $(\rho_n,r_n)$ in the sense of
  Definition~\ref{def:sol}. Moreover, assume that
  \begin{equation}
    \label{eq:70}
    \begin{array}{@{}rcll@{}}
      \displaystyle
      \lim_{n\to+\infty} (\rho_n,r_n)
      & =
      & (\rho_\infty,r_\infty)
      & \mbox{ in }
        \C0 \left([0,T]; \W11(\reali;\reali^k) \times
        \L1(\reali;\reali)\right)\,;
      \\
      \displaystyle
      \lim_{n\to+\infty} (\epsilon_n,e_n)
      & =
      & (0,0)
      & \mbox{ in }
        \C0 \left([0,T]; \W11(\reali;\reali^k) \times
        \L1(\reali;\reali)\right)\,.
    \end{array}
  \end{equation}
  Then, $(\rho_\infty,r_\infty)$ solves, in the sense of
  Definition~\ref{def:sol}, the problem
  \begin{equation}
    \label{eq:71}
    \begin{cases}
      \partial_t\rho^i
      +
      \partial_x\left(
      \rho^i \,
      v^i_{NL}\left(
      \Sigma \rho * \eta^i
      +
      r   * \eta^i
      \right)\right)=0 \qquad i=1, \ldots, k
      \\
      \partial_t r
      +
      \partial_x\left(
      r \, v_L\left( \Sigma \rho + r \right)
      \right)=0
      \\
      \rho(0,x) = \rho_\infty(0,x)
      \\
      r(0,x)=r_\infty(0,x) \,.
    \end{cases}
  \end{equation}
\end{proposition}

We are now ready to state the existence and stability result for solutions to~\eqref{eq:1}.

\begin{theorem}
  \label{thm:main}
  Let~\ref{item:17}, \ref{item:21} and~\ref{item:hyp-eta} hold. Then,
  problem~\eqref{eq:1} generates a unique map
  \begin{equation}
    \label{eq:79}
    \mathcal{S} \colon \reali_+ \times
    \left(
      \cBV{1}(\reali; \reali^k)
      \times
      \cBV{}(\reali; \reali)\right)
    \to
    \cBV{1}(\reali; \reali^k)
    \times
    \cBV{}(\reali; \reali)
  \end{equation}
  such that
  \begin{enumerate}[label=\textup{\textbf{($\mathcal{S}$\arabic*)}},
    ref=\textup{\textbf{($\mathbf{\mathcal{S}}$\arabic*)}}]
  \item \label{item:1}$\mathcal{S}$ is a global semigroup, i.e.,
    $\mathcal{S}_o = \Id$ and for all $t,\widehat t \in \reali_+$,
    $\mathcal{S}_{t} \circ \mathcal{S}_{\widehat t} =
    \mathcal{S}_{t+\widehat t}$.
  \item \label{item:23} For any
    $(\rho_o,r_o) \in \left( \cBV{1}(\reali; \reali^k) \times
      \cBV{}(\reali; \reali)\right)$ and any $\bar T > 0$, the orbit
    $t \mapsto \mathcal{S}_t (\rho_o, r_o)$ solves~\eqref{eq:1} on
    $[0, \bar T]$ in the sense of \Cref{def:sol}.
  \item \label{item:24} The semigroup $\mathcal{S}$ is locally
    Lipschitz continuous with respect to the norm~\eqref{eq:66}: for
    all $M > 0$, there exists a constant
    \begin{displaymath}
      \mathcal{C}_1 \coloneqq
      \mathcal{C} \left(
        \norma{\eta}_{\W3\infty(\reali; \reali^k)},
        R_L,
        \norma{v_L'}_{\W1\infty(\reali; \reali)},
        \norma{v_{NL}}_{\W3\infty(\reali; \reali^k)},
        M, T
      \right)
    \end{displaymath}
    such that for all $t, \widehat t \in [0,T]$, for all
    $(\rho_o, r_o), \, (\widehat \rho_o, \widehat r_o) \in
    \cBV{1}(\reali; \reali^k) \times \cBV{}(\reali; \reali)$ with
    \begin{displaymath}
      \begin{array}{r@{\,}c@{\,}l@{\qquad}r@{\,}c@{\,}l}
        \norma{\rho_o}_{\cBV{1} (\reali; \reali^k)}
        & \leq
        & M \,,
        & \norma{r_o}_{\cBV{} (\reali; \reali)}
        & \leq
        & M \,,
        \\
        \norma{\widehat \rho_o}_{\cBV{1} (\reali; \reali^k)}
        & \leq
        & M \,,
        & \norma{\widehat r_o}_{\L1 (\reali; \reali)}
        & \leq
        & M \,,
      \end{array}
    \end{displaymath}
    the following estimates hold:
    \begin{displaymath}
      \begin{array}{@{}r@{\,}c@{}l@{}}
        \norma{ \mathcal{S}_t(\rho_o, r_o)
        {-}
        \mathcal{S}_t(\widehat\rho_o,\widehat r_o) }_{\W11(\reali; \reali^k)
        \times
        \L1(\reali; \reali)}
        & \leq
        & (1 \!+\!\mathcal{C}_1 t)
        \norma{(\rho_o {-} \widehat\rho_o,
        r_o {-} \widehat r_o)}_{\W11(\reali; \reali^k)
        \times
        \L1(\reali; \reali)}\,,
        \\
        \norma{ \mathcal{S}_t(\rho_o, r_o)
        {-}
        \mathcal{S}_{\, \hat t}(\rho_o, r_o) }_{\W11(\reali; \reali^k)
        \times
        \L1(\reali; \reali)}
        & \leq
        & \mathcal{C}_1 \, \modulo{t - \widehat t}.
      \end{array}
    \end{displaymath}
  \item \label{item:3} For all $t \in [0,T]$ and for all
    $(\rho_o, r_o) \in \cBV{1}(\reali; \reali^k) \times \cBV{}
    (\reali; \reali)$, call
    $\left(\rho(t), r(t)\right) \coloneqq \mathcal{S}_t(\rho_o, r_o)$. Then,
    there exists a positive constant
    \begin{displaymath}
      \mathcal{C}_2 \coloneqq
      \mathcal{C} \left(
        \begin{array}{c}
          \norma{\eta}_{\W3\infty(\reali; \reali^k)},
          \norma{v'_{L}}_{\W1\infty(\reali; \reali)},
          \norma{v_{NL}}_{\W3\infty(\reali; \reali^k)},
          \\
          \norma{\rho_o}_{\cBV{1}(\reali; \reali^k)},
          \norma{r_o}_{\cBV{}(\reali; \reali)},
          T
        \end{array}
      \right)
    \end{displaymath}
    such that
    \begin{equation*}
      \begin{array}{l@{\qquad}l}
        \norma{\rho(t)}_{\L1(\reali; \reali^k)} =
        \norma{\rho_o}_{\L1(\reali; \reali^k)},
        & \norma{r(t)}_{\L1(\reali; \reali )} \leq
          \norma{r_o}_{\L1(\reali; \reali )},
        \\[.2cm]
        \norma{\rho(t)}_{\cBV{1}(\reali; \reali^k)} \leq
        \mathcal{C}_2,
        & \tv\left(r(t)\right) \leq
          \mathcal{C}_2.
      \end{array}
    \end{equation*}

  \item \label{item:13} Call
    $\left(\rho(t), r(t)\right) \coloneqq \mathcal{S}_t(\rho_o, r_o)$,
    for
    $(\rho_o, r_o) \in \cBV{1}(\reali; \reali^k) \times \cBV{}
    (\reali; \reali)$. If $\rho_o^i \geq 0$ for $i=1, \ldots, k$ and
    $r_o \in [0, R_L]$, then $\rho^i (t,x) \geq 0$ and
    $r (t,x) \in [0, R_L]$ for a.e.~$(t,x) \in \reali_+ \times \reali$
    and
    $\norma{r (t)}_{\L1 (\reali; \reali)} = \norma{r_o}_{\L1
      (\reali,\reali)}$.
  \end{enumerate}
\end{theorem}

\noindent The proof is deferred to~\S~\ref{subsec:coupled-problem}.

The asymmetry in~\ref{item:24} in the local initial conditions $r_o$
and $\widehat r_o$ is consequence of the application
of~\cite[Theorem~2.6]{MR2512832} about the stability of solutions to
local problems with respect to the fluxes. We take advantage of this
asymmetry in the uniqueness proof in Theorem~\ref{thm:unique}.

Note that the constant $\mathcal{C}_1$ in~\ref{item:24} actually
depends also on $\norma{v_L}_{\L\infty (\reali;\reali)}$, as is to be
expected. Indeed, by~\ref{item:21},
$\norma{v_L}_{\L\infty (\reali;\reali)} \leq \mathcal{C}
\left(R_L,\norma{v'_L}_{\L\infty(\reali; \reali)}\right)$.

The necessity of the requirement $\rho_o \in \W11(\reali; \reali^k)$
and the stability with respect to the $\W11$-norm of the nonlocal
initial datum are illustrated by the following example.
\begin{example}
  \label{ex:1}
  Set $k=1$ and choose the speed laws $v_{NL} \equiv 0$ and $v_L$
  satisfying~\ref{item:21} such that $v_L (r) = 1-r$ for
  $r \in [1/2,1]$.  Define for all $n \in \naturali$
  \begin{equation*}
    \rho_{o,n} \coloneq
    \left\{
      \begin{array}{cr @{\,}c@{\,}l}
        0
        & x
        & <
        & 1-1/n
        \\
        1+n(x-1)
        & x
        & \in
        & [1-1/n,1]
        \\
        1-n(x-1)
        & x
        & \in
        & \mathopen]1, 1+1/n]
        \\
        0
        & x
        & >
        & 1+1/n
      \end{array}
    \right.
    \qquad
    r_{o,n} \coloneq \left\{
      \begin{array}{cr@{\,}c@{\,}l}
        0
        & x
        & <
        & 0
        \\
        1
        & x
        & \in
        & [0, 1-1/n]
        \\
        n(1-x)
        & x
        & \in
        & \mathopen]1-1/n, 1]
        \\
        0
        & x
        & >
        & 1
      \end{array}
    \right.
  \end{equation*}
  see Figure~\ref{fig:init-cond}.
  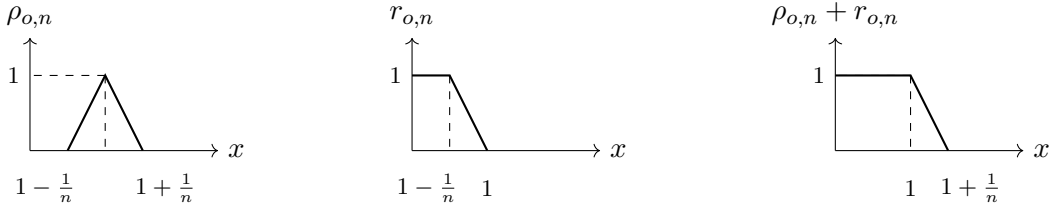
\begin{figure}[h!]
    \centering
    \begin{tikzpicture}
      \draw[->] (0,0) -- (2.5,0) node[right] {$x$}; %
      \draw[->] (0,0) -- (0,1.5) node[above] {$\rho_{o,n}$}; %
      \draw[thick] (1 - 1/2 ,0) -- (1,1) -- (1+1/2,0); \node at (0.2,
      -0.5) {\footnotesize $1-\frac{1}{n}$};%
      \node at (1.8, -0.5) {\footnotesize $1+\frac{1}{n}$}; %
      \draw[dashed] (1,1) -- (0,1) node[left] {\footnotesize $1$};%
      \draw[dashed] (1,1) -- (1,0);
    \end{tikzpicture} \qquad\qquad
    \begin{tikzpicture}
      \draw[->] (0,0) -- (2.5,0) node[right] {$x$}; %
      \draw[->] (0,0) -- (0,1.5) node[above] {$r_{o,n}$}; %
      \draw[thick] (0,1)--(1/2,1) -- (1,0); %
      \node at (0.2, -0.5) {\footnotesize $1-\frac{1}{n}$}; %
      \node at (1, -0.5) {\footnotesize $1$}; %
      \draw[dashed] (0,1) node[left] {\footnotesize $1$}; %
      \draw[dashed] (1/2,1) -- (1/2,0);
    \end{tikzpicture} \qquad\qquad
    \begin{tikzpicture}
      \draw[->] (0,0) -- (2.5,0) node[right] {$x$}; %
      \draw[->] (0,0) -- (0,1.5) node[above] {$\rho_{o,n} + r_{o,n}$}; %
      \draw[thick] (0, 1) -- (1,1) -- (1+1/2,0); %
      \node at (1.8,-0.5) {\footnotesize $1+\frac{1}{n}$}; %
      \draw[dashed] (1,1) -- (0,1) node[left] {\footnotesize $1$};%
      \draw[dashed] (1,1) -- (1,0);
      \node at (1, -0.5) {\footnotesize $1$}; %
    \end{tikzpicture} \qquad\qquad
    \caption{The initial condition $(\rho_{o,n}, r_{o,n})$ and their sum in
      \Cref{ex:1}.}
    \label{fig:init-cond}
  \end{figure}

  \noindent Note that $\rho_{o,n} (x) + r_{o,n} (x) =1$ for all
  $x \in [0,1]$. In this setting,
  \begin{displaymath}
    (\rho_{o,n}, r_{o,n}) \to (0, \caratt{[0,1]}) \mbox{ in } \L1 (\reali;\reali^2)
    \quad \mbox{ but for $t>0$ } \quad
    \mathcal{S}_t (\rho_{o,n}, r_{o,n}) \not \to \mathcal{S}_t (0, \caratt{[0,1]})
    \mbox{ in } \L1 (\reali;\reali^2) \,.
  \end{displaymath}
  Indeed, $(\rho_{o,n}, r_{o,n})$ yields the stationary solution
  to~\eqref{eq:1}, since $v_{NL} \equiv 0$ and
  $v_L\left(\rho_{o,n} (x) + r_{o,n} (x)\right) = v_L(1) = 0$ for all
  $x \in [0,1]$. Hence, for all $t > 0$,
  $\mathcal{S}_t(\rho_{o,n}, r_{o,n}) = (\rho_{o,n}, r_{o,n})$.

  On the other hand, calling
  $\left(\rho (t), r (t)\right) = \mathcal{S}_t(0, \caratt{[0,1]})$,
  we have that $\rho (t) \equiv 0$ and $r$ is the (non stationary)
  entropy solution to the Cauchy problem
  \begin{equation*}
    \left\{
      \begin{array}{ll}
        \partial_t r + \partial_x \left(r \; v_L (r)\right) = 0
        \\
        r(0, x) = \caratt{[0,1]} (x) \,.
      \end{array}
    \right.
  \end{equation*}
  Hence, the mere $\L1$-convergence of the initial data does not
  guarantee the convergence in $\L1(\reali; \reali^2)$ of the solution
  at any positive time.

  Note, for completeness, that $\rho_{o,n} \in \cBV{1} (\reali; \reali)$
  and $r_{o,n} \in \cBV{} (\reali;\reali)$ so that
  Theorem~\ref{thm:main} applies and~\Cref{item:24} holds.
\end{example}

We now prove the uniqueness of solutions to~\eqref{eq:1} in the same
class where existence is ensured by the semigroup trajectories
exhibited in Theorem~\ref{thm:main}.

\begin{theorem}
  \label{thm:unique}
  Let~\ref{item:17}, \ref{item:21} and~\ref{item:hyp-eta} hold. Fix
  $(\rho_o, r_o) \in \cBV1 (\reali; \reali^k) \times \cBV{} (\reali;
  \reali)$ and assume that
  $(\rho_1,r_1),\,(\rho_2, r_2) \in \C0\left([0,T]; \cBV1 (\reali;
    \reali^k) \times \cBV{} (\reali; \reali)\right)$
  solve~\eqref{eq:1} with initial datum $(\rho_o, r_o)$ in the sense
  of Definition~\ref{def:sol}. Then,
  \begin{displaymath}
    (\rho_1,r_1) = (\rho_2, r_2)\,.
  \end{displaymath}
\end{theorem}

\noindent The proof is deferred to~\S~\ref{subsec:proof-theor-refthm}.

As a consequence, the trajectories of the semigroup $\mathcal{S}$
constructed in Theorem~\ref{thm:main} are uniquely characterized as
the solutions to~\eqref{eq:1} in the sense of
Definition~\ref{def:sol}.

\section{Analytical Proofs}
\label{sec:analytical-proofs}

Throughout, we use the following notation:
\begin{displaymath}
  \begin{array}{rcl@{\qquad}rcl}
    \norma{\rho}_{\L1 (\reali; \reali^k)}
    & \coloneq
    & \sum_{i=1}^k \norma{\rho^i}_{\L1 (\reali; \reali)}
    & \norma{\rho}_{\W11 (\reali; \reali^k)}
    & \coloneq
    & \norma{\rho}_{\L1 (\reali; \reali^k)}
      + \norma{\rho'}_{\L1 (\reali; \reali^k)}
    \\
    \norma{\rho}_{\L\infty (\reali; \reali^k)}
    & \coloneq
    & \sum_{i=1}^k \norma{\rho^i}_{\L\infty (\reali; \reali)}
    & \norma{\rho}_{\W1\infty (\reali; \reali^k)}
    & \coloneq
    & \norma{\rho}_{\L\infty (\reali; \reali^k)}
      + \norma{\rho'}_{\L\infty (\reali; \reali^k)}
    \\
    \tv (\rho)
    & \coloneq
    & \sum_{i=1}^k \tv (\rho^i)
  \end{array}
\end{displaymath}
and the vector convolution in the space variable is, for $i=1, \ldots, k$,
\begin{equation}
  \label{eq:28}
  \left(\rho (t) * \eta^i\right) (x)
  = \left(
    \left(\rho^1 (t) * \eta^i\right) (x), \ldots,
    \left(\rho^k (t) * \eta^i\right) (x)
  \right) .
\end{equation}

The next two lemmas refer to general elementary functional
analytic results of use below.

\begin{lemma}
  \label{lem:lei}
  With the notation~\eqref{eq:75}, if
  $q \in \cBV{} (\reali; \reali^k)$, then
  $\norma{q}_{\L\infty (\reali, \reali^k)} \leq \frac12 \, \tv (q)$.
\end{lemma}

\begin{proofof}{Lemma~\ref{lem:lei}}
  We use below the left continuous representative of $q$. Since,
  $q \in \L1 (\reali;\reali^k)$, then for every $\epsilon>0$, there
  exist $a_\epsilon$ and $b_\epsilon$ with $a_\epsilon <0$,
  $b_\epsilon>0$, $\modulo{q (a_\epsilon)} < \epsilon$ and
  $\modulo{q (b_\epsilon)} < \epsilon$. For any $x \in \reali$ we thus
  have
  \begin{displaymath}
    \begin{array}{rclcl}
      \modulo{q (x)}
      & \leq
      & \modulo{q (a_\epsilon)}
        + \tv\left(q; \mathopen]-\infty, x\mathclose[\right)
      & \leq
      & \epsilon + \tv\left(q; \mathopen]-\infty, x\mathclose[\right) \,;
      \\
      \modulo{q (x)}
      & \leq
      & \modulo{q (b_\epsilon)}
        + \tv\left(q; [x, +\infty\mathclose[\right)
      & \leq
      & \epsilon + \tv\left(q; [x, +\infty\mathclose[\right) \,.
    \end{array}
  \end{displaymath}
  Summing the two latter results and passing to the infimum over
  $\epsilon$ completes the proof.
\end{proofof}

\begin{lemma}[Characterization of $\cBV{1} (\reali;\reali^k)$]
  \label{lem:lui}
  With the notation~\eqref{eq:75}--\eqref{eq:64},
  $\cBV{1} (\reali;\reali^k)$ is the set of\/ $\W11$ limits of\/
  $\W21$ bounded sequences in
  $(\C2 \cap \W2\infty \cap \W21) (\reali; \reali^k)$. Moreover, if
  $\rho \in \cBV{1} (\reali;\reali^k)$, then there exists
  $\rho_n \in (\C2 \cap \W2\infty \cap \W21) (\reali;\reali^k)$ such
  that $\lim_{n\to+\infty}\rho_n = \rho $ in $\W11 (\reali;\reali^k)$
  and
  \begin{equation}
    \label{eq:58}
    \lim_{n \to +\infty} \norma{\rho_n}_{\W21(\reali; \reali^k)}
    =
    \norma{\rho}_{\cBV{1}(\reali; \reali^k)} \,.
  \end{equation}
\end{lemma}

\begin{proofof}{Lemma~\ref{lem:lui}}
  Fix $\rho $ such that there exists a sequence $\rho_n$ in
  $(\C2 \cap \W2\infty \cap \W21) (\reali; \reali^k)$ bounded in
  $\W21 (\reali;\reali^k)$ and converging to $\rho$ in
  $\W11 (\reali;\reali^k)$. Hence
  $\rho' = \lim_{n\to +\infty} \rho'_n$ in $\L1 (\reali;\reali^k)$,
  with $\rho'_n$ bounded in $\W11
  (\reali;\reali^k)$. By~\cite[Theorem~3.9]{MR1857292}, we have
  $\rho \in \cBV{1} (\reali;\reali^k)$.

  On the other hand, introduce a mollifier
  $\zeta \in \Cc\infty (\reali; \reali_+)$ with
  $\spt \zeta {\subseteq} [-1,1]$ and $\int_{\reali} \zeta = 1$. For
  $n \in \naturali \setminus \{0\}$, define
  $\zeta_n (x) = n \, \zeta (n\,x)$ and $\rho_n = \rho *
  \zeta_n$. Clearly,
  $\rho_n \in (\C2 \cap \W2\infty \cap \W21) (\reali;\reali^k)$, and
  \begin{flalign*}
    \norma{\rho_n}_{\L1 (\reali;\reali)}
    \leq
    & \norma{\rho}_{\L1 (\reali;\reali)} \; \norma{\zeta}_{\L1 (\reali;\reali)}
    & [\mbox{By~\cite[Theorem~4.15]{zbMATH05633610}}]
    \\
    \norma{\rho'_n}_{\L1 (\reali;\reali)}
    \leq
    & \norma{\rho'}_{\L1 (\reali;\reali)} \; \norma{\zeta}_{\L1 (\reali;\reali)}
    & [\mbox{By~\cite[Theorem~4.15]{zbMATH05633610}}]
    \\
    \norma{\rho_n''}_{\L1 (\reali;\reali)}
    \leq
    & \tv(\rho') \; \norma{\zeta}_{\L1 (\reali;\reali)} \,,
    & [\mbox{By~\cite[Proposition~8.49]{MR1681462}}
  \end{flalign*}
  hence $\rho$ is the $\W11$ limit of a $\W21$ bounded sequence in
  $(\C2 \cap \W2\infty \cap \W21) (\reali; \reali^k)$. By the lower
  semicontinuity in $\L1$ of the total variation,
  $\tv(\rho') \leq \liminf_{n \rightarrow \infty} \tv(\rho'_n) =
  \liminf_{n \rightarrow \infty} \norma{\rho_n''}_{\L1(\reali;
    \reali)}$, proving~\eqref{eq:58}.
\end{proofof}

\subsection{Stability Property of Solutions}
\begin{proofof}{\Cref{lem:GoodSol}}
  Introduce for $i=1,\ldots, k$ the maps
  \begin{eqnarray*}
    v^i_n (t,x)
    & =
    & v_{NL}^i\left(
      \left(\Sigma \rho_n (t) * \eta^i\right) (x)
      +
      \left(\left(r_n (t) + e_n (t)\right)  * \eta^i\right) (x)
      \right)
    \\
    v^i_\infty (t,x)
    & =
    & v_{NL}^i\left(
      \left(\Sigma \rho_\infty (t) * \eta^i\right) (x)
      +
      \left(r_\infty (t)   * \eta^i\right) (x)
      \right)
  \end{eqnarray*}
  and observe that by~\ref{item:17} and~\eqref{eq:70}, we have the
  convergence
  $\lim_{n\to+\infty}v_n^i = v^i_\infty$ a.e. in
  $[0,T] \times \reali$. Moreover, since
  \begin{equation*}
    \norma{\left(\Sigma \rho_n * \eta^i\right)
      +
      \left(r_n  + e_n \right)  * \eta^i}_{\L\infty ([0,T]\times\reali;\reali)}
    \leq
    \norma{\Sigma\rho_n+ r_n+ e_n}_{\C0([0,T];\L1 (\reali;\reali))}
    \, \norma{\eta^i}_{\L\infty (\reali;\reali)}
  \end{equation*}
  is uniformly bounded by~\eqref{eq:70}, also
  $\sup_n \norma{v_n}_{\L\infty ([0,T]\times\reali;\reali^k)} <
  +\infty$. By~\cite[Theorem 4.9]{zbMATH05633610} and~\eqref{eq:70},
  there exists a function $h \in \L1([0,T] \times \reali; \reali)$
  such that for a.e.~$(t,x) \in[0,T] \times \reali$, up to a
  subsequence,
  \begin{equation}
    \label{eq:74}
    \sup_{n \in \naturali} \norma{\rho_n(t,x)} +
    \sup_{n \in \naturali} \norma{\partial_x\rho_n(t,x)} +
    \sup_{n \in \naturali} \norma{\partial_x\epsilon_n(t,x)} +
    \sup_{n \in \naturali} \modulo{r_n(t,x)} \leq h(t,x).
  \end{equation}
  This allows to apply the Dominated Convergence Theorem to the
  equality \begin{displaymath} \int_0^T \int_\reali
    \left(\rho_n^i(t,x) \,\partial_t \phi(t,x) + \rho_n^i(t,x) \,
      v_n^i(t,x) \, \partial_x \phi(t,x) \right) \d{x} \d{t} +
    \int_\reali (\rho_{o,n})^i(x) \, \phi(0,x) \d{x} = 0
  \end{displaymath}
  which holds by Definition~\ref{def:sol}, for all $i = 1, \ldots, k$,
  $n \in \naturali\setminus \{0\}$,
  $\phi \in \Cc1(\mathopen]-\infty, T\mathclose[\times \reali;
  \reali)$. Hence, $\rho_\infty$ solves the nonlocal equation
  in~\eqref{eq:71}.

  Concerning the local part of~\eqref{eq:71}, we proceed
  similarly. Introduce
  \begin{equation*}
    w_n (t,x,r)
    \coloneqq
    v_L\left(\Sigma\left(\rho_n (t,x)+\epsilon_n (t,x)\right)
      + r\right)
    \quad \mbox{ and } \quad
    w_\infty (t,x,r)
    \coloneqq
    v_L\left(\Sigma\rho_\infty (t,x) + r\right) \,.
  \end{equation*}
  Clearly, since $v_L$ satisfies~\ref{item:21} and~\eqref{eq:70},
  $\lim_{n\to+\infty} w_n\left(t,x,r_n (t,x)\right) =
  w_\infty\left(t,x,r_\infty (t,x)\right)$ for
  a.e.~$(t,x) \in [0,T]\times\reali$.  We aim at applying again the
  Dominated Convergence Theorem to the integral inequality in the
  definition of {Kru\v zkov} solution, i.e.
  \begin{align}
    \label{eq:44}
    & \int_0^T \int_\reali \modulo{r_n(t,x) - k} \partial_t \phi(t,x)
      \d{x}\d{t}
    \\
    \label{eq:72}
    +
    & \int_0^T \int_\reali \sgn(r_n(t,x) - k )
      \left[
      r_n(t,x) \, w_n\left(t,x,r_n (t,x)\right) -
      k \, w_n(t,x,k)
      \right] \partial_x \phi(t,x) \d{x}\d{t}
    \\
    \label{eq:73}
    -
    &  \int_0^T \int_\reali \sgn(r_n(t,x)-k) \,
      k \, \partial_x w_n(t,x,k) \, \phi(t,x) \d{x}\d{t} \geq 0
  \end{align}
  for $k \in \reali$, $n \in \naturali$ and
  $\phi \in \Cc\infty([0,T] \times \reali; \reali_+)$. To this aim,
  observe that the term in~\eqref{eq:44} is easily
  dominated. Concerning the integrands in~\eqref{eq:72}
  and~\eqref{eq:73}, compute them for a.e.~$(t,x) \in [0,T] \times \reali$ and use $h$ as in~\eqref{eq:74} to
  get that
  \begin{eqnarray*}
    \modulo{\eqref{eq:72}}
    & \leq
    & \left(\modulo{r_n(t,x)} + \modulo{k}\right)
      \norma{v_L}_{\L\infty(\reali; \reali)} \, \modulo{\partial_x \phi(t,x)}
    \\
    & \leq
    & \left(h(t,x) + \modulo{k}\right)
      \norma{v_L}_{\L\infty(\reali; \reali)} \, \modulo{\partial_x \phi(t,x)} \,,
    \\{}
    \modulo{\mbox{\eqref{eq:73}}}
      & \leq
      & \modulo{k} \,
        \norma{v_L'}_{\L\infty(\reali; \reali)} \,
        \modulo{\Sigma
        \left(\partial_x \rho_n(t,x) + \partial_x \epsilon_n(t,x)\right)} \,
        \modulo{\partial_x \phi(t,x)}
      \\
      & \leq
      & \modulo{k} \,
        \norma{v_L'}_{\L\infty(\reali; \reali)} \,
        h(t,x) \, \modulo{\partial_x \phi(t,x)}.
    \end{eqnarray*}
    Then, as functions of $(t,x)$, the dominating quantities obtained
    above are in $\L1([0,T] \times \reali; \reali)$ and the Dominated
    Convergence Theorem allows to
    pass~\eqref{eq:44}--\eqref{eq:72}--\eqref{eq:73} to the limit
    $n \to +\infty$, proving that $r_\infty$ solves~\eqref{eq:71}.
  \end{proofof}

\subsection{The Nonlocal Problem}
\label{subsec:non-local-problem}

We are interested in the system
\begin{equation}
  \label{eq:7}
  \left\{
    \begin{array}{l}
      \partial_t \rho^i +
      \partial_x \left(\rho^i \;
      V^i \left(t,x, \rho *\eta^i\right)\right) = 0
      \qquad i=1, \ldots, k
      \\
      \rho (0,x) = \rho_o (x)
    \end{array}
  \right.
\end{equation}
where $V = V (t,x,q)$ is such that $V \in \C0 ([0,T];\mathcal{V})$ and
$\mathcal{V}$ is the space of speed laws
\begin{equation}
  \label{eq:65}
  \begin{array}{@{}rcl@{}}
    \mathcal{V}
    & \coloneqq
    & \left\{
      V \in \C3(\reali \times \reali^k;\reali^k)
      \colon
      \begin{array}{l@{}}
        V(\cdot, 0) \in \L\infty(\reali;\reali^k),
        \\
        {}[\partial_x V \quad  \nabla_q V ] \in
        \W2\infty(\reali \times \reali^k; \reali^k \times \reali^{k\times k})
      \end{array}
      \right\}
    \vspace{.2cm}\\
    \norma{V}_{\mathcal{V}}
    & \coloneqq
    & \norma{V(\cdot, 0)}_{\L\infty (\reali;\reali^k)} + \norma{[\partial_x V
      \quad \nabla_q V ]}_{\W2\infty(\reali\times \reali^k;\reali^k\times \reali^{k\times k})}.
  \end{array}
\end{equation}

\begin{definition}
  \label{def:solution}
  By \emph{solution to~\eqref{eq:7} on the time interval $I$} we mean $\rho \in \C0 \left(I; \L1 (\reali; \reali^k)\right)$ that
  solves
  \begin{equation}
    \label{eq:20}
    \!\!\!\!\!
    \left\{
      \begin{array}{@{}l@{\quad}l}
        \partial_t \rho^i + \partial_x \left(\rho^i \; v^i (t,x)\right) = 0
        & (t,x) \in I {\times} \reali
        \\
        \rho^i (0,x) = \rho_o^i (x)
        & x \in \reali
      \end{array}
    \right.
    \mbox{ where }
    v^i (t,x)
    \coloneqq
    V^i \left(t, x, \left(\rho (t) * \eta^i\right) (x) \right) .
  \end{equation}
\end{definition}

\noindent Above, a \emph{solution to~\eqref{eq:20}} is a
distributional solution~\cite[Definition~4.2]{MR1816648}, which is
also a weak entropy solution in the sense
of~\cite[Definition~1]{MR267257}, see
also~\cite[Corollary~II.1]{MR1022305}.

For $i=1, \ldots, k$, given
$v^i \in \C0\left([0,T]; \W11(\reali; \reali)\right)$, the solution to
\begin{displaymath}
  \left\{
    \begin{array}{l}
      \partial_t \rho^i + \partial_x \left(\rho^i \; v^i (t,x)\right) = 0
      \\
      \rho^i (0,x) = \rho_o^i (x)
    \end{array}
  \right.
\end{displaymath}
in the sense of Definition~\ref{def:solution} is
\begin{eqnarray}
  \label{eq:13}
  \!\!\!\!\!\!\!\!\!
  \rho^i(t,x)
  & =
  & \rho_o^i\left(X^i(0;t,x)\right) \; \mathcal{E}^i(0,t,x)
    \qquad\qquad (t,x) \in [0,T] \times \reali, \qquad \mbox{ where}
  \\
  \label{eq:12}
  \!\!\!\!\!\!\!\!\!
  \mathcal{E}^i(t,t_o,x_o)
  & =
  & \exp \left(
    \int_{t_o}^{t} \partial_x v^i \left(s, X^i(s;t_o,x_o)\right) \d{s}
    \right)
    \qquad (t,t_o,x_o) \in [0,T]^2 \times \reali
\end{eqnarray}
and $s \mapsto X^i(s;t,x)$ is the solution to the Cauchy problem
\begin{equation}
  \label{eq:11}
  \begin{cases}
    \frac{\d{~}}{\d{s}} X^i(s;t,x) = v^i\left(s, X^i(s;t,x)\right)
    \\
    X^i(t;t,x) = x \,.
  \end{cases}
\end{equation}
The change of variable
\begin{equation}
  \label{eq:59}
  \xi = X^i(0;t_o,x_o)
\end{equation}
and the following {formul\ae}, from~\cite[Theorem~2.3.2,
Theorem~2.3.3]{zbMATH05197323}, are frequently used below:
\begin{equation}
  \label{eq:17}
  \partial_{x_o} X^i(t;t_o,x_o)
  = \mathcal{E}^i(t,t_o,x_o)
  \quad \mbox{ and } \quad
  \partial_{t_o} X^i(t;t_o,x_o)
  = - \mathcal{E}^i(t,t_o,x_o) \, v^i(t_o,x_o) \,.
\end{equation}

For completeness, we recall below --- without proof --- properties
concerning the problem~\eqref{eq:7} useful in the sequel.
\begin{lemma}[{\cite[Theorem~2.3]{OperaPrima}}]
  \label{lem:4}
  Assume $\eta \in \W2\infty(\reali;\reali^k)$,
  $V \in \C0\left([0,T];\mathcal{V}\right)$ and
  $\rho_o \in \cBV{}(\reali;\reali^k)$ Then,
  \begin{enumerate}[label=\bf\arabic*.]
  \item \label{item:4} There exists a unique solution $\rho$
    to~\eqref{eq:7} on $[0,T]$ in the sense of
    Definition~\ref{def:solution}.

  \item \label{item:44} If
    $\widehat{\rho}_o \in \cBV{}(\reali;\reali^k)$ and
    $\widehat{\rho}$ is the corresponding solution to~\eqref{eq:7} on
    $[0,T]$, then for all $t \in [0,T]$
    \begin{align*}
      &\quad\norma{\rho(t) - \widehat{\rho}(t)}_{\L1(\reali;\reali^k)}
      \\
      &\!\leq \!
        C\!\left(\! \norma{\eta}_{\W2\infty(\reali;\reali^k)},
        \norma{V}_{\C0([0,t];\mathcal{V})}, \norma{\rho_o}_{\L1(\reali;\reali^k)},
        \norma{\widehat{\rho}_o}_{\cBV{}(\reali;\reali^k)},
        t \right)\, \norma{\rho_o - \widehat{\rho}_o}_{\L1(\reali;\reali^k)}.
    \end{align*}
  \item \label{item:25} $\rho$ is locally Lipschitz continuous in $t$:
    for all $t,\widehat{t} \in [0,T]$
    \begin{equation*}
      \norma{\rho (t) - \rho (\widehat{t})}_{\L1 (\reali; \reali^k)} \!
      \leq
      C\!\left(\norma{\eta}_{\W2\infty (\reali;\reali^k)},
        \norma{V}_{\C0([0,T];\mathcal{V})},
        \norma{\rho_o}_{\cBV{} (\reali;\reali^k)}, T\right)
      \modulo{t {-} \widehat{t}} \,.
    \end{equation*}

  \item \label{item:45} If
    $\widehat{V} \in \C0\left([0,T];\mathcal{V}\right)$ and
    $\widehat{\rho}$ is the corresponding solution to~\eqref{eq:7} on
    $[0,T]$, then for all $t \in [0,T]$:
    \begin{eqnarray*}
      &
      & \norma{\rho(t) - \widehat{\rho}(t)}_{\L1(\reali;\reali^k)}
      \\
      & \leq
      & \, \mathcal C\left(\norma{\eta}_{\W2\infty(\reali;\reali^k)},
        \norma{V}_{\C0([0,t];\mathcal{V})},
        \norma{\rho_o}_{\cBV{}(\reali;\reali^k)},t \right)
        \norma{V - \widehat{V}}_{\C0([0,t];\mathcal{V})} \,\,t .
    \end{eqnarray*}

  \item \label{item:46} For all $t \in [0,T]$ and for
    $i = 1, \ldots,k$,
    $\norma{\rho^i(t)}_{\L1(\reali;\reali)} =
    \norma{\rho_o^i}_{\L1(\reali;\reali)}$.

  \item \label{item:22} For $i=1, \ldots, k$, if $\rho_o^i \geq 0$,
    then $\rho^i (t,x) \geq 0$ for
    a.e.~$(t,x) \in [0,T] \times \reali$.
  \end{enumerate}
\end{lemma}

\noindent We now establish
some qualitative and regularity properties of the solution to the
nonlocal problem~\eqref{eq:7} under more restrictive hypotheses on the
initial datum $\rho_o$, on the speed law $V$ and on the kernel function
$\eta$.

\begin{lemma}[Fine estimates on $\rho$]
  \label{lem:5}
  Suppose $\eta \in (\C3 \cap \W3\infty)(\reali;\reali^k)$,
  $V \in \C0\left([0,T];\mathcal{V}\right)$ and
  $\rho_o \in (\C2 \cap \W2\infty \cap \W21)(\reali;\reali^k)$. Then,
  the solution $\rho$ to the nonlocal problem~\eqref{eq:7}, constructed in
  Lemma~\ref{lem:4}, satisfies the following properties.
  \begin{enumerate}[label=\bf (N\arabic*)]
  \item \label{item:32}
    $\rho \in \C1\left([0,T] \times \reali ; \reali^k\right)$.

  \item\label{item:39}
    $\partial_x \rho \in \L1
    \left([0,T];\W11(\reali;\reali^k)\right)$.  In particular, for a.e.~$t \in [0,T]$
    \begin{align}
      \label{eq:45}
      \!\!\!
      \norma{\partial_x \rho(t)}_{\L1(\reali; \reali^k)}
      & \leq
       \mathcal{C} \left(
        \norma{\eta}_{\W3\infty(\reali; \reali^k)},
        \norma{V}_{\C0([0,T];\mathcal{V})},
        \norma{\rho_o}_{\W11(\reali; \reali^k)},
        t
        \right) ,
      \\
      \label{eq:46}
      \!\!\!
      \norma{\partial^2_{xx} \rho(t)}_{\L1(\reali; \reali^k)}
      & \leq
       \mathcal{C} \left(
        \norma{\eta}_{\W3\infty(\reali; \reali^k)},
        \norma{V}_{\C0([0,T];\mathcal{V})},
        \norma{\rho_o}_{\W21(\reali; \reali^k)},
        t
        \right) ,
      \\
      \label{eq:54}
      \!\!\!\!\!\!\!\!\!
      \norma{\partial_x \rho}_{\L1([0,t]\times \reali; \reali^k)}
      & \leq
      \mathcal{C} \left(
        \norma{\eta}_{\W3\infty(\reali; \reali^k)},
        \norma{V}_{\C0([0,T];\mathcal{V})},
        \norma{\rho_o}_{\W11(\reali; \reali^k)},
        T
        \right) t,
      \\
      \label{eq:55}
      \!\!\!\!\!\!\!\!\!\!\!\!
      \norma{\partial^2_{xx} \rho}_{\L1([0,t] \times \reali; \reali^k)}
      & \leq
      \mathcal{C} \left(
        \norma{\eta}_{\W3\infty(\reali; \reali^k)},
        \norma{V}_{\C0([0,T];\mathcal{V})},
        \norma{\rho_o}_{\W21(\reali; \reali^k)},
        T
        \right) t.
    \end{align}

  \item \label{item:26} $\rho$ is $\W11$ locally Lipschitz continuous
    in $t$, i.e., for all $t, \widehat{t} \in [0,T]$
    \begin{displaymath}
      \norma{\rho (t) - \rho (\widehat{t})}_{\W11 (\reali; \reali^k)}
      \leq
      \mathcal{C}
      \left(
        \norma{\eta}_{\W3\infty (\reali; \reali^k)},
        \norma{V}_{\C0([0,T];\mathcal{V})},
        \norma{\rho_o}_{\W21 (\reali;\reali^k)},
        T\right)
      \modulo{t-\widehat{t}} \,.
    \end{displaymath}

  \item\label{item:47} $\rho \in \L\infty\left([0,T] \times \reali ; \reali^k\right)$ with
    \begin{equation}
      \label{eq:40}
      \!\!\!\!\!\!\!\!\!\!\!\!\!\!\!\!\!\!
      \norma{\rho}_{\L\infty([0,T]\times \reali;\reali^k)}
      \leq
      \mathcal{C}\left(
        \norma{\eta}_{\W3\infty (\reali; \reali^k)},
        \norma{V}_{\C0([0,T];\mathcal{V})},
        \norma{\rho_o}_{\L1(\reali;\reali^k)},
        \norma{\rho_o}_{\L\infty(\reali;\reali^k)}, T
      \right) .
    \end{equation}
  \item \label{item:36}
    $\partial_{xx}^2 \rho \in \C0\left([0,T] \times
      \reali;\reali^k\right)$.
  \item \label{item:38}
    $\partial_x \rho, \partial_t \rho, \partial^2_{xx}\rho, \partial_{tx}^2 \rho \in
    \L\infty\left([0,T] \times \reali;\reali^k\right)$.
    In
    particular,
    \begin{equation}
      \label{eq:48}
      \!\!\!\!\!\!\!\!\!\!\!\!\!\!\!\!\!\!\!\!\!
      \norma{\partial_x  \rho}_{\L\infty([0,T] \times \reali; \reali^k)} \!
      \leq
      \!\mathcal{C} \!\left(
        \norma{\eta}_{\W3\infty(\reali; \reali^k)},
        \norma{V}_{\C0([0,T];\mathcal{V})},
        \norma{\rho_o}_{\L1(\reali; \reali^k)},
        \norma{\rho_o}_{\W1\infty(\reali; \reali^k)}, T \right)\!.
    \end{equation}
  \end{enumerate}
\end{lemma}

\begin{proof}
  Fix $i \in \{1, \ldots, k\}$. Referring to~\eqref{eq:28}, introduce
  $v^i(t,x) \coloneq V^i\left(t,x,(\rho(t) * \eta^i)(x)\right)$ and
  use the notation~\eqref{eq:13}. We split the proof in different
  parts.

  \paragraph{Claim~1. The maps
    $(t,x) \to v(t,x), \partial_x v(t,x) \in \C0\left([0,T] \times \reali;
    \reali^k\right)$. Moreover, for all $t \in [0,T]$:
    $x \mapsto v(t,x) \in (\C3 \cap \W3\infty) (\reali;\reali^k)$ and
    $v \in \L\infty \left([0,T]; \W3\infty (\reali;\reali^k)\right)$.}

  The continuity of the maps $(t,x) \to v(t,x), \partial_x v(t,x)$ is
  a consequence of the assumptions on $V, \eta$ and the hypothesis
  $\rho \in \C0([0,T]; \L1(\reali; \reali^k)$. Similarly, for all
  $t \in [0,T]$, $x \mapsto v (t,x)$ is of class $\C3$ and allows the
  following estimates:
  \begin{align*}
    \modulo{v^i (t,x)}
    \ghost{
    & \leq
      \modulo{V^i (t,x,0)} +
      \norma{\nabla_q V^i}_{\L\infty ([0,t]\times\reali\times\reali^k;\reali)} \,
      \norma{\eta^i}_{\L\infty (\reali;\reali)} \,
      \norma{\rho (t)}_{\L1 (\reali;\reali^k)}
    \\
    }
    & \leq
      \norma{V}_{\C0([0,t];\mathcal{V})}
      \left(
      1 +
      \norma{\eta^i}_{\L\infty (\reali;\reali)} \,
      \norma{\rho_o}_{\L1 (\reali;\reali^k)}
      \right)
    \\
    \modulo{\partial_x v^i(t,x)}
    \ghost{& =
             \modulo{\partial_x V^i(t,x, \left(\rho(t)*\eta^i)(x)\right) +
             \nabla_q V^i\left(t,x, (\rho(t)*\eta^i)(x)\right)
             \left( \left(\rho(t)*(\eta^i)' \right)(x) \right)
             }
    \\}
    & \leq
      \norma{V}_{\C0([0,t];\mathcal{V})}
      \left(
      1 +
      \norma{(\eta^i)'}_{\L\infty (\reali;\reali)} \,
      \norma{\rho_o}_{\L1 (\reali;\reali^k)}
      \right)
    \\
    \modulo{\partial_{xx}^2 v^i(t,x)}
    \ghost{& =
             \big|
             \partial_{xx}^2 V^i\left(t,x,(\rho(t)*\eta^i)(x)\right) +
             2\nabla_q\partial_x V^i\left(t,x,(\rho(t)*\eta^i)(x)\right)
             \left( \left(\rho(t)*(\eta^i)' \right)(x) \right)
    \\
    & \quad
      + \partial_{\rho\rho}^2 V^i\left(t,x,(\rho(t)*\eta^i)(x)\right)
      \left(\left(\rho(t)*(\eta^i)' \right)(x)\right)^2
    \\
    & \quad
      + \nabla_q V^i\left(t,x,(\rho(t)*\eta^i)(x)\right)
      \left(\left( \rho(t)*(\eta^i)'' \right)(x)\right) \big|
    \\}
    & \leq
      \norma{V}_{\C0([0,t];\mathcal{V})} \!
      \left(
      1 {+}
      2 \norma{(\eta^i)'}_{\W1\infty (\reali;\reali)}
      \norma{\rho_o}_{\L1 (\reali;\reali^k)}
      {+}
      \norma{(\eta^i)'}_{\L\infty (\reali;\reali)}^2
      \norma{\rho_o}_{\L1 (\reali;\reali^k)}^2
      \right)
    \\
    \modulo{\partial_{xxx}^3 v^i(t,x)}
    \ghost{
    & =
      \big| \partial_{xxx}^3 V^i\left( t,x,\left(\rho(t)*\eta^i \right)(x)\right)
    \\
    & \quad
      + 3
      \partial_{xx\rho}^3 V^i\left( t,x,\left(\rho(t)*\eta^i \right)(x)\right)
      \left( \left( \rho(t)*(\eta^i)' \right)(x)\right)
    \\
    &\quad
      + 3\partial_{x\rho\rho}^3 V^i\left( t,x,\left(\rho(t)*\eta^i \right)(x)\right)
      \left( \left( \rho(t)*(\eta^i)' \right)(x)\right)^2
    \\
    &\quad
      + \partial_{x\rho}^2 V^i\left( t,x,\left(\rho(t)*\eta^i \right)(x)\right)
      \left( \left( \rho(t)*(\eta^i)'' \right)(x)\right)
    \\
    &\quad
      + \partial_{\rho\rho\rho}^3 V^i\left( t,x,\left(\rho(t)*\eta^i \right)(x)\right)
      \left( \left( \rho(t)*(\eta^i)' \right)(x)\right)^3
    \\
    &\quad
      + 3\partial_{\rho\rho}^2 V^i\left( t,x,\left(\rho(t)*\eta^i \right)(x)\right)
      \left( \left( \rho(t)*(\eta^i)' \right)(x)\right)
      \left( \left( \rho(t)*(\eta^i)'' \right)(x)\right)
    \\
    &\quad
      + \partial_{x\rho}^2
      V^i\left( t,x,\left(\rho(t)*\eta^i \right)(x)\right)
      \left( \left( \rho(t)*(\eta^i)'' \right)(x)\right)
    \\
    &\quad
      + \partial_{\rho} V^i\left( t,x,\left(\rho(t)*\eta^i \right)(x)\right)
      \left( \left( \rho(t)*(\eta^i)''' \right)(x)\right) \big|
    \\}
    & \leq
      \norma{V}_{\C0([0,t];\mathcal{V})}
      \big(
      1
      +
      3\norma{\eta^i}_{\W3\infty(\reali;\reali)}\,
      \norma{\rho_o}_{\L1(\reali;\reali^k)}
    \\
    &
      \qquad
      +
      3 \norma{(\eta^i)'}^2_{\W1\infty(\reali;\reali)}\,
      \norma{\rho_o}^2_{\L1(\reali;\reali^k)}
      +
      \norma{(\eta^i)'}^3_{\L\infty(\reali;\reali)} \,
      \norma{\rho_o}^3_{\L1(\reali;\reali^k)}
      \big) \,.
  \end{align*}
  Finally, we also obtain
  \begin{equation}
    \label{eq:16}
    \begin{array}{rcl}
      &
      & \norma{v}_{\L\infty([0,T];\W3\infty(\reali;\reali))}
        \leq
        Q \qquad \mbox{ where}
      \\
      Q
      & \coloneqq
      &  4 \, \norma{V}_{\C0([0,T];\mathcal{V})}
        \big(
        1
        +
        \norma{\eta}_{\W3\infty(\reali;\reali^k)} \,
        \norma{\rho_o}_{\L1(\reali;\reali^k)}
      \\
      &
      & \qquad\qquad
        +
        \norma{\eta'}^2_{\W1\infty(\reali;\reali^k)} \,
        \norma{\rho_o}^2_{\L1(\reali;\reali^k)}
        +
        \norma{\eta'}^3_{\L\infty(\reali;\reali^k)} \,
        \norma{\rho_o}^3_{\L1(\reali;\reali^k)}
        \big)
    \end{array}
  \end{equation}
  completing the proof Claim~1.

  \paragraph{
  Claim~2: The following regularities hold:
    \begin{displaymath}
      \begin{array}{r@{\,}c@{\,}l@{\qquad\qquad}r@{\,}c@{\,}l@{\,}c@{\,}l}
        \forall \, t
        & \in
        & [0,T]
        & [(t_o,x_o)
        & \mapsto
        & X (t;t_o,x_o)]
        & \in
        & \C1 \left([0,T]\times\reali;\reali^k\right)
        \\
        \forall \, t
        & \in
        & [0,T]
        & [(t_o,x_o)
        & \mapsto
        & \mathcal{E} (t,t_o,x_o)]
        & \in
        & \C1 \left([0,T]\times\reali;\reali^k\right)
      \end{array}
    \end{displaymath}
    and moreover, for $i=1, \ldots, k$, $(t,t_o,x_o) \in [0,T]\times[0,T]\times\reali$,
            \begin{equation}
      \label{eq:18}
      \modulo{\mathcal{E}^i(t,t_o,x_o)}
 \leq e^{QT} \,.
\end{equation}}

Due to Claim~1, the continuity of $X$ is an immediate consequence of
$V \in \C0\left([0,T];\mathcal{V}\right)$ and
$\rho \in \C0 \left([0,T];\L1 (\reali; \reali^k)\right)$. The
continuity of $(t_o,x_o) \mapsto \mathcal{E} (t,t_o,x_o)$ follows
from~\eqref{eq:12} and the $\C1$ regularity of $X$ holds, due
to~\eqref{eq:17}. An application of~\eqref{eq:17} yields that
\begin{eqnarray}
  \nonumber
  \!\!\!\!\!\!
  \partial_{x_o} \mathcal{E}^i(t,t_o,x_o)
  & =
  & \partial_{x_o}
    \int_{t_o}^t \partial_x v^i\left(s, X^i(s;t_o,x_o)\right) \d{s}
  \\
  \ghost{
  \nonumber
  & =
  & \int_{t_o}^t \partial_{xx}^2 v^i\left(s, X^i(s;t_o,x_o)\right) \partial_{x_o} X^i(s;t_o,x_o) \d{s}
  \\}
  \nonumber
  & =
  &\int_{t_o}^t \partial_{xx}^2 v^i \left(s, X^i(s;t_o,x_o)\right) \,
    \mathcal{E}^i(s;t_o,x_o) \d{s}
  \\
  \nonumber
  \!\!\!\!\!\!
  \partial_{t_o} \mathcal{E}^i(t,t_o,x_o)
  & =
  & \partial_{t_o} \int_{t_o}^t \partial_x v^i\left(s, X^i(s;t_o,x_o)\right) \d{s}
  \\
  \ghost{
  \nonumber
  & =
  & -\partial_{x}v^i(t_o, x_o) + \int_{t_o}^t \partial_{xx}^2
    v^i\left(s, X^i(s;t_o,x_o)\right)\partial_{t_o} X^i(s;t_o,x_o) \d{s}
  \\}
  \label{eq:60}
  & =
  & {-}\partial_x v^i(t_o, x_o) {-} \!
    \int_{t_o}^t \! \partial_{xx}^2 v^i\left(s, X^i(s;t_o,x_o)\right)
    \mathcal{E}^i(s;t_o,x_o) \, v^i(t_o,x_o) \d{s}\qquad
\end{eqnarray}
proving the $\C1$ regularity of
$(t_o,x_o) \mapsto \mathcal{E} (t,t_o,x_o)$.

Moreover, referring to~\eqref{eq:16}, the bound~\eqref{eq:18} follows,
since for all $(t,t_o,x_o) \in [0,T] \times [0,T] \times\reali$,
$\modulo{\mathcal{E}^i(t,t_o,x_o)} \leq e^{Q \modulo{t-t_o}}$.

\paragraph{Claim~3: \ref{item:32} holds.}  One obtains that
$\rho \in \C0\left([0,T]\times \reali;\reali^k\right)$ because of the continuity
of $\rho_o$ and Claim~2.  Compute now
\begin{eqnarray}
  \label{eq:27}
  \partial_x \rho^i(t,x)
  \ghost{
  & =
  & (\rho_o^i)'\left(X^i(0;t,x)\right) \, \partial_x X^i(0;t,x)
    \, \mathcal{E}^i(0,t,x)
    + \rho_o^i\left(X^i(0;t,x)\right) \, \partial_x \mathcal{E}^i(0,t,x)
  \\}
  & =
  & (\rho_o^i)'\left(X^i(0;t,x)\right) \, \mathcal{E}^i(0,t,x)^2
  \\
  \nonumber
  &
  & \qquad + \rho_o^i\left(X^i(0;t,x)\right)
    \int_{t}^0 \partial_{xx}^2 v^i\left(s, X^i(s;t,x)\right) \,
    \mathcal{E}^i(s,t,x) \d{s}
  \\
  \nonumber
  \partial_t \rho^i(t,x)
  \ghost{
  & =
  & (\rho_o^i)'\left(X^i(0;t,x)\right) \, \partial_t X^i(0;t,x) \mathcal{E}^i(0,t,x)
    + \rho_o^i\left(X^i(0;t,x)\right) \, \partial_t \mathcal{E}^i(0,t,x)
  \\}
  \nonumber
  & =
  & (\rho_o^i)'\left(X^i(0;t,x)\right)
    \left( - \mathcal{E}^i(0,t,x) \, v^i(t,x) \right) \mathcal{E}^i(0,t,x)
  \\
  \nonumber
  &
  & + \rho_o^i\left(X^i(0;t,x)\right)
    \left( -\partial_x v^i(t, x)
    - \int_t^0 \partial_{xx}^2 v^i\left(s, X^i(s;t,x)\right)
    \mathcal{E}^i(s; t, x) \, v^i(t,x) \d{s} \right) ,
\end{eqnarray}
proving that $\rho \in \C1\left([0,T] \times \reali;\reali^k\right)$ due to
$\rho_o \in \C1(\reali;\reali^k)$, Claim~1 and Claim~2.

\paragraph{Claim~4: \ref{item:39} holds.} \!\!\!\!\!\! Use the change
of variable~\eqref{eq:59} and~\eqref{eq:17} to obtain for
a.e.~$t \in [0,T]$
\begin{eqnarray*}
  \norma{\partial_x\rho^i(t)}_{\L1 (\reali;\reali)}
  & \leq
  &  \int_\reali \modulo{(\rho_o^i)'\left(X^i(0;t,x)\right) \,
    \mathcal{E}^i(0,t,x)^2} \d{x}
  \\
  &
  &  +  \int_\reali \modulo{ \rho_o^i\left(X^i(0;t,x)\right)
    \int_0^t \partial_{xx}^2 v^i\left(s, X^i(s;t,x)\right)
    \mathcal{E}^i(s,t,x) \d{s}} \d{x}
  \\
  & \leq
  &  \int_\reali\modulo{ (\rho_o^i)'(\xi)
    \mathcal{E}^i\left(0,t,X^i(t;0,\xi)\right)}  \d\xi
    +
    \int_\reali \modulo{ \rho_o^i(\xi)} Q\,t \, e^{Qt}
    \mathcal{E}^i(0;t,\xi) \d\xi
  \\
     & \leq
     & \norma{\rho_o^i}_{\W11(\reali;\reali)} e^{Qt}
        \left( 1 + Q\, t \, e^{Qt} \right)
\end{eqnarray*}
proving~\eqref{eq:45}. Since
\begin{align*}
  \partial_{xx}^2 \rho^i(t,x) =
  & (\rho_o^i)''\left(X^i(0;t,x)\right) \, \mathcal{E}^i(0,t,x)^3
  \\
  & + 2 (\rho_o^i)' \left(X(0;t,x)\right) \, \mathcal{E}^i(0,t,x) \partial_x \mathcal{E}^i(0,t,x)
  \\
  & + (\rho_o^i)'\left(X^i(0;t,x)\right) \mathcal{E}^i(0,t,x) \int_t^0\partial_{xx}^2 v^i\left(s, X(s;t,x)\right) \mathcal{E}^i(s,t,x) \d{s}
  \\
  & + \rho_o^i\left(X^i(0;t,x)\right) \int_t^0 \partial_{xxx}^3 v^i\left(s, X^i(s;t,x)\right)\mathcal{E}^i(s,t,x)^2 \d{s}
  \\
  & + \rho_o^i\left(X^i(0;t,x)\right) \int_t^0 \partial_{xx}^2 v^i\left(s, X^i(s;t,x)\right)\partial_x \mathcal{E}^i(s,t,x) \d{s} \,,
\end{align*}
computations entirely similar to the ones used above ensure that
\begin{displaymath}
        \norma{\partial^2_{xx} \rho(t)}_{\L1(\reali; \reali^k)}
  \leq
  \norma{\rho_o^i}_{\W21(\reali;\reali)}e^{2Qt} \left(
    1
    +  4e^{Qt}Qt   + Q^2t^2
  \right).
\end{displaymath}
proving~\eqref{eq:46}. The bounds~\eqref{eq:54} and~\eqref{eq:55}
readily follow.

\paragraph{Claim~5:~\ref{item:26} and~\ref{item:47} hold.} Thanks to~\ref{item:25} in
Lemma~\ref{lem:4}, it is sufficient to prove the $\L1$ local Lipschitz
continuity in time of $\partial_x \rho$. To this aim, start
from~\eqref{eq:27} and, using~\eqref{eq:18}, compute:
\begin{flalign*}
  & \norma{\partial_x \rho^i (t) - \partial_x\rho^i (\widehat t)}_{\L1 (\reali;\reali)}
  \\
  \leq
  & \int_{\reali} \!
    \modulo{
    (\rho_o^i)'\left(X^i (0;t,x)\right) \mathcal{E}^i (0,t,x)^2
    {-}
    (\rho_o^i)'\left(X^i (0;\widehat t,x)\right) \mathcal{E}^i (0,\widehat t,x)^2
    } \d{x}
    \!\!\!\!\!\!\!\!\!\!\!\!\!\!\!\!\!\!\!\!\!
  \\
  & +
    \int_{\reali}
    \Big|
    \rho_o^i\left(X^i (0;t,x)\right)
    \int_0^t \partial^2_{xx}v^i\left(s, X^i (s;t,x)\right)
    \mathcal{E}^i (s,t,x) \d{s}
  \\
  &  -
    \rho_o^i\left(X^i (0;\widehat t,x)\right)
    \int_0^{\widehat t} \partial^2_{xx}v^i\left(s, X^i (s;\widehat t,x)\right)
    \mathcal{E}^i (s,\widehat t,x) \d{s} \Big| \d{x}
  \\
  \ghost{
  \leq
  & e^{2QT} \int_{\reali}
    \modulo{
    (\rho_o^i)'\left(X^i (0;t,x)\right)
    -
    (\rho_o^i)'\left(X^i (0;\widehat t,x)\right)
    } \d{x}
  \\
  & + \int_{\reali}
    \modulo{(\rho_o^i)'\left(X^i (0;\widehat t,x)\right)}
    \modulo{\mathcal{E}^i (0,t,x)^2
    -
    \mathcal{E}^i (0,\widehat t,x)^2} \d{x}
  \\
  & + \int_{\reali}
    \modulo{
    \rho_o^i\left(X^i (0;t,x)\right) - \rho_o^i\left(X^i (0;\widehat t,x)\right)
    }
    \int_t^0
    \modulo{\partial^2_{xx} v^i\left(s,X^i (s;t,x)\right)}
    \mathcal{E}^i (s,t,x)
    \d{x}
  \\
  &  +
    \int_{\reali} \modulo{\rho_o^i\left(X^i (0; \widehat t,x)\right)
    \int_{\widehat t}^t \partial^2_{xx} v^i \left(s,X^i (s;t,x)\right)
    \mathcal{E}^i (s,t,x) \d{s} } \d{x}
  \\
  &  + \int_{\reali} \modulo{\rho_o^i\left(X^i (0; \widehat t,x)\right)}
    \int_0^{\widehat t} \modulo{
    \partial^2_{xx} v^i \left(s,X^i (s;t,x)\right)
    {-}
    \partial^2_{xx} v^i \left(s,X^i (s;\widehat t,x)\right)}
    \mathcal{E}^i (s,t,x) \d{s}  \d{x}
  \\
  &  +
    \int_{\reali} \modulo{\rho_o^i\left(X^i (0; \widehat t,x)\right)}
    \int_0^{\widehat t} \modulo{\partial^2_{xx} v^i \left(s,X^i (s;t,x)\right)}
    \modulo{ \mathcal{E}^i (s,t,x) - \mathcal{E}^i (s,\widehat t,x) \d{s} }
    \d{x}
  \\}
  \leq
  &  2e^{QT} \tv\left((\rho_o^i)'\right)
    \sup_{x\in \reali} \modulo{X^i (0;t,x) - X^i (0;\widehat t,x)}
  & [\mbox{By~\cite[Lemma~4.1]{OperaPrima}}]
  \\
  & + 2 e^{2QT}
    \norma{ (\rho_o^i)'}_{\L1 (\reali;\reali)}
    \; \sup_{x \in \reali}
    \modulo{\mathcal{E}^i (0,t,x) - \mathcal{E}^i (0, \widehat t,x)}
  & [\mbox{By~\eqref{eq:59}, \eqref{eq:17}}]
  \\
  & +
    2\, Q\, T \, e^{QT} \tv\left((\rho_o^i)'\right)
    \sup_{x\in \reali} \modulo{X^i (0;t,x) - X^i (0;\widehat t,x)}
  & [\mbox{By~\cite[Lemma~4.1]{OperaPrima}, \eqref{eq:16}}]
  \\
  & + Q\, e^{2QT} \norma{ \rho_o^i}_{\L1 (\reali;\reali)}
    \modulo{t-\widehat t}
  & [\mbox{By~\eqref{eq:59}, \eqref{eq:17}}]
  \\
  & + Q e^{2QT}\norma{\rho_o^i}_{\L1 (\reali;\reali)}
    \int_0^{\widehat t} \sup_{x \in \reali}
    \modulo{X^i (s;t,x) - X^i (s;\widehat t, x)} \d{s}
  & [\mbox{By~\eqref{eq:59}, \eqref{eq:17}, \eqref{eq:16}}]
  \\
  & + Q e^{QT}\norma{\rho_o^i}_{\L1 (\reali;\reali)}
    \int_0^{\widehat t}
    \sup_{x \in \reali}
    \modulo{\mathcal{E}^i(s,t,x) - \mathcal{E}^i (s,\widehat t, x)} \d{s}\,.
  & [\mbox{By~\eqref{eq:59}, \eqref{eq:17}, \eqref{eq:16}}]
\end{flalign*}
To complete the above estimate, note that for all $s \in [0,T]$ and
for all $x \in \reali$,
\begin{flalign*}
  \modulo{X^i (s;t,x) - X^i (s;\widehat t, x)}
  \leq
  & Q \, e^{QT} \, \modulo{t-\widehat t} \,;
  & [\mbox{By~\eqref{eq:11}, \eqref{eq:17}
    and~\eqref{eq:16}}]
  \\
  \modulo{\mathcal{E}^i(s,t,x) - \mathcal{E}^i (s,\widehat t, x)}
  \leq
  & \left( Q + Q^2 T e^{QT}\right) \, \modulo{t - \widehat t} \,.
  & [\mbox{By~\eqref{eq:16}, \eqref{eq:18} and~\eqref{eq:60}}]
\end{flalign*}
and~\ref{item:26} follows. The standard embedding
$\W11 (\reali; \reali^k) \hookrightarrow \L\infty (\reali; \reali^k)$
allows to prove~\ref{item:47}.

\paragraph{Claim~6: \ref{item:36} holds.} The computations in Claim~4
show that
$\partial_{xx}^2\rho \in \C0\left([0,T] \times \reali;\reali^k\right)$
due to $\rho_o \in \C2(\reali;\reali^k)$, Claim~1 and~Claim~2.

\paragraph{Claim~7: \ref{item:38} holds. }
For $(t,x) \in [0,T]\times\reali$, by the computations in Claim~3,
~\eqref{eq:16} and~\eqref{eq:18}
   \begin{eqnarray}
     \label{eq:41}
     \modulo{\partial_x\rho^i(t,x)}
     & \leq
     & \norma{\rho_o^i}_{\W1\infty(\reali;\reali)} e^{2QT}
     \\
     \nonumber
     \modulo{\partial_t\rho^i(t,x)}
     & \leq
     & \norma{(\rho_o^i)'}_{\L\infty(\reali;\reali)} e^{2QT}Q + \norma{(\rho_o^i)}_{\L\infty(\reali;\reali)} \left(Q + Q^2e^{QT}T \right)
     \\
     \nonumber
     \modulo{\partial_{xx}^2\rho^i(t,x)}
     & \leq
     & \norma{\rho_o^i}_{\W2\infty(\reali; \reali)} \left(
     e^{2QT} + 4e^{QT}QT + Q^2T^2
     \right) e^{QT}
     \\\nonumber
     \modulo{\partial_{tx}^2\rho^i(t,x)}
     & \leq
     & \modulo{ ((\rho_o^i)'')\left(X^i(0;t,x)\right)\partial_t X^i(0;t,x)\mathcal{E}^i(0,t,x)^2}
     \\\nonumber
     &
     & + \modulo{(\rho_o^i)'\left(X^i(0;t,x)\right)2\mathcal{E}^i(0,t,x)\partial_t \mathcal{E}^i(0,t,x)}
     \\\nonumber
     &
     & +\modulo{ (\rho_o^i)'\left(X^i(0;t,x)\right)\partial_t X^i(0;t,x) \int_t^0 \partial_{xx}^2 v^i\left(s, X^i(s;t,x)\right)\mathcal{E}^i(s,t,x) \d{s}}
     \\\nonumber
     &
     &+ \modulo{ \rho_o^i\left(X^i(0;t,x)\right) \partial_{xx}^2v^i(t,x)}
     \\\nonumber
     &
     & + \modulo{ \rho_o^i\left(X^i(0;t,x)\right)
       \int_t^0 \partial_{xxx}^3 v^i\left(s, X^i(s;t,x)\right)\partial_t X^i(s;t,x) \mathcal{E}^i(s,t,x) \d{s}
       }
     \\
     \nonumber
     &
     & + \modulo{ \rho_o^i\left(X^i(0;t,x)\right) \int_t^0  \partial_{xx}^2 v^i\left(s, X^i(s;t,x)\right)\partial_t \mathcal{E}^i(s,t,x) \d{s} }
     \\
     \ghost{
     \nonumber
     & \leq
     & \norma{(\rho_o^i)''}_{\L\infty(\reali;\reali)} \modulo{\mathcal{E}^i(0,t,x)v^i(t,x)} e^{2QT}
     \\\nonumber
     &
     & +
       2\norma{(\rho_o^i)'}_{\L\infty(\reali;\reali)} e^{QT}
       \modulo{
       \partial_x v^i(t,x)
       +
       \int_t^0 \partial_{xx}^2 v^i\left(s, X(s;t,x)\right)
       \mathcal{E}^i(s,t,x)v^i(t,x)\d{s}}
     \\\nonumber
     &
     &+ \norma{(\rho_o^i)'}_{\L\infty(\reali;\reali)}\modulo{\mathcal{E}^i(0,t,x)v^i(t,x)}Qe^{QT}T
     \\\nonumber
     &
     &+ \norma{\rho_o^i}_{\L\infty(\reali;\reali)} Q
     \\\nonumber
     &
     &+ \norma{\rho_o^i}_{\L\infty(\reali;\reali)} TQe^{QT}e^{QT}Q
     \\\nonumber
     &
     &+ \norma{\rho_o^i}_{\L\infty(\reali;\reali)}TQ (
       Q+ Qe^{QT}QT
       )
     \\}
     \nonumber
     &\leq
     & \norma{\rho_o^i}_{\W2\infty(\reali;\reali)}Q \left( 1 + TQ + e^{QT} \left( 2 + T^2Q^2
       \right) + e^{2QT} 4QT +  e^{3TQ} \right) \,,
   \end{eqnarray}
   completing the proof of Claim~7.
 \end{proof}

 \begin{lemma}[\emph{A priori} estimate on $V$]
   \label{lem:3}
   Assume that $\eta \in \W3\infty(\reali;\reali^k)$, $v_{NL}$
   satisfies~\ref{item:17} and
   $r \in \C0\left([0,T];\L1(\reali;\reali)\right)$.  Define for
   $i=1, \ldots, k$ and for
   $(t,x,q) \in [0,T]\times\reali\times\reali^k$
   \begin{displaymath}
     V^i(t,x,q) \coloneq v_{NL}^i\left( \Sigma q + (r(t)*\eta^i)(x)\right) \,.
   \end{displaymath}
   Then, $V$ belongs to $\C0\left([0,T];\mathcal{V}\right)$ and for
   all $t \in [0,T]$
   \begin{align}\label{eq:24}
     \norma{V}_{\C0([0,t];\mathcal{V})} \leq \mathcal{C}\left(
     \norma{\eta}_{\W3\infty(\reali; \reali^k)},
     \norma{v_{NL}}_{\W3 \infty(\reali; \reali^k)} ,
     \norma{r}_{\C0([0,t]; \L1(\reali; \reali))}
     \right).
   \end{align}
 \end{lemma}

 \begin{proof}
   Compute, for all $(t,x,q) \in [0,T]\times\reali\times\reali$ and $i,j,l,m=1, \ldots, k $
   \begin{eqnarray*}
     \modulo{V^i(t,x,0)}
     &\leq
     & \norma{v^i_{NL}}_{\L\infty(\reali;\reali)}
     \\
     \modulo{\partial_x V^i(t,x,q)}
     \ghost{
     & =
     & \modulo{(v^i_{NL})'(\Sigma q + (r(t)*\eta^i)(x))
       \, \partial_x(r(t)*\eta^i)(x)}
     \\
     }
     & \leq
     & \norma{(\eta^i)'}_{\L\infty(\reali;\reali)} \,
       \norma{(v^i_{NL})'}_{\L\infty(\reali;\reali)} \,
       \norma{r(t)}_{\L1(\reali;\reali)}
     \\
     \modulo{\partial_{q_j} V^i(t,x,q)}
     \ghost{
     & =
     & \modulo{(v^i_{NL})'(\Sigma q + (r(t)*\eta^i)(x))}
     \\
     }
     &\leq
     & \norma{(v_{NL}^i)'}_{\L\infty(\reali;\reali)}
     \\
     \modulo{\partial_{xx}^2V^i(t,x,q)}
     \ghost{
     & \leq
     & \modulo{(v^i_{NL})''(\Sigma q+(r(t)*\eta^i)(x))
       \left((r(t)*(\eta^i)')(x)\right)^2}
     \\
     &
     & + \modulo{ (v^i_{NL})'(\Sigma q + (r(t)*\eta^i(x)))
       \left( (r(t)*(\eta^i)'')(x)\right)}
     \\
     }
     & \leq
     & \norma{(\eta^i)'}_{\L\infty(\reali;\reali)}^2 \,
       \norma{(v^i_{NL})''}_{\L\infty(\reali;\reali)} \,
       \norma{r(t)}^2_{\L1(\reali;\reali)}
     \\
     &
     &+ \norma{(\eta^i)''}_{\L\infty(\reali;\reali)} \,
       \norma{(v^i_{NL})'}_{\L\infty(\reali;\reali)} \,
       \norma{r(t)}_{\L1(\reali;\reali)}
     \\
     \modulo{\partial_{xq_j}^2 V^i(t,x,q)}
     \ghost{
     & =
     & \modulo{(v^i_{NL})''\left(\Sigma q+(r(t)*\eta^i)(x)\right) \,
       (r(t)*(\eta^i)')(x)}
     \\
     }
     &\leq
     & \norma{(\eta^i)'}_{\L\infty(\reali;\reali)} \,
       \norma{(v^i_{NL})''}_{\L\infty(\reali;\reali)} \,
       \norma{r(t)}_{\L1(\reali;\reali)}
     \\
     \modulo{\partial_{q_jq_l}^2 V^i(t,x,q)}
     \ghost{
     & =
     & \modulo{(v^i_{NL})''(\Sigma q + (r(t)*\eta^i)(x))}
     \\
     }
     & \leq
     & \norma{(v^i_{NL})''}_{\L\infty(\reali;\reali)}
     \\
     \modulo{\partial_{xxx}^3V^i(t,x,q)}
     \ghost{
     & \leq
     & \modulo{(v^i_{NL})'''(\Sigma q + \left(r(t)*\eta^i\right)(x))
       \left( (r(t)*(\eta^i)')(x)\right)^3}
     \\
     &
     & + \modulo{
       3 (v^i_{NL})''(\Sigma q + \left(r(t)*\eta^i\right)(x))
       \left( (r(t)*(\eta^i)')(x)\right)
       \left( (r(t)*(\eta^i)'')(x)\right)
       }
     \\
     &
     & + \modulo{
       (v^i_{NL})'(\Sigma q + \left(r(t)*\eta^i\right)(x))
       \left(r(t)*(\eta^i)''')(x)\right)
       }
     \\
     }
     &\leq
     & \norma{(\eta^i)'}_{\L\infty(\reali;\reali)}^3 \,
       \norma{(v^i_{NL})'''}_{\L\infty(\reali;\reali)} \,
       \norma{r(t)}_{\L1(\reali;\reali)}^3\,
     \\
     &
     & +
       3\norma{(\eta^i)'}_{\L\infty(\reali;\reali)}\,
       \norma{(\eta^i)''}_{\L\infty(\reali;\reali)} \,
       \norma{(v^i_{NL})''}_{\L\infty(\reali;\reali)} \,
       \norma{r(t)}_{\L1(\reali;\reali)}^2
     \\
     &
     & +
       \norma{(\eta^i)'''}_{\L\infty(\reali;\reali)}
       \norma{(v^i_{NL})'}_{\L\infty(\reali;\reali)}\,
       \norma{r(t)}_{\L1(\reali;\reali)}\,
     \\
     \modulo{\partial_{xxq_j}^3  V^i(t,x,q)}
     \ghost{
     & \leq
     & \modulo{(v^i_{NL})'''(\Sigma q+(r(t)*\eta^i)(x))
       \left((r(t)*(\eta^i)')(x)\right)^2}
     \\
     &
     & + \modulo{ (v^i_{NL})''
       \left(\Sigma q + \left(r(t)*\eta^i(x)\right)\right)
       \left( (r(t)*(\eta^i)'')(x)\right)}
     \\
     }
     & \leq
     & \norma{(\eta^i)'}_{\L\infty(\reali;\reali)}^2 \,
       \norma{(v^i_{NL})'''}_{\L\infty(\reali;\reali)} \,
       \norma{r(t)}^2_{\L1(\reali;\reali)}
     \\
     &
     &+ \norma{(\eta^i)''}_{\L\infty(\reali;\reali)} \,
       \norma{(v^i_{NL})''}_{\L\infty(\reali;\reali)} \,
       \norma{r(t)}_{\L1(\reali;\reali)}
     \\
     \modulo{\partial_{xq_jq_l}^3 V^i(t,x,q)}
     \ghost{
     & =
     & \modulo{(v^i_{NL})'''\left(\Sigma q+(r(t)*\eta^i)(x)\right) \,
       (r(t)*(\eta^i)')(x)}
     \\
     }
     &\leq
     & \norma{(\eta^i)'}_{\L\infty(\reali;\reali)} \,
       \norma{(v^i_{NL})'''}_{\L\infty(\reali;\reali)} \,
       \norma{r(t)}_{\L1(\reali;\reali)}
     \\
     \modulo{\partial_{q_jq_lq_m}^3 V^i(t,x,q)}
     \ghost{
     & =
     & \modulo{(v^i_{NL})''\left(\Sigma q + (r(t)*\eta^i)(x)\right)}
     \\
     }
     & \leq
     &  \norma{(v^i_{NL})'''}_{\L\infty(\reali;\reali)} \,,
   \end{eqnarray*}
hence for all $t\in[0,T]$, $V(t) \in \mathcal{V}$. Moreover, since
for all
   $t, t_o \in [0,T]$,
   \begin{displaymath}
     \norma{r (t) *\eta^i - r (t_o)*\eta^i}_{\W3\infty (\reali;\reali)}
     \leq
     \norma{r (t) - r (t_o)}_{\L1 (\reali;\reali)} \,
     \norma{\eta^i}_{\W3\infty (\reali;\reali)}
   \end{displaymath}
   we get that $V \in \C0 ([0,T];\mathcal{V})$.  The
   $\L\infty (\reali;\reali^k)$ continuity of
   $t \mapsto V (t,\cdot,0)$ and of
   $t \mapsto [\partial_x V (t) \quad \nabla_q V (t)]$ in
   $\W2\infty (\reali\times\reali^k; \reali^k \times \reali^{k \times
     k})$ then follows by the regularity of $v_{NL}$.
 \end{proof}

 \begin{lemma}[Stability in $\W11$]
   \label{lem:Linfty}
   Fix $\mathcal{R}>0$.  Let $\eta \in \W3\infty(\reali;\reali^k)$,
   $v_{NL}$ satisfy~\ref{item:17}. For all
   $r,\widehat r \in \C0 \left([0,T];\L1 (\reali;\reali)\right)$ with
   $\norma{r}_{\C0 \left([0,T];\L1 (\reali;\reali)\right)} \leq
   \mathcal{R}$,
   $\norma{\widehat r}_{\C0 \left([0,T];\L1 (\reali;\reali)\right)}
   \leq \mathcal{R}$, $\rho_o \in \cBV{1} (\reali; \reali^k)$ and
   $\widehat \rho_o \in \W11 (\reali;\reali^k)$, call
   $\rho,\widehat\rho$ the solutions, in the sense of
   Definition~\ref{def:solution}, to
   \begin{equation}
     \label{eq:23}
     \left\{
       \begin{array}{l}
         \partial_t \rho^i
         + \partial_x
         \left(\rho^i \; v_{NL}^i \left(
         \Sigma \rho *\eta^i + r*\eta^i
         \right)\right) =0
         \qquad i=1, \ldots, k
         \\
         \rho (0,x) = \rho_o (x)
       \end{array}
     \right.
   \end{equation}
   \begin{equation}
     \label{eq:26}
     \left\{
       \begin{array}{l}
         \partial_t \widehat\rho^i
         + \partial_x
         \left(\widehat\rho^i \; v_{NL}^i \left(
         \Sigma \widehat \rho *\eta^i + \widehat r*\eta^i
         \right)\right) =0
         \qquad i=1, \ldots, k
         \\
         \widehat \rho (0,x) = \widehat \rho_o (x)
       \end{array}
     \right.
   \end{equation}
   Then, there exists a constant $\mathcal{C}_{3}$ such that for all $t \in [0,T]$
   \begin{equation}
     \label{eq:56}
     \norma{\rho (t) - \widehat\rho (t)}_{\W11 (\reali; \reali^k)}
     \leq
     (1+\mathcal{C}_{3}\,t)
     \norma{\rho_o - \widehat\rho_o}_{\W11 (\reali;\reali^k)}
     +
     \mathcal{C}_{3} \int_0^t \norma{r (\tau) - \widehat r (\tau)}_{\L1 (\reali;\reali)} \d\tau
   \end{equation}
   where
   \begin{equation}
     \label{eq:57}
     \mathcal{C}_{3}\! \coloneqq \!
     \mathcal{C}\!
     \left(
       \norma{\eta}_{\W3\infty(\reali;\reali^k)},
       \norma{v_{NL}}_{\W3\infty(\reali;\reali^k)},
       \mathcal{R},
       \norma{\rho_o}_{\cBV{1}(\reali;\reali^k)},
       \norma{\widehat\rho_o}_{\cBV{} (\reali;\reali^k)},
       t
     \right)
   \end{equation}
 \end{lemma}

 Using the particular form of~\eqref{eq:23}--\eqref{eq:26},
 Lemma~\ref{lem:Linfty} improves the stability
 estimates~\ref{item:44}--\ref{item:45} in Lemma~\ref{lem:4}.

 \begin{proofof}{Lemma~\ref{lem:Linfty}}
   Throughout, fix $i=1, \ldots,k$. \Cref{lem:3} ensures that we
   can apply \Cref{lem:4} to~\eqref{eq:23} and~\eqref{eq:26}. Hence,
   there exist unique
   $\rho, \widehat{\rho} \in \C0\left([0,T];
     \L1(\reali;\reali^k)\right)$ solutions to~\eqref{eq:23}
   and~\eqref{eq:26} in the sense of \Cref{def:solution}. Introduce
   the velocities
   \begin{eqnarray*}
     v^i (t,x)
     & \coloneqq
     & v_{NL}^i \left(
       \left(\Sigma \rho (t) *\eta^i\right) (x) +
       \left(r (t)*\eta^i\right) (x)
       \right),
     \\
     \widehat v^i (t,x)
     & \coloneqq
     & v_{NL}^i \left(
       \left(\Sigma \widehat\rho (t) *\eta^i\right) (x) +
       \left(\widehat r (t)*\eta^i\right) (x)
       \right).
   \end{eqnarray*}
   The same computations as in Claim~1 in~\Cref{lem:5} and the
   estimate~\eqref{eq:24} in~\Cref{lem:3}, yield that there exists a
   constant
   \begin{equation*}
     Q \coloneqq \mathcal{C}\left(
       \norma{\eta'}_{\W1\infty(\reali; \reali^k)},
       \norma{v_{NL}}_{\W2\infty(\reali; \reali^k)},
       \mathcal{R},
       \norma{\rho_o}_{\L1(\reali; \reali^k)},
       \norma{\widehat\rho_o}_{\L1(\reali; \reali^k)}
     \right)
   \end{equation*}
   such that
   \begin{equation}
     \label{eq:37}
     \max
     \left\{\norma{\partial_x v^i}_{\C0 ([0,t];\W1\infty (\reali;\reali))}
       \,,\;
       \norma{\partial_x \widehat v^i}_{\C0 ([0,t];\W1\infty (\reali;\reali))}
     \right\} \leq Q \,.
   \end{equation}
   With this choice, we have
   \begin{equation}
     \label{eq:52}
     \max
     \left\{
       \mathcal{E}^i(t,t_o,x_o) ,
       \widehat{\mathcal{E}}^i(t,t_o,x_o)
     \right\}
     \leq
     e^{Q \modulo{t-t_o}}
     \leq e^{QT} \, \quad
     \forall t,t_o\in [0,T], x_o \in \reali \,.
   \end{equation}
   Moreover, for all $s,t \in [0,T], s\leq t$, the following estimate
   holds by an adaptation of~\cite[Proposition~4.2]{OperaPrima}:
   \begin{align}
     \label{eq:30}
     \modulo{X^i(s;t,x) - \widehat X^i(s;t,x)}
     \leq
     & \int_s^t \norma{v^i(\tau) - \widehat v^i(\tau)}_{ \L\infty(\reali;\reali))}
       e^{Q(t-\tau)} \d\tau
     \\
     \label{eq:31}
     \leq
     & \modulo{t-s} \, e^{Q \modulo{t-s}} \,
       \norma{v^i - \widehat v^i}_{\C0([0,t]; \L\infty(\reali;\reali))}\,.
   \end{align}
   At last, taking advantage of the expression~\eqref{eq:12} and
   applying the estimates~\eqref{eq:37}--\eqref{eq:31}, one gets that
   for all $s,t \in [0,T]$, $s \leq t$ $x \in \reali$
   \begin{align}
     \nonumber
     & \modulo{\mathcal{E}^i(s,t,x) - \widehat{\mathcal{E}}^i(s,t,x)}
     \\
     \ghost{
     \nonumber
     =
     & \modulo{ \exp \left( \int_t^s
       \partial_x v^i\left(\tau, X^i(\tau; t,x)\right)
       \d\tau \right) - \exp \left( \int_t^s \partial_x \widehat
       v^i\left(\tau, \widehat X^i(\tau; t,x)\right) \d\tau \right) }
     \\
     }
     \nonumber
     \leq
     & \max \left\{
       e^{\norma{\partial_x \widehat v^i}_{\C0([0,t];       \L\infty(\reali;\reali))}
       \modulo{t-s}},
       e^{\norma{\partial_x v^i}_{\C0([0,t]; \L\infty(\reali;\reali))}
       \modulo{t-s}}
       \right\}
     \\
     \nonumber
     &  \times \modulo{ \int_t^s \partial_x v^i\left(\tau, X^i(\tau;
       t,x)\right) - \partial_x \widehat v^i\left(\tau, \widehat X^i(\tau; t,x)\right) \dd
       \tau }
     \\
     \nonumber
     \leq
     & e^{ Q\modulo{t-s}}
       \modulo{ \int_t^s \partial_x v^i\left(\tau, X^i(\tau;t,x)\right) - \partial_x v^i\left(\tau, \widehat X^i(\tau; t,x)\right) \d\tau }
     \\
     \nonumber
     & + e^{ Q\modulo{t-s}} \modulo{ \int_t^s \partial_x
       v^i\left(\tau, \widehat X^i(\tau; t,x)\right) -
       \partial_x \widehat v^i\left(\tau, \widehat X^i(\tau; t,x)\right)
       \d\tau }
     \\
     \nonumber
     \leq
     & e^{ Q\modulo{t-s}} \left(
       Q \, \int_s^t \modulo {X^i(\tau; t,x) - \widehat
       X^i(\tau; t,x)} \d\tau
       +
       \int_s^t \norma{\partial_x v^i (\tau) -
       \partial_x \widehat v^i(\tau)}_{\L\infty(\reali;\reali)} \d\tau
       \right)
     \\
     \nonumber
     \leq
     & e^{ Q\modulo{t-s}}\!
       \left(
       \!Q \int_s^t \!\!\int_\tau^t
       \norma{v^i(\tau') - \widehat v^i(\tau')}_{ \L\infty(\reali;\reali))} \!
       e^{Q(t-\tau')} \d\tau' \! \d\tau
       +\! \int_s^t \!\norma{\partial_x v^i (\tau) -
       \partial_x \widehat v^i(\tau)}_{\L\infty(\reali;\reali)} \d\tau\!
       \right)
     \\
     \nonumber
     \leq
     & e^{Q\modulo{t-s}}
       \left( e^{ Q\modulo{t-s}}
       Q \modulo{t-s} \int_s^t
       \norma{v^i(\tau) - \widehat v^i(\tau)}_{ \L\infty(\reali;\reali))} \d\tau
       +   \int_s^t \norma{\partial_x v^i (\tau) -
       \partial_x \widehat v^i(\tau)}_{\L\infty(\reali;\reali)} \d\tau
       \right)
     \\
     \label{eq:49}
     \leq
     & e^{ 2 Q\modulo{t-s}} \left(Q \modulo{t-s} + 1 \right)
       \norma{v^i - \widehat v^i}_{\L1((s, t); \W1\infty(\reali; \reali))} \,.
   \end{align}
   Now observe that for all $\tau \in [0,T]$
   \begin{align}
     \nonumber
     &
       \norma{v^i(\tau, \cdot) -
       \widehat v^i(\tau, \cdot )}_{\L\infty(\reali;\reali)}
     \\
     \nonumber
     = &
         \sup_{x \in \reali}
         \modulo{v_{NL}^i \!\left(
         \left(\Sigma\rho(\tau)*\eta^i\right)(x) + \!(r(\tau) *\eta^i)(x)\right)
         -
         v_{NL}^i \!\left(\!
         \left(\Sigma \widehat\rho(\tau)*\eta^i\right)(x)+(\widehat r(\tau) *\eta)(x) \right)}
     \\
     \label{eq:32}
     \leq &
            \norma{v_{NL}'}_{\L\infty(\reali;\reali^k)} \,
            \norma{\eta}_{\L\infty(\reali;\reali^k)}
            \left(
            \norma{\rho(\tau) - \widehat \rho(\tau)}_{\L1(\reali;\reali^k)}
            +
            \norma{r(\tau) - \widehat r(\tau)}_{\L1(\reali;\reali)}
            \right),
   \end{align}
   and
   \begin{eqnarray}
     \nonumber
     &
     & \norma{\partial_x v^i(\tau) - \partial_x \widehat v^i(\tau)}_{\L\infty(\reali;\reali)}
     \\
     \nonumber
     & \leq
     & \left( \norma{(v_{NL}^i)'}_{\L\infty(\reali;\reali)}
       + \norma{(v_{NL}^i)''}_{\L\infty(\reali;\reali)}
       \norma{\eta^i}_{\L\infty(\reali;\reali)}
       \norma{\Sigma \widehat \rho(\tau) +
       \widehat r(\tau)}_{\L1(\reali;\reali)}
       \right)
     \\
     \nonumber
     &
     & \times
       \norma{(\eta^i)'}_{\L\infty(\reali;\reali)}
       \left(
       \sum_{j=1}^k \norma{\rho^j(\tau)- \widehat \rho^j(\tau)}_{\L1(\reali;\reali)}
       +
       \norma{r(\tau)- \widehat r(\tau) }_{\L1(\reali;\reali)}
        \right)
     \\
     \nonumber
     & \leq
     & \left(
       \norma{v_{NL}'}_{\L\infty(\reali;\reali^k)}
       +
       \norma{v_{NL}''}_{\L\infty(\reali;\reali^k)}
       \norma{\eta}_{\L\infty(\reali;\reali^k)}
       \left(
 \norma{\widehat\rho_o}_{\L1(\reali;\reali^k)}
        + \mathcal{R}
       \right)
       \right) \norma{\eta'}_{\L\infty(\reali;\reali^k)}
     \\
     \label{eq:33}
     &
     & \times
       \left(
       \norma{\rho(\tau)- \widehat \rho(\tau )}_{\L1(\reali;\reali^k)}
       +
       \norma{r(\tau)- \widehat r(\tau) }_{\L1(\reali;\reali)}
       \right)
   \end{eqnarray}
   by~\Cref{item:46} in~\Cref{lem:4} and
   \begin{eqnarray}
     \nonumber
     &
     & \norma{\partial_{xx}^2 v^i(\tau, \cdot) -
       \partial_{xx}^2 \widehat v^i(\tau, \cdot)}_{\L\infty(\reali;\reali)}
     \\
     \label{eq:34}
     & \leq
     & \left(
       1 +
       3 \norma{\eta}_{\W1\infty(\reali;\reali^k)}
       \left( \norma{\rho_o}_{\L1(\reali;\reali^k)} + \mathcal{R}
       \right)
       +
       \norma{\eta'}^2_{\L\infty(\reali;\reali^k)}
       \left(\norma{\rho_o}_{\L1(\reali;\reali^k)} + \mathcal{R}
       \right)^2
       \right)
     \\
     \nonumber
     &
     & \times
       \norma{\eta}_{\W2\infty(\reali;\reali^k)}
       \norma{v_{NL}'}_{\W2\infty(\reali;\reali^k)}
       \left( \norma{\rho(\tau) - \widehat{\rho}(\tau) }_{ \L1(\reali;\reali^k)}
       + \norma{r(\tau) - \widehat r(\tau)}_{ \L1(\reali;\reali)}
       \right).
   \end{eqnarray}

   \paragraph{Bound on the distance between $\rho^i (t)$ and
     $\widehat\rho^i (t)$ in $\L1(\reali; \reali^k)$.}
   Owing to~\eqref{eq:13}, one obtains that for all $t \in [0,T]$
   \begin{align}
     \nonumber
     & \norma{\rho^i(t) - \widehat{\rho}^i(t)}_{\L1(\reali; \reali)}
     \\
     \ghost{
     \nonumber
     =
     & \int_\reali \modulo{
       \rho_o^i\left(X^i(0;t,x)\right)
       \mathcal{E}^i (0;t,x)
       -
       \widehat\rho_o^i\left( \widehat{X}^i(0;t,x) \right)
       \widehat{\mathcal{E}}^i (0;t,x)
       }  \, \d{x}
     \\
     }
     \nonumber
     \leq
     & \int_\reali
       \modulo{
       \left(
       \rho_o^i\left( X^i(0;t,x) \right)
       -
       \widehat \rho_o^i\left( X^i(0;t,x) \right)
       \right)
       \mathcal{E}^i(0,t,x)
       }
       \d x
     \\
     \nonumber
     &
       +
       \int_\reali \modulo{
       \left(
       \widehat \rho_o^i\left(X^i(0;t,x) \right) -
       \widehat \rho_o^i\left( \widehat{X}^i(0;t,x) \right)
       \right)
       \mathcal{E}^i(0,t,x)}  \, \d{x}
     \\
     \nonumber
     & +
       \int_\reali \modulo{
       \widehat \rho_o^i \left( \widehat{X}^i(0;t,x) \right)
       \left(
       \mathcal{E}^i(0,t,x) - \widehat{\mathcal{E}}^i(0;t,x)
       \right)
       }  \, \d{x}
     \\
     \nonumber
     \leq
     & \norma{\rho_o^i - \widehat \rho_o^i}_{\L1(\reali; \reali)}
       +
       2 \tv(\widehat\rho_o^i) \,
       \sup_{x \in \reali} \modulo{X^i(0;t,x) - \widehat X^i(0;t,x)}
     & [\mbox{By~\cite[Lemma 4.1]{OperaPrima}}]
     \\
     \nonumber
     & + e^{3Qt} \norma{\widehat \rho_o^i}_{\L1(\reali; \reali)}
       (Q t + 1)
       \norma{v^i - \widehat v^i}_{\L1([0, t]; \W1\infty(\reali; \reali))}
     & [\mbox{By~\eqref{eq:49}}]
     \\ \label{eq:42}
     \leq
     & \norma{\rho_o^i - \widehat \rho_o^i}_{\L1(\reali; \reali)}
       {+}
       e^{Qt} \!\!
       \left( \!
       2\tv(\widehat \rho_o^i) {+}
       e^{2Qt} \norma{\widehat \rho_o^i}_{\L1(\reali; \reali)}
       (Q t \!+\! 1)
       \right)
       \norma{v^i {-} \widehat v^i}_{\L1([0, t]; \W1\infty(\reali; \reali))} \!.
       \!\!\!\!\!\!\!\!\!\!\!\!\!\!\!\!\!\!\!\!\!\!\!\!
       \!\!\!\!\!\!\!\!\!\!\!\!\!\!\!\!\!\!\!\!
   \end{align}
where in the previous computations we exploit~\eqref{eq:30}.  Now,
inserting~\eqref{eq:32} and~\eqref{eq:33} in~\eqref{eq:42} allows to
prove
\begin{eqnarray*}
  &
  & \norma{\rho(t) - \widehat\rho(t)}_{\L1(\reali; \reali^k)}
  \\
  & \leq
  & \norma{\rho_o - \widehat \rho_o}_{\L1(\reali; \reali^k)}
  \\
  &
  & +\;
    \mathcal{C}\left(
    \norma{\eta}_{\W2\infty(\reali; \reali^k)},
    \norma{v_{NL}}_{\W2\infty(\reali; \reali^k)},
    \norma{\rho_o}_{\L1(\reali; \reali^k)},
    \norma{\widehat \rho_o}_{\cBV{}(\reali; \reali^k)},
    \mathcal{R},
    t
    \right)
  \\
  &
  & \quad\times
    \int_0^t
    \left(
    \norma{\rho(\tau)- \widehat \rho(\tau)}_{\L1(\reali; \reali^k)}
    +
    \norma{r(\tau)- \widehat r(\tau)}_{\L1(\reali; \reali)}
    \right)
    \d \tau,
\end{eqnarray*}
and, by an application of Gronwall Lemma, it is possible to
obtain
\begin{eqnarray}
  \nonumber
  &
  & \norma{\rho(t) - \widehat\rho(t)}_{\L1(\reali; \reali^k)}
  \\
  \nonumber
  & \leq
  & \left(1
    +
    \mathcal{C}\left(
    \norma{\eta}_{\W2\infty(\reali; \reali^k)},
    \norma{v_{NL}}_{\W2\infty(\reali; \reali^k)},
    \norma{\rho_o}_{\L1(\reali; \reali^k)},
    \norma{\widehat \rho_o}_{\cBV{}(\reali; \reali^k)},
    \mathcal{R},
    t
    \right) t\right)
  \\
  \label{eq:50}
  &
  & \times
    \left(
    \norma{\rho_o - \widehat \rho_o}_{\L1(\reali; \reali^k)}
    +
    \int_0^t
    \norma{r(\tau)- \widehat r(\tau)}_{\L1(\reali; \reali)} \d\tau
    \right) \,.
\end{eqnarray}

\paragraph{Bound on the distance between $\partial_x\rho$ and
  $\partial_x\widehat\rho$ in $\C0\left([0,T]; \L1(\reali; \reali^k)\right)$.}
An application of~\cite[Lemma~4.1]{OperaPrima}, together with the
estimates~\eqref{eq:37} and~\eqref{eq:52}, leads to
\begin{align*}
  & \norma{\partial_x \rho (t) - \partial_x \widehat{\rho}^i (t)}_{\L1(\reali;\reali)}
  \\
  \leq &
         \int_\reali \modulo{
         \left(
         (\rho^i_o)'\left(X^i(0;t,x)\right) - (\rho^i_o)'\left(\widehat X^i(0;t,x)\right)
         \right)
         \mathcal{E}^i(0,t,x)^2
         } \d{x}
  \\
  & + \int_\reali \modulo{
    (\rho^i_o)'\left(\widehat X^i(0;t,x)\right)
    \left(
    \mathcal{E}^i(0,t,x)^2 - \widehat{\mathcal{E}^i}(0;t,x)^2
    \right)} \d{x}
  \\
  &  +
    \int_\reali \modulo{
    (\rho^i_o)'\left(\widehat X^i(0;t,x)\right) - (\widehat \rho_o^i) '\left(\widehat X^i(0;t,x)\right)
    }
    \widehat{\mathcal{E}^i}(0;t,x)^2
    \d{x}
  \\
  & +
    \int_\reali
    \modulo{
    \left(
    \rho^i_o\left(X^i(0;t,x)\right) - \rho_o^i\left(\widehat X^i(0;t,x)\right)
    \right)
    \int_0^t \partial_{xx}^2v^i\left(s, X^i(s;t,x)\right)\mathcal{E}^i(s,t,x)\d{s}
    } \d{x}
  \\
  &  +
    \int_\reali \modulo{
    \rho^i_o\left(\widehat X^i(0;t,x)\right)
    \int_0^t
    \left(
    \partial_{xx}^2v^i\left(s, X^i(s;t,x)\right) -
    \partial_{xx}^2\widehat v^i\left(s, X^i(s;t,x)\right)
    \right)
    \mathcal{E}^i(s,t,x)\d{s}} \d{x}
  \\
  &  +
    \int_\reali \modulo{
    \rho^i_o\left(\widehat X^i(0;t,x)\right)
    \int_0^t \left(
    \partial_{xx}^2\widehat v^i\left(s, X^i(s;t,x)\right)
    -
    \partial_{xx}^2\widehat v^i\left(s, \widehat X^i(s;t,x)\right)
    \right)
    \mathcal{E}^i(s,t,x)\d{s}
    } \d{x}
  \\
  &  +
    \int_\reali \modulo{
    \rho^i_o\left(\widehat X^i(0;t,x)\right)
    \int_0^t \partial_{xx}^2\widehat v^i\left(s, \widehat X^i(s;t,x)\right)
    \left(
    \mathcal{E}^i(s,t,x) -
    \widehat{\mathcal{E}^i}(s;t,x)
    \right)\d{s}
    } \d{x}
  \\
  &  +
    \int_\reali
    \modulo{
    \left(
    \rho^i_o\left(\widehat X^i(0;t,x)\right) - \widehat \rho^i_o\left(\widehat X^i(0;t,x)\right)
    \right)
    \int_0^t \partial_{xx}^2 \widehat v^i\left(s,\widehat X^i(s;t,x)\right)
    \widehat{\mathcal{E}^i}(s;t,x) \d{s}
    }
    \d{x}
  \\
  \leq &
         2 e^{Qt}
         \tv((\rho^i_o)')
         \norma{X^i(0;t,\cdot) - \widehat X^i(0;t,\cdot)}_{\L\infty(\reali;\reali)}
         \\
         &
         +
         2 e^{2Qt}
         \norma{(\rho^i_o)'}_{\L1(\reali,\reali)}
         \norma{\mathcal{E}^i(0,t,\cdot) - \widehat{\mathcal{E}^i}(0,t, \cdot)}_{\L\infty(\reali;\reali)}
  \\
  &
    +
    e^{Qt}\norma{(\rho^i_o)' - (\widehat\rho_o^i) '}_{\L1(\reali; \reali)}
    +
    2 Qt e^{2Qt} \tv(\rho_o^i)
    \norma{X^i(0;t,\cdot) - \widehat X^i(0;t,\cdot)}_{\L\infty(\reali;\reali)}
  \\
  &  +
    e^{2Qt}
    \norma{\rho_o^i}_{\L1(\reali;\reali)}
    \int_0^t
    \norma{\partial_{xx}^2v^i(s) - \partial_{xx}^2 \widehat v^i(s)}_{\L\infty(\reali;\reali)} \d{s}
  \\
  &  +
    e^{2Qt} \norma{\rho_o^i}_{\L1(\reali;\reali)}
    \norma{\partial_{xxx}^3 \widehat v^i}_{\L\infty([0,t]\times \reali;\reali)}
    \int_0^t
    \norma{X^i(s;t,\cdot) - \widehat X^i(s;t,\cdot)}_{\L\infty(\reali;\reali)} \d{s}
  \\
  &  +
    Qe^{Qt}
    \norma{\rho_o^i}_{\L1(\reali;\reali)}
    \int_0^t
    \norma{\mathcal{E}^i(s;t,\cdot) - \widehat{\mathcal{E}^i}(s; t,\cdot) }_{\L\infty(\reali; \reali)} \d{s}
    +
    Qt e^{2Qt} \norma{\rho_o^i - \widehat\rho_o^i}_{\L1(\reali; \reali)}
  \\
  \leq
  & \left(1
    +
    \mathcal{C}
    \left(
    \norma{\eta'}_{\W1\infty(\reali; \reali^k)},
    \norma{v_{NL}}_{\W2\infty(\reali; \reali^k)},
    \norma{\rho_o}_{\L1(\reali; \reali^k)},
    \norma{\widehat\rho_o}_{\L1(\reali; \reali^k)}, \mathcal{R}
    \right) t
    \right)
    \\
    & \quad \times
    \norma{(\rho_o^i)' - (\widehat \rho_o^i)'}_{\L1(\reali; \reali)}
  \\
  & + \mathcal{C}
    \left(
    \norma{\eta'}_{\W1\infty(\reali; \reali^k)},
    \norma{v_{NL}}_{\W2\infty(\reali; \reali^k)},
    \norma{\rho_o}_{\L1(\reali; \reali^k)},
    \norma{\widehat\rho_o}_{\L1(\reali; \reali^k)}, \mathcal{R}
    \right) t \,
    \norma{\rho_o^i - \widehat \rho_o^i}_{\L1(\reali; \reali)}
  \\
  &
    +
    \mathcal{C}
    \left(
    \norma{\eta'}_{\W2\infty(\reali; \reali^k)},
    \norma{v_{NL}}_{\W3\infty(\reali; \reali^k)},
    \norma{\rho_o}_{\cBV{1}(\reali; \reali^k)},
    \norma{\widehat\rho_o}_{\L1(\reali; \reali^k)},
    \mathcal{R},
    t
    \right)
  \\
  &
    \times
    \int_0^t
    \norma{v^i(\tau; \cdot) - \widehat v^i(\tau, \cdot)}_{ \W2\infty(\reali;\reali))} \d \tau
\end{align*}
   where in the last inequality we exploit~\eqref{eq:30},~\eqref{eq:49} and
   \begin{equation*}
     \norma{\partial^3_{xxx}\widehat v}_{\C0([0,t]; \L\infty(\reali;\reali^k))}
     \leq
     \mathcal{C} \left(
       \norma{\eta'}_{\W2\infty(\reali;\reali^k)},
       \norma{v_{NL}}_{\W3\infty(\reali;\reali^k)},
       \norma{\widehat \rho_o}_{\L1 (\reali;\reali^k)},  \mathcal{R}
     \right).
   \end{equation*}
Owing to~\eqref{eq:32}--\eqref{eq:34} and~\eqref{eq:50}, we get
\begin{align}
  \nonumber
  & \norma{\partial_x \rho (t) -
    \partial_x \widehat{\rho}(t)}_{\L1(\reali;\reali^k)}
  \\
  \nonumber
  \leq &
  \left(1
    +
    \mathcal{C}
    \left(
    \norma{\eta'}_{\W1\infty(\reali; \reali^k)},
    \norma{v_{NL}}_{\W2\infty(\reali; \reali^k)},
    \norma{\rho_o}_{\L1(\reali; \reali^k)},
    \norma{\widehat\rho_o}_{\L1(\reali; \reali^k)}, \mathcal{R}
    \right) t
    \right)
    \norma{\rho_o' - \widehat \rho_o'}_{\L1(\reali; \reali^k)}
    \\
  \nonumber
  & + \mathcal{C}
    \left(
    \norma{\eta}_{\W3\infty(\reali; \reali^k)},
    \norma{v_{NL}}_{\W3\infty(\reali; \reali^k)},
    \norma{\rho_o}_{\cBV{1}(\reali; \reali^k)},
    \norma{\widehat\rho_o}_{\cBV{}(\reali; \reali^k)}, \mathcal{R},t
    \right) t \,
    \\
  \label{eq:35}
    & \times
    \left(
    \norma{\rho_o - \widehat \rho_o}_{\L1(\reali; \reali^k)}
    +
    \int_0^t
    \norma{r(\tau)- \widehat r(\tau)}_{\L1(\reali; \reali)}
    \d \tau
    \right)
    \,.
\end{align}
Finally, \eqref{eq:56} and~\eqref{eq:57} follow from~\eqref{eq:50}
and~\eqref{eq:35}.
\end{proofof}

\subsection{The Local Problem}
\label{subsec:local-problem}

This paragraph is devoted to the Cauchy problem
\begin{equation}
  \label{eq:2}
  \left\{
    \begin{array}{l}
      \partial_t r + \partial_x \left(r \; w (t,x,r)\right) =0
      \\
      r (0,x) = r_o (x).
    \end{array}
  \right.
\end{equation}
We refer to~\cite[Definition~1]{MR267257} for the definition of
solution to~\eqref{eq:2}. Uniqueness and Lipschitz continuous
dependence on the initial data follow from~\cite[Theorem~1]{MR267257}.

\begin{lemma}[Dependence on the initial datum and uniqueness]
  \label{lem:0}
  Fix $W>0$. Assume \newcounter{ctmp}
  \begin{enumerate}[label=\bf (L\arabic*)]
  \item
    \label{item:5}
    $w \in \C1 ([0,T]\times \reali \times \reali; [0,W])$.
  \item \label{item:3new}
    $\partial^2_{tr} w, \partial^2_{xr}w \in \Lloc\infty ([0,T]\times
    \reali \times \reali; \reali)$.
    \setcounter{ctmp}{\value{enumi}}
  \end{enumerate}
  \noindent If $r,\widehat r \in \L\infty ([0,T]\times\reali; \reali)$
  are {Kru\v zkov} solutions to~\eqref{eq:2} with initial data
  $r_o,\widehat{r}_o \in \L\infty (\reali; \reali)$, then, for a.e.
  $t \in [0, T]$,
  \begin{equation}
    \label{eq:62}
    \begin{array}{c}
      \norma{r (t) - \widehat r (t)}_{\L1 ([a+\sigma \,t, b-\sigma \, t];\reali)}
      \leq
      \norma{r_o-\widehat{r_o}}_{\L1 ([a,b];\reali)}
      \qquad \mbox{ where }
      \\
      \sigma = \sup \left\{
      \modulo{w (\tau,\xi,q) + q \, \partial_r w (\tau,\xi,q)}
      \colon \tau \in [0,t] \,,\; \xi \in [a,b] \,,\; q \in [-M,M]
      \right\}
      \\
      M =
      \max\left\{\norma{r_o}_{\L\infty (\reali; \reali)},
      \norma{\widehat r_o}_{\L\infty (\reali; \reali)}\right\}
      \,.
    \end{array}
  \end{equation}
  If moreover
  $\norma{r_o - \widehat r_o}_{\L1 (\reali;\reali)} < +\infty$ and
  \begin{enumerate}[label=\bf (L\arabic*)]
    \setcounter{enumi}{\value{ctmp}}
  \item \label{item:3newnew}
    $\sup\left\{\modulo{\partial_r w (\tau,\xi,q)} \colon \tau \in
      [0,t] \,,\; \xi \in \reali \,,\; q \in [-M,M]\right\} <
    +\infty$.
  \end{enumerate}
  \setcounter{ctmp}{\value{enumi}} then
  \begin{equation}
    \label{eq:63}
    \norma{r (t) - \widehat r (t)}_{\L1 (\reali;\reali)}
    \leq
    \norma{r_o-\widehat{r}_o}_{\L1 (\reali;\reali)} \,.
  \end{equation}
  Hence, \eqref{eq:2} admits at most a unique {Kru\v zkov} solution.
\end{lemma}

\begin{proofof}{Lemma~\ref{lem:0}}
  We set $\phi (t,x,r) = r \, w(t,x,r)$ and, with reference to the
  conservation law $\partial_t r + \partial_x \phi (t,x,r) = 0$, we
  verify the assumptions in~\cite[\S~3]{MR267257}. Indeed:
  $\phi \in \C1 ([0,T] \times\reali\times\reali;\reali)$ due
  to~\ref{item:5}. Furthermore, the maps
  $r \mapsto \partial_t \phi (t,x,r)$ and
  $r \mapsto \partial_x \phi (t,x,r)$ are Lipschitz continuous on
  every compact subset of $[0,T] \times \reali \times \reali$, due
  to~\ref{item:3new}. Hence, \eqref{eq:62} follows
  from~\cite[Formula~(3.1)]{MR267257} while~\cite[Theorem~2]{MR267257}
  ensures the uniqueness of the solution. Finally, \eqref{eq:63} is
  obtained passing to the limits $a\to-\infty$ and $b \to+\infty$
  in~\eqref{eq:62}, thanks to~\ref{item:3newnew}.
\end{proofof}

\begin{corollary}[{Invariance of $[0, R_L]$ and mass conservation}]
  \label{cor:pos}
  Let~\ref{item:5} and~\ref{item:3new} hold. Fix $R_L>0$ and assume
  moreover that
  \begin{equation}
    \label{eq:86}
    w (t,x,R_L) = 0 \quad \mbox{ for all } (t,x) \in \reali_+ \times \reali \,.
  \end{equation}
  If $r_o \in \L1 (\reali; [0, R_L])$, then, for all $t\geq 0$, the
  {Kru\v zkov} solution $r$ to~\eqref{eq:2} satisfies
  \begin{equation}
    \label{eq:82}
    r (t,x) \in [0, R_L] \quad \mbox{ for a.e. } x \in \reali
    \qquad \mbox{ and } \qquad
    \norma{r (t)}_{\L1 (\reali; \reali)} = \norma{r_o}_{\L1 (\reali; \reali)} \,.
  \end{equation}
\end{corollary}

\begin{proofof}{Corollary~\ref{cor:pos}}
  Clearly, $r_o \equiv 0$ yields the null solution
  by~\eqref{eq:86} while $\widehat r_o \equiv R_L$ yields the constant solution
  $\widehat r (t,x) = R_L$. The comparison
  principle~\cite[Theorem~3]{MR267257} ensures the invariance
  $r(t,x) \in [0, R_L]$ for all $(t,x) \in \reali_+ \times
  \reali$. Then, the conservation of mass yields the latter equality.
\end{proofof}

We now consider the existence of solutions to~\eqref{eq:2}.
\begin{lemma}[Existence]
  \label{lem:1}
  Fix $W > 0$. Let~\ref{item:5} and~\ref{item:3new} hold and assume
  also that
  \begin{enumerate}[label=\bf (L\arabic*)]
    \setcounter{enumi}{\value{ctmp}}
  \item \label{item:2}
    $\partial_{xx}^2 w, \partial^2_{xr}w \in \C0\left([0,T]\times \reali
    \times \reali; \reali\right)$
  \item \label{item:6}
    $\partial_x w, \partial_r w \in \L\infty
    \left([0,T]\times\reali\times\reali;\reali\right)$.
  \item \label{item:7} There exists $R\geq0$ such that for all
    $r \leq -R$, $w (t,x,r) = W$ and for all $r \geq R$,
    $w (t,x,r) = 0$.
  \item \label{item:8}
    $\partial^2_{xr} w \in \L\infty
    \left([0,T]\times\reali\times\reali;\reali\right)$.
  \end{enumerate}
  \setcounter{ctmp}{\value{enumi}}
  \noindent Then, for any $r_o \in \L\infty(\reali;\reali)$,
  problem~\eqref{eq:2} admits a {Kru\v zkov} solution.
\end{lemma}

\begin{proofof}{Lemma~\ref{lem:1}}
  Using the notation in~\cite{MR267257} we set
  $\phi (t,x,r) = r \, w(t,x,r)$ and we check that the assumptions
  in~\cite[$4^\circ$ of \S~5]{MR267257} are verified. Indeed: $\phi$
  and $\partial_r \phi$ are continuous by~\ref{item:5};
  $\partial^2_{xr}\phi$ is continuous by~\ref{item:2};
  $\partial^2_{xx} \phi$ is continuous by~\ref{item:2};
  $\partial_x \phi$ and $\partial_r \phi$ are bounded in
  $[0,T] \times \reali \times [-M,M]$ for any $M$ by~\ref{item:6};
  condition~\cite[(4.1)]{MR267257} directly follows;
  condition~\cite[(4.2)]{MR267257} holds
  by~\ref{item:6}--\ref{item:7}--\ref{item:8}. Then,
  \cite[Theorem~5]{MR267257} applies, completing the proof.
\end{proofof}

\begin{lemma}[$\BV{}$ estimate and stability]
  \label{lem:2}
  Fix $r_o \in \BV \left(\reali; [0,R]\right)$. Consider the problems
  \begin{equation}
    \label{eq:3}
    \left\{
      \begin{array}{l}
        \partial_t r + \partial_x \left(r \; w (t,x,r)\right) =0
        \\
        r (0,x) = r_o (x)
      \end{array}
    \right.
    \qquad \mbox{ and } \qquad
    \left\{
      \begin{array}{l}
        \partial_t \widehat r
        +
        \partial_x \left(\widehat r \; \widehat w (t,x,\widehat r)\right) =0
        \\
        \widehat r (0,x) = r_o (x)
      \end{array}
    \right.
  \end{equation}
  with both $w,\widehat w$ satisfying~\ref{item:5}, \ref{item:6},
  \ref{item:7} and~\ref{item:8}.  Moreover, $w$ also satisfies
  \begin{enumerate}[label=\bf(L\arabic*)]
    \setcounter{enumi}{\value{ctmp}}
  \item \label{item:11} For all $q \in [-R,R]$,
    $\int_0^T \int_{\reali} \modulo{\partial^2_{xx} w (t,x,q)} \d{x}
    \d{t} < +\infty$.
  \item \label{item:10}
    $\partial_t w, \partial^2_{t r} w, \partial^2_{t x} w \in
    \L\infty([0,T]\times\reali\times\reali; \reali)$.
  \item \label{item:12}
    $\partial_{xx}^2 w, \partial_{xr}^2 w \in \C0([0,T]\times \reali
    \times \reali;\reali)$.
  \end{enumerate}
  \setcounter{ctmp}{\value{enumi}} Call $r$ the solution
  to~\eqref{eq:3}, left. Then, for all $t \in [0,T]$,
  $r(t) \in \BV{}(\reali; \reali)$ and
  \begin{align}\label{eq:tvest}
    \tv\left(r(t)\right)
    \leq
    & \tv(r_o) e^{\kappa_o t} + R \int_0^t e^{\kappa_o (t-\tau)}
      \int_\reali
      \norma{\partial_{xx}^2 w(\tau, x, \cdot)}_{\L\infty(\reali; \reali)}
      \d{x} \d{\tau}
  \end{align}
  where
  \begin{equation*}
    \kappa_o \leq
    3
    \left(
      \norma{\partial_x w}_{\L\infty([0,T]\times \reali\times\reali;\reali)}
      +
      R \,
      \norma{\partial_{xr}^2 w}_{\L\infty([0,T]\times\reali\times\reali;\reali)}
    \right)
  \end{equation*}
  Additionally, calling $\widehat{r}$ the solution to~\eqref{eq:3},
  right, if $w, \widehat w$ satisfy
  \begin{enumerate}[label=\bf(L\arabic*)]
    \setcounter{enumi}{\value{ctmp}}
  \item \label{item:9}
    $\sup_{t \in [0,T]} \norma{\partial_x w (t, \cdot, \cdot
      )}_{\L1(\reali;\L\infty( \reali;\reali))} +\sup_{t \in [0,T]} \norma{\partial_x \widehat w (t, \cdot, \cdot
      )}_{\L1(\reali;\L\infty( \reali;\reali))} < +\infty$.
  \end{enumerate}%
  \setcounter{ctmp}{\value{enumi}}%
  \begin{equation}
    \label{eq:4}
    \!\!\!
    \norma{r (t) {-} \widehat r (t)}_{\L1 (\reali; \reali)}
    {\leq} t  e^{\kappa_o t}
    \mathcal{C}\left (\tv(r_o), \int_0^t \!\! \int_\reali
      \norma{\partial^2_{xx} w(\tau, x, \cdot)}
      _{\L\infty(\reali; \reali)} \d x \d \tau,R\right)
    \norma{w {-} \widehat w}_{\mathcal{W}}.
  \end{equation}
\end{lemma}

In~\eqref{eq:4}, we use the set
\begin{equation}
  \label{eq:21}
  \mathcal{W}
  \coloneqq
  \left\{
    w\colon [0,T]\times \reali \times \reali \rightarrow \reali
    \mbox{ satisfying~\ref{item:5}, \ref{item:6}, \ref{item:7}, \ref{item:9}}
  \right\}
\end{equation}
with norm
\begin{equation}\label{eq:25}
  \!\!\!
  \norma{w}_{\mathcal{W}}
  {\coloneqq}
  \norma{w}_{\L\infty([0,T]\times \reali \times \reali;\reali)}
  {+}
  R \,
  \norma{\partial_r w}_{\L\infty([0,T]\times \reali \times \reali;\reali)}
  {+}
  \sup_{t \in [0,T]}
  \norma{\partial_x w (t, \cdot, \cdot )}_{\L1(\reali;\L\infty( \reali;\reali))}
\end{equation}
where
$\norma{\partial_x w (t, \cdot,
  \cdot)}_{\L1(\reali;\L\infty(\reali;\reali))} = \int_{\reali}
\esssup_{q \in \reali} \modulo{\partial_x w (t, x, q)} \d{x}$.

\begin{proof}
  Set $f (t,x,r) = r \, w (t,x,r)$ and
  $g (t,x,\widehat r) = \widehat r \, \widehat w (t,x,\widehat r)$.
  We aim at applying~\cite[Theorem~2.6]{MR2512832}. To this aim, we
  specialize to the present case the key assumptions~\textbf{(H1)},
  \textbf{(H2)} and~\textbf{(H3)} therein, as relaxed in the remark
  before~\cite[Theorem~2.3]{MR2512832}.
  \begin{description}
  \item[(H1)] $f,g \in \C1 \left([0,T]\times\reali\times\reali;\reali\right)$;\\
    $\partial^2_{xr}f , \, \partial^2_{xx}f, \, \partial^2_{x\widehat
      r}g , \, \partial^2_{xx}g \in
    \C0\left([0,T]\times\reali\times\reali;\reali\right)$;\\
    $\partial_x f, \, \partial_r f, \, \partial^2_{xr} f, \,
    \partial_x g, \, \partial_{\widehat r} g, \,
    \partial^2_{x\widehat{r}} g \in
    \L\infty\left([0,T]\times\reali\times\reali;\reali\right)$.

  \item[(H2)] $f \in \C1 \left([0,T]\times\reali\times\reali;\reali\right)$;\\
      $\partial^2_{xr}f , \, \partial^2_{tx}f, \, \partial^2_{tr}f
      \in \L\infty\left([0,T]\times\reali\times\reali;\reali\right)$;\\
      $\int_0^T \int_{\reali} \norma{\partial^2_{xx}f
        (t,x,\cdot)}_{\L\infty (\reali;\reali)} \d{x}\d{t} < +\infty$.
\item[(H3)] $f-g \in \C1 \left([0,T]\times\reali\times\reali;\reali\right)$;\\
      $\partial_r(f-g) \in \L\infty
      \left([0,T]\times\reali\times\reali;\reali\right)$;\\
      $\int_0^T \int_{\reali} \norma{\partial_x f (t,x,\cdot) -
        \partial_x g (t,x,\cdot)}_{\L\infty (\reali;\reali)}
      \d{x}\d{t} < +\infty$.
    \end{description}

    \noindent Requirement~\textbf{(H1)} holds: $f$ and $g$ are of
    class $\C1$ by~\ref{item:5}; the existence and continuity of the
    second derivatives follow from~\ref{item:5} and~\ref{item:12};
    $\partial_rf$ and $\partial_{\widehat r} g$ are in
    $\L\infty \left([0,T]\times \reali \times \reali; \reali\right)$
    by~\ref{item:5}, \ref{item:6} and~\ref{item:7};
    $\partial^2_{xr} f$ and $\partial^2_{x \widehat r} g$ are in
    $\L\infty \left([0,T]\times\reali\times\reali; \reali\right)$ due
    to~\ref{item:6}, \ref{item:7} and~\ref{item:8}; $\partial_xf$ and
    $\partial_x g$ are in
    $\L\infty \left([0,T]\times\reali\times\reali; \reali\right)$ due
    to~\ref{item:6} and~\ref{item:7}.

    Requirement~\textbf{(H2)} holds: $\partial^2_{tr} f$ and
    $\partial^2_{t x} f$ are of class $\L\infty$ by~\ref{item:7}
    and~\ref{item:10}; finally
    $\int_0^T \int_{\reali} \norma{\partial^2_{xx} f
      (t,x,\cdot)}_{\L\infty (\reali; \reali)} \d{x} \d{t} \leq R
    \int_0^T \int_{\reali} \norma{\partial^2_{xx} w
      (t,x,\cdot)}_{\L\infty (\reali; \reali)} \d{x} \d{t} < +\infty$
    by~\ref{item:7}, \ref{item:11}.

    Requirement~\textbf{(H3)} holds:
    $f-g \in \C1 \left([0,T]\times\reali\times\reali; \reali\right)$
    by~\ref{item:5} while $\partial_r(f-g)$ is in
    $\L\infty\left([0,T]\times\reali\times\reali; \reali\right)$ due
    to~\ref{item:5}, \ref{item:6}, \ref{item:7};
    $\int_0^T \int_{\reali} \norma{\partial_x f (t,x,\cdot) -
      \partial_x g (t,x,\cdot)}_{\L\infty (\reali; \reali)} \d{x}
    \d{t}$ $< +\infty$ by~\ref{item:7} and~\ref{item:9}.  With reference
    to the notation in~\cite[Formula~(2.4)]{MR2512832}, introduce
    \begin{equation*}
      \kappa_o
      \coloneqq
      3
      \norma{\partial_{xr}^2 f}_{\L\infty([0,T] \times \reali \times\reali; \reali)}
      \leq
      3\left(\norma{\partial_x w}_{\L\infty ([0,T]\times\reali\times\reali;
          \reali)}
        + R \,
        \norma{\partial^2_{xr} w}_{\L\infty
          ([0,T]\times\reali\times\reali; \reali)}\right).
    \end{equation*}
    Since $f$ satisfies~\textbf{(H1)} and~\textbf{(H2)}, we
    apply~\cite[Theorem 2.5]{MR2512832}, which reads for all
    $t \in [0,T]$
    \begin{align*}
      \tv\left(r(t)\right)
      \leq
      & \tv(r_o) e^{\kappa_o t} +
        \int_0^t e^{\kappa_o (t-\tau)}\int_\reali \norma{\partial_{xx}^2 f(\tau, x, \cdot)}_{\L\infty(\reali; \reali)} \d{x} \d{\tau}
      \\
      \leq
      & \tv(r_o) e^{\kappa_o t} + R \int_0^t e^{\kappa_o (t-\tau)}\int_\reali
        \norma{\partial_{xx}^2 w(\tau, x, \cdot)}_{\L\infty(\reali; \reali)}
        \d{x} \d{\tau},
    \end{align*}
    proving~\eqref{eq:tvest}.  Recall for later use the estimate
    \begin{equation}
      \label{eq:6}
      e^a -1 \leq a \,e^a
      \qquad \forall\, a \geq 0 \,.
    \end{equation}
    With reference to the notation in~\cite[Theorem~2.6]{MR2512832},
    we have
    \begin{flalign*}
      \kappa
      & = 2 \norma{\partial^2_{xr} f}_{\L\infty
               ([0,T]\times\reali\times\reali; \reali)}
      \\
    \dfrac{e^{\kappa_o t} - e^{\kappa t}}{\kappa_o - \kappa} & =
    \dfrac{e^{\norma{\partial^2_{xr}
          f}_{\L\infty([0,T]\times\reali\times\reali;
          \reali)}t}-1}{\norma{\partial^2_{xr}
        f}_{\L\infty([0,T]\times\reali\times\reali; \reali)}} \;
    e^{2\norma{\partial^2_{xr}
        f}_{\L\infty([0,T]\times\reali\times\reali; \reali)}t}
    \\
    & \leq e^{ 3 \norma{\partial^2_{xr}
        f}_{\L\infty([0,T]\times\reali\times\reali; \reali)}t } t &
    [\mbox{By~\eqref{eq:6}}]
    \\
    & = e^{\kappa_o \,t} t.
  \end{flalign*}
We can now apply~\cite[Theorem~2.6]{MR2512832} and obtain the estimate
\begin{eqnarray*}
  &
  & \norma{r(t) - \widehat r(t)}_{\L1(\reali;\reali)}
  \\
  & \leq
  & \dfrac{e^{\kappa_o t} - e^{\kappa t}}{\kappa_o - \kappa} \tv(r_o)
    \norma{\partial_r(f-g)}_{\L\infty([0,T]\times \reali\times\reali;\reali)}
  \\
  &
  & + \left( \int_0^t
    \dfrac{e^{\kappa_o (t-\tau)} - e^{\kappa (t-\tau)}}{\kappa_o - \kappa}
    \int_\reali \norma{\partial^2_{xx}f(\tau,x, \cdot)}_{\L\infty(\reali;\reali)}
    \, \d x \d \tau \right)
    \norma{\partial_r(f-g)}_{\L\infty([0,T]\times \reali\times\reali;\reali)}
  \\
  &
  & + \int_0^t e^{\kappa (t-\tau)} \int_\reali
    \norma{\partial_x(f-g)(\tau,x,\cdot)}_{\reali;\reali)} \d x \d \tau
  \\
  & \leq
  & t \, e^{\kappa_o t} \tv(r_o) \left(
    \norma{w - \widehat w}_{\L\infty([0,T]\times\reali\times\reali; \reali)}
    + R
    \norma{\partial_rw - \partial_r\widehat w}_{\L\infty([0,T]\times\reali\times\reali; \reali)} \right)
  \\
  &
  & + R \, t \, e^{\kappa_o t} \left(
    \int_0^t \int_\reali
    \norma{\partial^2_{xx} w(\tau, x, \cdot)}_{\L\infty(\reali; \reali)}
    \d x \d \tau
    \right)
  \\
  &
  & \quad \times
    \left( \norma{w - \widehat w}_{\L\infty([0,T]\times\reali\times\reali; \reali)}
    + R
    \norma{\partial_rw - \partial_r\widehat w}_{\L\infty([0,T]\times\reali\times\reali; \reali)} \right)
  \\
  &
  & + R \, e^{\kappa t} \int_0^t
    \int_\reali
    \norma{\partial_x w(\tau, x, \cdot)
    - \partial_x \widehat w(\tau, x, \cdot)}_{\L\infty(\reali; \reali)}
    \d x \d \tau
  \\
  \ghost{
  & \leq
  & t \, e^{\kappa_o t} \, \tv(r_o) \left(
    \norma{w - \widehat w}_{\L\infty([0,T]\times\reali\times\reali; \reali)}
    + R
    \norma{\partial_rw - \partial_r\widehat w}_{\L\infty([0,T]\times\reali\times\reali; \reali)} \right)
  \\
  &
  & + R \, t \, e^{\kappa_o t} \left(
    \int_0^t \int_\reali
    \norma{\partial^2_{xx} w(\tau, x, \cdot)}_{\L\infty(\reali; \reali)}
    \d x \d \tau
    \right)
  \\
  &
  & \quad \times
    \left( \norma{w - \widehat w}_{\L\infty([0,T]\times\reali\times\reali; \reali)}
    +
    R \,
    \norma{\partial_rw - \partial_r\widehat w}_{\L\infty([0,T]\times\reali\times\reali; \reali)} \right)
  \\
  &
  & + R  \, e^{\kappa t}
    \int_0^t
    \norma{ \partial_x w(\tau, \cdot, \cdot)
    - \partial_x \widehat w(\tau,  \cdot, \cdot)}_{\L1(\reali; \L\infty(\reali; \reali)))} \d \tau
  \\
  }
  & \leq
  & t \, e^{\kappa_o t} \left(   \tv(r_o)
    + R \left(\int_0^t
    \int_\reali
    \norma{\partial^2_{xx} w(\tau, x, \cdot)}_{\L\infty(\reali; \reali)}
    \d x \d \tau\right) \right)
    \norma{w - \widehat w}_{\L\infty([0,T]\times\reali\times\reali; \reali)}
  \\
  &
  & + R \, t \, e^{\kappa_o t}  \left(  \tv(r_o)
    +  R \left(\int_0^t
    \int_\reali
    \norma{\partial^2_{xx} w(\tau, x, \cdot)}_{\L\infty(\reali; \reali)}
    \d x \d \tau\right) \right)
  \\
  && \quad \times
     \norma{\partial_r w - \partial_r\widehat w}_{\L\infty([0,T]\times\reali\times\reali; \reali)}
  \\
  &
  & + R t \, e^{\kappa t}
    \sup_{t \in [0,T]}
    \norma{ \partial_x w(\tau,  \cdot, \cdot)
    - \partial_x \widehat w(\tau,  \cdot, \cdot)}_{\L1(\reali; \L\infty(\reali; \reali)))}
  \\
  & \leq
  & t \, e^{\kappa_o t}  \left(  \tv(r_o) +  R \left(\int_0^t
    \int_\reali
    \norma{\partial^2_{xx} w(\tau, x, \cdot)}_{\L\infty(\reali; \reali)}
    \d x \d \tau\right) \right)
  \\
  &
  & \times \left( \norma{w - \widehat w}_{\L\infty([0,T]\times\reali\times\reali; \reali)}
    + R \,
    \norma{\partial_r w - \partial_r\widehat w}_{\L\infty([0,T]\times\reali\times\reali; \reali)}
    \right)
  \\
  &
  & + R t \, e^{\kappa t}
    \sup_{t \in [0,T]}
    \norma{ \partial_x w(\tau,  \cdot, \cdot) - \partial_x \widehat w(\tau,  \cdot, \cdot)}_{\L1(\reali; \L\infty(\reali; \reali)))}
  \\
  & \leq
  & t \, e^{\kappa_o t} \max \left\{
    \left(
    \tv(r_o)
    +
    R \left(\int_0^t
    \int_\reali
    \norma{\partial^2_{xx} w(\tau, x, \cdot)}_{\L\infty(\reali; \reali)}
    \d x \d \tau\right) \right), R
    \right\}
    \norma{w - \widehat w}_{\mathcal{W}} \,.
\end{eqnarray*}
The proof is completed.
\end{proof}

Remark that a straightforward consequence of \Cref{lem:lei}, \Cref{lem:2} and~\eqref{eq:tvest} is that for all $t \in [0,T]$
\begin{equation}
  \label{eq:51}
    \norma{r (t)}_{\L\infty (\reali;\reali)}
    +
    \tv\left(r (t)\right)
    \leq
    \mathcal{C}\left(
      \begin{array}{c}
        \norma{\partial_x w}_{\L\infty([0,T]\times \reali\times\reali;\reali)}
        ,
        \norma{\partial_{xr}^2 w}_{\L\infty([0,T]\times\reali\times\reali;\reali)}
        ,
        \\
        \int_0^T \int_\reali
          \norma{\partial_{xx}^2 w(\tau, x, \cdot)}_{\L\infty(\reali; \reali)}
          \d{x} \d{\tau}
          ,
          R
          ,
          \tv (r_o)
          ,
          T
      \end{array}
    \right)
\end{equation}
where $R$ is as in~\ref{item:7}.

\begin{lemma}[$\L1$ continuity in time]
  \label{lem:continuity}
  Assume~\ref{item:5}--\ref{item:9}.  Then, the
  solution $r$ built in Lemma~\ref{lem:0} and \Cref{lem:1} belongs to
  $\C0 \left([0,T]; \L1 (\reali; \reali)\right)$. Actually,
  $r$ is locally $\L1$ Lipschitz continuous in time, i.e., calling $K$
  the constant in the right hand side of~\eqref{eq:51}, for all
  $t_1,t_2 \in [0,T]$
  \begin{eqnarray}
    \label{eq:19}
    \norma{r (t_2) - r (t_1)}_{\L1 (\reali;\reali)}
    & \leq
    & K
      \left(
      \sup_{t \in [0,T]}
      \norma{\partial_x w (t, \cdot, \cdot)}_{\L1 (\reali; \L\infty (\reali;\reali))}
      \right)       \modulo{t_2-t_1}
    \\
    \nonumber
    &
    & + K \left(
      + K\,\norma{\partial_r w}_{\L\infty ([0,T]\times\reali\times\reali; \reali)}
      + \norma{w}_{\L\infty ([0,T]\times\reali\times\reali;\reali)}
      \right)
      \modulo{t_2-t_1}\,.
  \end{eqnarray}
\end{lemma}

Under conditions~\ref{item:6}--\ref{item:10} and~\ref{item:9}, the
{Kru\v zkov} solution to~\eqref{eq:2} belongs to the space
$\C0\left([0,T];\Lloc{1}(\reali;\reali)\right)$
by~\cite[Remark~2.4]{MR2512832}. We here prove that we also have the  $\L1$ Lipschitz continuity, locally in time.

\begin{proofof}{Lemma~\ref{lem:continuity}}
  We follow the lines of~\cite[Theorem~4.3.1]{MR3468916}. Since $r$ is
  a distributional solution to~\eqref{eq:2}, for any
  $t_1,t_2 \in [0,T]$ with $t_1 < t_2$,
  $\phi \in \C1\left([0,T];\reali\right)$ with $\phi (t) = 1$ for all
  $t \in [0,t_2]$, $\psi \in \Cc1 (\reali;\reali)$ with
  $\norma{\psi}_{\L\infty (\reali;\reali)} \leq 1$, we have for
  $i=1,2$
  \begin{eqnarray*}
    \int_{t_i}^T \int_{\reali} r (t,x) \, \partial_t \phi (t) \, \psi (x) \d{x} \d{t}
    +
    \int_{t_i}^T \int_{\reali}
    r (t,x) \, w\left(t,x,r (t,x)\right) \, \phi (t) \, \partial_x \psi (x)
    \d{x} \d{t}
    \\
    +
    \int_{\reali} r (t_i,x) \, \phi (t_i)\, \psi (x)\d{x}
    & =
    & 0 \,.
  \end{eqnarray*}
  Subtracting these expressions and using~\ref{item:9}, we get
  \begin{flalign*}
    & \int_{\reali} \left(r (t_2,x) - r (t_1,x)\right) \psi (x) \d{x}
    \\
    =
    & \int_{t_1}^{t_2} \int_{\reali} r (t,x) \, w\left(t,x,r (t,x)\right)
      \partial_x \psi (x) \d{x} \d{t}
    \\
    \leq
    &   \int_{t_1}^{t_2}
      \tv\left(r (t,\cdot) \, w\left(t, \cdot, r (t, \cdot)\right)\right)
      \d{t}
    & \mbox{\cite[Def.~3.4]{MR1857292}}
    \\
    \leq
    & \int_{t_1}^{t_2}
      \left(
      \norma{r (t)}_{\L\infty (\reali;\reali)}
      \tv\left(w \left(t,\cdot,r (t,\cdot)\right)\right)
      +
      \tv\left(r (t)\right)
      \norma{w}_{\L\infty ([0,T]\times\reali\times\reali;\reali)}
      \right) \d{t}
    & \mbox{\cite[Ex.~3.17]{MR1857292}}
    \\
    \leq
    & \int_{t_1}^{t_2}
      K \left(
      \norma{\partial_x w (t, \cdot, \cdot)}_{\L1 (\reali; \L\infty (\reali;\reali))}
      +
      \norma{\partial_r w}_{\L\infty ([0,T]\times\reali\times\reali,\reali)} \,K
      \right) \d{t}
    & [\mbox{By~\eqref{eq:51}}]
    \\
    & + \int_{t_1}^{t_2} K\, \norma{w}_{\L\infty ([0,T]\times\reali\times\reali;\reali)}
      \d{t} \,.
    & [\mbox{By~\eqref{eq:51}}]
  \end{flalign*}
  Passing to the supremum over all $\psi \in \Cc1 (\reali; \reali)$
  with $\modulo{\psi (x)} \leq 1$ the
  proof is completed.
\end{proofof}

\begin{lemma}
  \label{lem:6} Assume~\ref{item:21}.
  Suppose
  \begin{enumerate}[label=\bf(V\arabic*)]
  \item \label{item:14} $\rho \in \C1([0,T] \times \reali ; \reali^k)$.
  \item \label{item:40}
    $\rho \in \C0 \left([0,T]; \W11(\reali; \reali^k)\right)$.
  \item \label{item:19}
    $\partial_{xx}^2 \rho \in (\C0 \cap \L1) \left([0,T] \times \reali;\reali^k\right)$.
  \item \label{item:16}
    $\partial_t \rho, \partial_{tx}^2 \rho \in
    \L\infty\left([0,T] \times \reali;\reali^k\right)$.
  \end{enumerate}
  Then, the map $w(t,x,r) \coloneqq v_L\left(\Sigma \rho(t,x) + r\right)$
  satisfies all the requirements~\ref{item:5} to~\ref{item:9} where the constant $R$ in~\ref{item:7} can be chosen bigger than $R_L + \norma{\rho}_{\L\infty([0,T] \times \reali;\reali^k)}$.
\end{lemma}

\begin{proof}
  We group the different requirements as follows.
  \begin{description}
  \item[\ref{item:5} and~\ref{item:3newnew}:] Compute
    \begin{displaymath}
      \begin{array}{rcl}
        \partial_t w(t,x,r)
        & =
        & v_L'\left(\Sigma\rho(t,x)+r\right)\; \Sigma\partial_t\rho(t,x)
        \\
        \partial_x w(t,x,r)
        & =
        & v_L'\left(\Sigma\rho(t,x)+r\right)\; \Sigma\partial_x\rho(t,x)
      \end{array}
      \quad \mbox{ and } \quad
      \partial_r w(t,x,r)
      =
      v_L'\left(\Sigma\rho(t,x)+r\right) \,,
    \end{displaymath}
    so that~\ref{item:21} and~\ref{item:14} imply~\ref{item:5},
    while~\ref{item:21} implies~\ref{item:3newnew}.

  \item[\ref{item:3new}:] Direct computations show that
    $\partial^2_{tr} w$ and $\partial^2_{xr} w$ involve only first and
    second derivatives of $v_L$ and first derivatives of
    $\rho$. Hence, \ref{item:21} and~\ref{item:14} ensure that
    $\partial^2_{tr} w$ and $\partial^2_{xr} w$ are locally bounded.

  \item[\ref{item:2}:] Assumptions~\ref{item:21}, \ref{item:14}
    and~\ref{item:19} clearly imply~\ref{item:2}.

  \item[\ref{item:6}, \ref{item:8}, and~\ref{item:10}:] {Note that as
      $\partial_x \rho \in \L1\left([0,T]; \W11(\reali;
        \reali^k)\right)$ by~\ref{item:40}
      and~\ref{item:19},~\ref{item:16} implies that
      $\partial_x\rho \in \L\infty([0,T] \times \reali; \reali^k)$.
      Then, assumptions~\ref{item:21} and~\ref{item:16} imply that
      $\partial_t w$, $\partial_x w$, $\partial_r w$,
      $\partial^2_{x r} w$, $\partial^2_{t r} w$, and
      $\partial^2_{tx} w$ belong to
      $\L\infty\left([0, T] \times \reali \times \reali;
        \reali\right)$.}

  \item[\ref{item:7}:] { By~\ref{item:40} define
      $M = \norma{\rho}_{\L\infty\left([0, T] \times \reali; \reali^k
        \right)}$.  If $r \ge R_L + M$, then
      $r + \Sigma\rho(t,x) \ge r - M \ge R_L$ and so,
      by~\ref{item:21},
      $w(t, x, r) = v_L\left(\Sigma\rho(t, x) + r\right) = 0$.  If
      $r \le - M$, then $r + \Sigma\rho(t,x) \le r + M \le 0$ and so,
      by~\ref{item:21},
      $w(t, x, r) = v_L\left(\Sigma\rho(t, x) + r\right) = W$.
      Then~\ref{item:7} holds with $R = R_L + M$.}
  \item[\ref{item:11}:] Assumptions~\ref{item:21} and~\ref{item:19} allow to compute
    \begin{displaymath}
      \partial^2_{xx}w(t,x,r)
      = v_L''\left(\Sigma\rho(t,x)+r\right) \,
      \left(\Sigma\partial_x \rho(t,x)\right)^2
      +
      v_L'\left(\Sigma\rho(t,x)+r\right) \,
      \Sigma\partial_{xx}^2 \rho(t,x) \,.
    \end{displaymath}
  Moreover, for every
    $r \in \reali$, we have
    \begin{align}\label{eq:22}
      & \int_0^T \int_\reali \modulo{\partial^2_{xx} w(t, x, r)} \d x \d t
      \\
      \nonumber
      \leq
      & \norma{v_L''}_{\L\infty(\reali;\reali)}
        \int_0^T \int_\reali
        \modulo{ \Sigma \partial_x \rho(t, x)}^2 \d x \d t
        + \norma{v_L'}_{\L\infty(\reali;\reali)} \int_0^T \int_\reali
        \modulo{\Sigma \partial^2_{xx} \rho(t, x)} \d x \d t
      \\
      \nonumber
      \leq
      & \norma{v_L''}_{\L\infty(\reali;\reali)}
        \norma{\partial_x\rho}_{\L\infty ([0,T]\times\reali;\reali^k)}
        \norma{\partial_x\rho}_{\L1 ([0,T]\times\reali;\reali^k)}
        + \norma{v_L'}_{\L\infty(\reali;\reali)}
        \norma{\partial^2_{xx}\rho}_{\L1 ([0,T]\times\reali;\reali^k)} \,,
    \end{align}
    proving~\ref{item:11}.

  \item[\ref{item:12}:] Clearly, assumptions~\ref{item:21},
      \ref{item:14}, and~\ref{item:19} imply that the functions
      $\partial^2_{xx} w$ and $\partial^2_{x r} w$ belong to
      $\C0\left([0, T] \times \reali \times \reali; \reali\right)$,
      proving~\ref{item:12}.

    \item[\ref{item:9}:] Assumptions~\ref{item:21} and~\ref{item:40}
      imply that, for every $M>0$ and $t \in [0,T]$, it holds that
      \begin{eqnarray*}
        \sup_{t \in [0,T]}
        \norma{\partial_x w (t, \cdot, \cdot)}_{\L1 (\reali;\L\infty (\reali;\reali))}
        & =
          & \sup_{t \in [0,T]}
            \int_{\reali} \esssup_{q\in\reali}
            \modulo{v'_L\left(\Sigma\rho (t,x) + q\right) \;
            \Sigma \partial_x \rho (t,x)} \d{x}
          \\
        & \leq
        & \norma{v_L'}_{\L\infty(\reali;\reali)}
          \norma{ \partial_x \rho}_{\C0([0,T];\L1 (\reali;\reali^k))}
          < +\infty\,,
      \end{eqnarray*}
    proving~\ref{item:9}.
  \end{description}
  \vspace{-\baselineskip}
\end{proof}

\begin{lemma}\label{distloc}
Let~\ref{item:21} hold. Moreover, suppose $r_o , \widehat r_o \in \cBV{}(\reali; \reali)$ and $\rho, \widehat \rho$ satisfy~\ref{item:14}--\ref{item:16}.
Call $r,\widehat r$ the Kru\v zkov solutions to
   \begin{equation}
   \label{eq:43}
     \left\{
       \begin{array}{@{}l}
         \partial_t r
         + \partial_x
         \left(r \; v_{L} \left(
         \Sigma \rho(t,x) + r\right)\right) = 0
         \\
         r (0,x) = r_o (x)
       \end{array}
     \right.
     \mbox{ and }
     \left\{
       \begin{array}{@{}l@{}}
         \partial_t \widehat r
         + \partial_x
         \left(\widehat r \; v_{L} \left(
         \Sigma\widehat \rho(t,x) + \widehat r\right)\right)=0
         \\
         \widehat r (0,x) =\widehat r_o (x).
       \end{array}
     \right.
\end{equation}
   Then, for all $t \in [0,T]$
   \begin{align*}
     &
       \norma{r(t) - \widehat r(t)}_{\L1(\reali; \reali)}
     \\
     \leq
     &  \norma{r_o - \widehat r_o}_{\L1(\reali; \reali)}
     \\
     &
       +
       \mathcal{C} \left(
       \begin{array}{@{}c@{}}
         R_L,
         \norma{v'_L}_{\W1\infty(\reali; \reali)},
         \norma{\rho}_{\C0([0,t]; \W1\infty(\reali; \reali^k)},
         \\
         \norma{\widehat \rho}_{\C0([0,t]; \W1\infty(\reali; \reali^k)},
         \norma{\partial_x \rho}_{\L1([0,t]; \W11(\reali; \reali^k))},
         \tv(\widehat r_o), t
       \end{array}
       \right)
       t \, \norma{\rho - \widehat \rho}_{\C0([0,t]; \W11(\reali; \reali^k ))}.
   \end{align*}
\end{lemma}

\begin{proof}
  Call $w(t,x,q) \coloneq v_L\left(\Sigma \rho (t,x) + q \right)$,
  $\widehat w(t,x,q) \coloneq v_L\left(\Sigma \widehat\rho(t,x) + q \right)$
  and note that, since~\ref{item:21} holds and $\rho, \widehat \rho$
  satisfy~\ref{item:14}--\ref{item:16}, by \Cref{lem:6},
  $w, \widehat w$ satisfy~\ref{item:5}--\ref{item:9} and
  $w, \widehat w \in \mathcal{W}$ where $\mathcal{W}$ is defined
  in~\eqref{eq:21}, with
\begin{equation}
  R \coloneqq R_L
  + \max\left\{ \
    \norma{\rho}_{\L\infty([0,T]\times \reali; \reali^k)},
    \norma{\widehat \rho}_{\L\infty([0,T]\times \reali; \reali^k)}
  \right\}.
\end{equation}
Hence, we may apply \Cref{lem:0}, \Cref{lem:1} and \Cref{lem:2} to
problems~\eqref{eq:43}.  In particular,~\eqref{eq:63} and~\eqref{eq:4}
yields that for all $t \in [0,T]$
\begin{align*}
    \norma{r(t)- \widehat r(t)}_{\L1(\reali; \reali)} \leq & \norma{r_o - \widehat r_o}_{\L1(\reali; \reali)}
    \\
    &
    + t \, e^{\kappa_o t} \, \mathcal{C}\left (\tv(\widehat r_o), \int_0^t
      \int_\reali
      \norma{\partial^2_{xx} w(\tau, x, \cdot)}_{\L\infty(\reali; \reali)}
      \d x \d \tau,R\right) \norma{w - \widehat w}_{\mathcal{W}}
\end{align*}
where
\begin{eqnarray*}
  \kappa_o
  & \leq
  & 3
    \left(
    \norma{\partial_x w}_{\L\infty([0,T]\times \reali\times\reali;\reali)}
    +
    R \,
    \norma{\partial_{xr}^2 w}_{\L\infty([0,T]\times\reali\times\reali;\reali)}
    \right)
  \\
  & \le
  & 3 \left(
    \norma{v'_L}_{\L\infty(\reali;\reali)} +
    R \,
    \norma{v_L''}_{\L\infty(\reali;\reali)}
    \right)
    \norma{\partial_x \rho}_{\L\infty([0,T]\times \reali; \reali^k)}.
\end{eqnarray*}
Furthermore, owing to~\eqref{eq:22},
\begin{align}
  \nonumber
  & \int_0^t \int_\reali \norma{\partial_{xx}^2 w(\tau,x,\cdot)}
    _{\L\infty(\reali;\reali)} \d{x} \d{\tau}
  \\
  \leq
  & \norma{v_L''}_{\L\infty(\reali;\reali)}
    \norma{\partial_x {\rho}}_{\L\infty([0,t] \times \reali;\reali^k)}
    \norma{\partial_x{\rho}}_{\L1([0,t]\times \reali;\reali^k)}
    {+}
    \norma{v_L'}_{\L\infty(\reali;\reali)}
    \norma{\partial_{xx}^2 {\rho}}_{\L1([0,t] \times \reali;\reali^k)}.
\end{align}
We proceed now to estimate the term $\norma{w-\widehat w}_{\mathcal{W}}$
by evaluating separately each of the summands in~\eqref{eq:25}. It
holds that
\begin{eqnarray*}
  \norma{w-\widehat w}_{\L\infty([0,T]\times \reali\times\reali;\reali)}
  & =
  & \sup_{t \in [0,T], x \in \reali, q \in \reali}
    \modulo{v_L\left(\Sigma{\rho}(t,x) + q \right)
    - v_L\left(\Sigma{\widehat \rho}(t,x) + q\right)
    }
  \\
  &\leq
  & \norma{v_L'}_{\L\infty(\reali;\reali)} \,
    \norma{{\rho}-\widehat{\rho}}_{\L\infty([0,T]\times\reali;\reali^k)}
  \\
  &\leq
  &
    \norma{v_L'}_{\L\infty(\reali;\reali)} \,
    \norma{{\rho}-\widehat{\rho}}_{\C0([0,T]; \W11(\reali;\reali^k))}\,.
\end{eqnarray*}
  Similarly, direct
  computations lead to
  \begin{eqnarray*}
    \norma{\partial_r w - \partial_r \widehat w}_{\L\infty([0,T]\times \reali \times\reali; \reali)}
    & =
    & \sup_{t \in [0,T], x \in \reali, q \in \reali} \modulo{
      v_L'\left(\Sigma \rho(t,x) +q\right)
      - v_L' \left( \Sigma{\widehat \rho}(t,x)+q\right)
      }
    \\
    & \leq
    & \norma{v_L''}_{\L\infty(\reali;\reali)}
      \norma{{\rho}-\widehat{\rho}}_{\C0([0,T]; \W11(\reali;\reali^k))}.
  \end{eqnarray*}
  At last, we get for all $t \in [0,T]$
  \begin{eqnarray*}
    &
    & \norma{\partial_x w(t, \cdot, \cdot) - \partial_x \widehat w(t, \cdot, \cdot)}_{\L1(\reali;\L\infty(\reali;\reali))}
    \\
    & =
    & \int_\reali \sup_{ q \in \reali}
      \modulo{v_L'(\Sigma \rho(t,x) +q) \,
      \Sigma\partial_x \rho(t,x)
      - v_L'(\Sigma{\widehat \rho}(t,x) + q) \,
      \Sigma\partial_x {\widehat \rho}(t,x)} \, \d{x}
    \\
    & \leq
    & \int_\reali \sup_{ q \in \reali} \modulo{
      v_L'( \Sigma{\rho}(t,x)+q)
      \Sigma
      \left( \partial_x {\rho}(t,x) - \partial_x {\widehat \rho}(t,x)
      \right)
      } \, \d{x}
    \\
    &
    & +
      \int_\reali \sup_{ q \in \reali} \modulo{
      \left(
      v_L'(\Sigma{\rho}(t,x) +q)
      - v_L'(\Sigma{\widehat \rho}(t,x)+q)
      \right)
      \Sigma\partial_x {\widehat \rho}(t,x)
      } \, \d{x}
    \\
    &\leq
    & \norma{v_L'}_{\L\infty(\reali;\reali)}
      \norma{\partial_x{\rho}(t) - \partial_x{\widehat \rho}(t) }_{\L1(\reali;\reali^k)}
    \\
    &
    & +
      \norma{v_L''}_{\L\infty(\reali;\reali)}
      \norma{\partial_x {\widehat \rho}(t)}_{\L\infty(\reali;\reali^k)} \,
      \norma{{\rho(t)} - {\widehat \rho(t)} }_{\L1(\reali;\reali^k)}.
  \end{eqnarray*}
Hence,
\begin{align*}
&
    \sup_{t \in [0,T]} \norma{\partial_x w(t, \cdot, \cdot) - \partial_x \widehat w(t, \cdot, \cdot)}_{\L1(\reali;\L\infty(\reali;\reali))}
    \\
    \leq
    &
    \left(
    \norma{v_L'}_{\L\infty(\reali;\reali)} + \norma{v_L''}_{\L\infty(\reali;\reali)}
      \norma{\partial_x {\widehat \rho}}_{\L\infty([0,T] \times \reali;\reali^k)}
    \right)
    \norma{\rho - \widehat \rho }_{\C0([0,T]; \W11(\reali;\reali^k))}
\end{align*}
and, for all $t >0$, collecting the previous computations we get
\begin{align*}
&
    \norma{r(t) - \widehat r(t)}_{\L1(\reali; \reali)}
    \\\leq &
    \norma{r_o - \widehat r_o}_{\L1(\reali; \reali)}
    \\
    &
    +
    \mathcal{C} \left(
      \begin{array}{@{}c@{}}
        R_L,
        \norma{v'_L}_{\W1\infty(\reali; \reali)},
        \norma{\rho}_{\C0([0,t]; \W1\infty(\reali; \reali^k)},
        \\
        \norma{\widehat \rho}_{\C0([0,t]; \W1\infty(\reali; \reali^k)},
        \norma{\partial_x \rho}_{\L1([0,t]; \W11(\reali; \reali^k))},
        \tv(\widehat r_o), t
      \end{array}
      \right)
    \norma{\rho - \widehat \rho}_{\C0([0,t]; \W11(\reali; \reali^k ))} t.
\end{align*}
Note that as $\rho, \widehat \rho$ satisfy
\ref{item:14}--\ref{item:16}, then $\rho, \widehat \rho$ belong to
$\C0\left([0,t]; \W1\infty(\reali; \reali^k)\right)$ and $\partial_x\rho \in
 \L1\left([0,t]; \W11(\reali; \reali^k)\right).$ Indeed,
$\rho \in \C0\left([0,t]; \L\infty( \reali; \reali^k)\right)$ due to
\ref{item:40} and Sobolev embedding
$\W11(\reali; \reali^k) \hookrightarrow \L\infty(\reali; \reali^k)$,
see~\cite[Theorem~8.8]{zbMATH05633610}.  Moreover
$\partial_x \rho \in \C0\left([0,t]; \L\infty( \reali; \reali^k)\right)$: first,
$\partial_x \rho \in \L1\left([0,t]; \W11(\reali; \reali^k)\right)
\hookrightarrow \L1\left([0,t]; \L\infty( \reali; \reali^k)\right) $
owing to~\ref{item:40} and~\ref{item:19}; second,
$\partial_{tx}^2 \rho \in \L\infty([0,t]\times \reali; \reali^k)$ due
to~\ref{item:16}. Then, as
$\partial_x \rho \in \W11\left([0,t]; \L\infty(\reali;
  \reali^k)\right)$, by~\cite[Theorem~2, Section~5.9]{MR2597943} we get
$\partial_x\rho \in \C0\left([0,t]; \L\infty(\reali;
  \reali^k)\right)$. The case of $\widehat\rho$ is analogous.
\end{proof}

\subsection{Proof of Theorem~\ref{thm:main}}
\label{subsec:coupled-problem}

Fix a positive constant $M$. We consider first the case of an initial
datum $(\rho_o, r_o)$ with
\begin{equation}
  \label{eq:67}
    \begin{array}{c}
    \rho_o \in \left( \C2 \cap \W2\infty \cap \W21 \right)(\reali;
    \reali^k)\quad
    r_o \in \cBV{} (\reali; \reali)
    \\
    \norma{\rho_o}_{\W21 (\reali; \reali^k)}
    +
    \norma{r_o}_{\cBV{} (\reali; \reali)}
    \leq
    M.
  \end{array}
\end{equation}
Under this condition, the proof is achieved by means of a sequence of approximate solutions $(\rho_n,r_n)$ iteratively defined as follows.
\begin{equation}
  \label{eq:14}
  \begin{cases}
    \partial_t\rho_1^i + \partial_x \left( \rho_1^i \, v^i_{NL}\left(
    \Sigma \rho_1 * \eta^i + (r_o * \eta^i) (x) \right) \right)=0
    \qquad i=1, \ldots, k
    \\
    \rho_1(0,x) = \rho_o(x)
  \end{cases}
\end{equation}
and
\begin{equation}
  \label{eq:15}
  \begin{cases}
    \partial_tr_1 +
    \partial_x \left( r_1 \, v_{L}\left( \Sigma \rho_1 (t,x) + r_1 \right)
    \right)=0,
    \\
    r_1(0,x)=r_o(x) \,.
  \end{cases}
\end{equation}
For $n \geq 2$, define recursively the functions $\rho_n$ and $r_n$ as
solutions to the problems
\begin{equation}
  \label{eq:9}
  \begin{cases}
    \partial_t\rho_n^i +
    \partial_x\left(
    \rho_n^i \,v^i_{NL}\left(\Sigma \rho_n  * \eta^i + (r_{n-1} (t)  * \eta^i) (x)\right)\right)
    =0 \qquad i=1, \ldots, k
    \\
    \rho_n(0,x) = \rho_o(x),
  \end{cases}
\end{equation}
and
\begin{equation}
  \label{eq:10}
  \begin{cases}
    \partial_t r_n + \partial_x\left(r_n \, v_L\left( \Sigma \rho_{n} (t,x) + r_n \right)
     \right)=0
    \\
    r_n(0,x)=r_o(x) \,.
  \end{cases}
\end{equation}

\begin{enumerate}[label=\bf{Step~1.}, ref=\textup{\textbf{Step~1}},
  align=left]
    \item  \label{lem:benDef}
    \textbf{The functions $\rho_n$ and $r_n$ defined in~\eqref{eq:9}
  and~\eqref{eq:10} are well defined with
  $\rho_n \in \C0\left([0,T]; \W11(\reali; \reali^k)\right)$ and
  $r_n \in \C0\left([0,T]; \L1(\reali;\reali)\right)$.}
\end{enumerate}
The inductive procedure on which the proof of the present step is
based follows this scheme:
\begin{center}
  \begin{tabular}{cc@{\hspace{-.025cm}}c@{\hspace{-.025cm}}c}
    $n=1$
    & \ref{item:17} \ref{item:hyp-eta} \eqref{eq:67}
    & $\implies$
    & $\rho_1$ satisfies~\ref{item:32}--\ref{item:38}
    \\[6pt]
    $n\geq1$
    & $\rho_n$ satisfies~\ref{item:32}--\ref{item:38}
      and~\ref{item:21}
    & $\implies$
    & \ref{item:14}--\ref{item:16}
    \\
    & \ref{item:21} and \ref{item:14}--\ref{item:16}
    & $\implies$
    & \ref{item:5}--\ref{item:9}
    \\
    & \ref{item:5}--\ref{item:9}  and~\eqref{eq:67}
    & $\implies$
    & $r_n \in \C0 \left([0,T];\L1 (\reali;\reali)\right)$
    \\
    & \ref{item:17} \ref{item:hyp-eta} \eqref{eq:67}  and
      $r_n \in \C0 \left([0,T];\L1 (\reali;\reali)\right)$
    & $\implies$
    & $\rho_{n+1}$ satisfies~\ref{item:32}--\ref{item:38}
  \end{tabular}
\end{center}

Fix an arbitrary $T>0$. Under assumptions~\ref{item:17},
\ref{item:hyp-eta} and with $(\rho_o, r_o)$ as in~\eqref{eq:67},
problem~\eqref{eq:14} fits into Lemma~\ref{lem:4} and \Cref{lem:5}.
Hence $\rho_1 \in \C0\left([0,T]; \W11(\reali; \reali^k)\right)$ is well defined
and ${\rho_1}$ satisfies~\ref{item:14}--\ref{item:16} of \Cref{lem:6}.

Similarly, assumptions~\ref{item:21}, \eqref{eq:67} and the
properties~\ref{item:14}--\ref{item:16} of ${\rho_1}$ ensure that
\Cref{lem:0} and \Cref{lem:1} apply, so that $r_1$ is well defined
and $r_1 \in \C0\left([0, T]; \L1(\reali; \reali)\right)$.

\textit{Induction process.}  Fix $n \in \mathbb{N}, n \ge 2$ and
suppose that $r_{n-1} \in \C0\left([0,T]; \L1(\reali;\reali)\right)$. Then, under
assumptions~\ref{item:17}, \ref{item:hyp-eta} and~\eqref{eq:67},
$\rho_n \in \C0\left([0,T]; \W11(\reali; \reali^k)\right)$, solution
to~\eqref{eq:9} is well defined by \Cref{lem:4} and
\Cref{lem:5}. Furthermore, ${\rho_n}$ satisfies
~\ref{item:14}--\ref{item:16} of \Cref{lem:6} by assumptions on
$r_{n-1}$.  Finally, assumptions~\ref{item:21}, \eqref{eq:67} and
the properties~\ref{item:14}--\ref{item:16} of ${\rho_n}$ ensure that
\Cref{lem:0} and \Cref{lem:1} apply, so that $r_n$ is well defined
and $r_n \in \C0\left([0, T]; \L1(\reali; \reali)\right)$.

\begin{enumerate}[label=\bf{Step~2.}, ref=\textup{\textbf{Step~2}},
  align=left]
    \item  \label{lem:7}
    \textbf{
    There exists a constant
  $\mathcal{C}_4 \coloneqq \mathcal{C} (
  \norma{\eta}_{\W3\infty(\reali,\reali^k)},
  \norma{v_{NL}}_{\W3\infty(\reali,\reali^k)}, M,T)$ such that
  $ \norma{\rho_n(t)- \rho_{n-1}(t)}_{\W11(\reali;\reali^k)} \le
  \mathcal{C}_4 \, \norma{r_{n-1} - r_{n-2}} _{\C0([0,t];
    \L1(\reali;\reali))} \, t $ for all $t \in [0,T]$.
    }
\end{enumerate}
Fix $n \in\naturali \setminus \{0,1\}$. We have
$r_{n-2}, r_{n-1} \in \C0\left([0,T]; \L1(\reali;\reali)\right)$ by~\ref{lem:benDef}. Define
$\mathcal{R} \coloneqq \norma{r_o}_{\L1(\reali; \reali)}$.
Lemma~\ref{lem:0} ensures that
$\max\{\norma{r_{n-2}}_{\C0([0,T]; \L1(\reali; \reali))}, \norma{r_{n-1}}_{\C0([0,T]; \L1(\reali; \reali))}\} \leq \mathcal{R}$. Apply Lemma~\ref{lem:Linfty} with
  $r \coloneqq r_{n-2}$, $\widehat{r} \coloneqq r_{n-1}$ obtaining that, for all
  $t \in [0,T]$,
  \begin{eqnarray*}
    &
    & \norma{\rho_n(t) - \rho_{n-1}(t)}_{\W11(\reali; \reali^k)}
    \\
    & \leq
    & \mathcal{C}\left(
      \norma{\eta}_{\W3\infty(\reali; \reali^k)},
      \norma{v_{NL}}_{\W3\infty(\reali; \reali^k)},
      \norma{\rho_o}_{\W21(\reali; \reali^k)},
      \norma{r_o}_{\L1(\reali; \reali)},
      t
      \right)
    \\
    &
    & \times
      \int_0^t \norma{r_{n-1}(\tau) - r_{n-2}(\tau)}_{\L1(\reali; \reali)}
      \d{\tau},
  \end{eqnarray*}
  concluding the step.

  \begin{enumerate}[label=\bf{Step~3.}, ref=\textup{\textbf{Step~3}},
    align=left]
  \item \label{lem:8} \textbf{There exists
      $\mathcal{C}_5 \coloneqq \mathcal{C}\left(
        \norma{\eta}_{\W3\infty(\reali;\reali^k)},
        {\norma{v_{NL}}_{\W3\infty(\reali;\reali^k)}}, R_L,
        \norma{v_L'}_{\W1\infty(\reali;\reali)}, M, T \right)$ such
      that
      $ \norma{r_n(t)- r_{n-1}(t)}_{\L1(\reali;\reali)} \le
      \mathcal{C}_5 \, \norma{r_{n-1}-r_{n-2}}_{\C0([0,t];
        \L1(\reali;\reali))} \,t^2 $ for $t \in [0,T]$.  }
  \end{enumerate}
Fix $n \in \naturali \setminus \{0,1\}$ and let $\rho_{n-1}, \rho_n$ be as defined in~\ref{lem:benDef}.
Then, as they satisfy~\ref{item:14}--\ref{item:16}, we may apply \Cref{distloc}, which reads for all $t \in [0,T]$
\begin{align*}
  & \norma{r_n(t) -  r_{n-1}(t)}_{\L1(\reali; \reali)}
  \\
  \leq
  & \mathcal{C} \! \left(
    \begin{array}{@{}c@{}}
      R_L,
      \norma{v'_L}_{\W1\infty(\reali; \reali)},
      \norma{\rho_{n}}_{\C0([0,t]; \W1\infty(\reali; \reali^k)},
      \\
      \norma{\rho_{n{-}1}}_{\C0([0,t]; \W1\infty(\reali; \reali^k)},
      \norma{\partial_x \rho_n}_{\L1([0,t]; \W11(\reali; \reali^k))},
      \tv(r_o)
    \end{array}
    \right)\!
    t \norma{\rho_n {-}  \rho_{n{-}1}}_{\C0([0,t]; \W11(\reali; \reali^k ))}.
\end{align*}
  Recall now that ${\rho_{n}} = (\rho_n^1, \ldots, \rho_n^k)$ is such
  that for $i=1, \ldots, k$, $\rho_n^i$ solves
  \begin{equation}
    \!\!\!\!\!
    \begin{cases}
      \partial_t {\rho_{n}^i}
      + \partial_x \left( {\rho_{n}^i} V^i(t,x,{\rho_{n}} * \eta) \right)
      = 0
      \\
      \rho_{n}^i (0,x) = \rho_o^i(x)
    \end{cases}
    \!\!
    \mbox{for }
    V^i(t,x,q) \coloneqq
    v_{NL}^i\left(\Sigma q + (r_{n-1}(t)*\eta^i) (x)\right),
  \end{equation}
  with $V \in \C0([0,T]; \mathcal{V})$. Indeed,
  \begin{flalign*}
    \norma{V}_{\C0([0,T];\mathcal{V})}
    \leq
    & \mathcal{C} \left(
      \norma{\eta}_{\W3\infty(\reali; \reali^k)},
      \norma{v_{NL}}_{\W3\infty(\reali; \reali^k)},
      \norma{r_{n-1}}_{\C0([0,T]; \L1(\reali; \reali))}
      \right) \!\!
    &[\mbox{By \Cref{lem:3}}]
    \\
    \leq
    & \mathcal{C} \left(
      \norma{\eta}_{\W3\infty(\reali; \reali^k)},
      \norma{v_{NL}}_{\W3\infty(\reali; \reali^k)},
      \norma{r_o}_{ \L1(\reali; \reali)}
      \right). \!\!
    & [\mbox{By~\Cref{lem:0}}]
  \end{flalign*}
  The latter bound, due to~\ref{item:39}, \ref{item:47} and~\ref{item:38} in
  \Cref{lem:5}, implies that
  \begin{eqnarray}
    \nonumber
    \!\!\!\!\!\!\!\!\!
    &
    & \norma{{\rho_{n}}}_{\C0([0,t]; \W1\infty(\reali; \reali^k))}
    \\
    \!\!\!\!\!\!\!\!\!
    \label{eq:36}
    & \leq
    & \mathcal{C}  \left(
      \norma{\eta}_{\W3\infty(\reali;\reali^k)},
      \norma{v_{NL}}_{\W3\infty(\reali;\reali^k)},
      \norma{\rho_o}_{\L1(\reali;\reali^k)},
      \norma{\rho_o}_{\W1\infty(\reali;\reali^k)},
      \norma{r_o}_{\L1(\reali;\reali)},
      t
      \right)\quad
    \\
    \nonumber
    \!\!\!\!\!\!\!\!\!
    &
    & \norma{\partial_x {\rho_{n}}}_{\L1([0,t]; \W11(\reali;\reali^k)}
    \\
    \label{eq:38}
    \!\!\!\!\!\!\!\!\!
    & \leq
    & \mathcal C\left(
      \norma{\eta}_{\W3\infty(\reali;\reali^k)},
      \norma{v_{NL}}_{\W3\infty(\reali;\reali^k)},
      \norma{\rho_o}_{\W21(\reali;\reali^k)},
      \norma{r_o}_{\L1(\reali;\reali)},
      t
      \right).
  \end{eqnarray}
  Identical bounds are valid for $\rho_{n-1}$, too.
Then,~\ref{lem:7} allows to conclude that
\begin{align*}
  & \norma{r_n(t) -  r_{n-1}(t)}_{\L1(\reali; \reali)}
  \\
  \leq
  &
  \mathcal{C} \! \left(
    \begin{array}{@{}c@{}}
    \norma{\eta}_{\W3\infty(\reali; \reali^k)}, \!
    \norma{v_{NL}}_{\W3\infty(\reali; \reali^k)},\!
    R_L,
    \norma{v'_L}_{\W1\infty(\reali; \reali)},\\
    \norma{\rho_o}_{\W21(\reali; \reali^k)},
    \norma{r_o}_{\cBV{}(\reali; \reali)},
    t
    \end{array}
    \right) \,\mathcal{C}_4
    \, \norma{r_{n-1} -  r_{n-2}}_{\C0([0,t]; \L1(\reali; \reali))}
    \, t^2,
\end{align*}
proving the claim of the present step.
\begin{enumerate}[label=\bf{Step~4.}, ref=\textup{\textbf{Step~4}},
  align=left]
\item \label{lem:9} \textbf{ \textit{A priori} estimates on $(\rho,r)$
    in
    $\C0\left([0,T]; \W11(\reali; \reali^k) \times \L1(\reali;
      \reali)\right)$, solution to~\eqref{eq:1} in the sense of
    \Cref{def:sol}.  }
\end{enumerate}
Once called
$V^i(t,x,q) \coloneqq v^i_{NL} \left( \Sigma q + \left(
    r(t)*\eta^i\right)(x)\right)$ and
$w(t,x,r) \coloneqq v_L\left( \Sigma \rho(t,x) + r \right)$, by
\Cref{def:sol}, $\rho$ and $r$ are solutions to~\eqref{eq:7}
and~\eqref{eq:2} respectively. We claim that the function $w$
satisfies the requirements~\ref{item:5}--\ref{item:9}. Indeed, the
following scheme holds:
\begin{center}
\begin{tabular}{@{}cccc}
$r \in \C0\left([0,T]; \L1(\reali; \reali)\right)$, \ref{item:17}, \ref{item:hyp-eta}
&
$\Rightarrow$
&
$V \in \C0([0,T]; \mathcal{V})$
&
by \Cref{lem:3}
\\
$V \in \C0([0,T]; \mathcal{V})$, \eqref{eq:67}, \ref{item:hyp-eta}
&
$\Rightarrow$
&
\ref{item:32}--\ref{item:38} hold
&
by \Cref{lem:5}
\\
\ref{item:32}--\ref{item:38} , \ref{item:21}
&
$\Rightarrow$
&
\ref{item:14}--\ref{item:16}
&
\\
\ref{item:21} and~\ref{item:14}--\ref{item:16}
&
$\Rightarrow$
&
$w$ satisfies~\ref{item:5}--\ref{item:9}\qquad
&
by \Cref{lem:6}.
\end{tabular}
\end{center}
Because of the fact $w$ meets the requirements~\ref{item:5}--\ref{item:9}, we apply \Cref{lem:0}: as $r \in \C0([0,T]; \L1(\reali; \reali))$ and the function $\widehat r \equiv 0$ are both Kru\v zkov solutions to~\eqref{eq:2}, then
\begin{equation}
    \label{apriori1}
    \norma{r(t)}_{\L1(\reali;\reali)}
    \leq
    \norma{r_o}_{\L1(\reali;\reali)}.
\end{equation}
By~\eqref{eq:24} in \Cref{lem:3} and~\eqref{apriori1}, one gets that
\begin{eqnarray}
  \nonumber
  \norma{V}_{\C0([0,T];\mathcal{V})}
  & \leq
  & \mathcal{C} \left(
    \norma{\eta}_{\W3\infty(\reali; \reali^k)},
    \norma{v_{NL}}_{\W3\infty(\reali; \reali^k)},
    \norma{r}_{\C0([0,T];\L1(\reali; \reali))}
    \right)
  \\
  \label{eq:68}
  & \leq
  & \mathcal{C} \left(
    \norma{\eta}_{\W3\infty(\reali; \reali^k)},
    \norma{v_{NL}}_{\W3\infty(\reali; \reali^k)},
    \norma{r_o}_{\L1(\reali; \reali)}
    \right).
\end{eqnarray}
Moreover by \Cref{lem:5} we  obtain the
following \textit{a priori} estimates:
\begin{displaymath}
  \begin{array}{@{\!}r@{\,}c@{\,}l@{}}
    \norma{\rho(t)}_{\L1(\reali; \reali^k)}
    & =
    & \norma{\rho_o}_{\L1(\reali; \reali^k)}
      \hfill [\mbox{By~\ref{item:46} in \Cref{lem:4}}]
    \\
    \norma{\partial_x \rho(t)}_{\L1(\reali; \reali^k)}
    & \leq
    & \mathcal{C} \left(
      \norma{\eta}_{\W3\infty(\reali; \reali^k)},
      \norma{V}_{\C0([0,T];\mathcal{V})},
      \norma{\rho_o}_{\W11(\reali; \reali^k)},
      t
      \right)
      \hfill [\mbox{By~\ref{item:39}}]
    \\
    \norma{\partial^2_{xx} \rho(t)}_{\L1(\reali; \reali^k)}
    & \leq
    &  \mathcal{C} \left(
      \norma{\eta}_{\W3\infty(\reali; \reali^k)},
      \norma{V}_{\C0([0,T];\mathcal{V})},
      \norma{\rho_o}_{\W21(\reali; \reali^k)},
      t
      \right)
      \hfill [\mbox{By~\ref{item:39}}]
    \\
    \norma{\rho(t)}_{\L\infty(\reali; \reali^k)}
    & \leq
    & \mathcal{C}\left(
      \norma{\eta}_{\W3\infty (\reali; \reali^k)},
      \norma{V}_{\C0([0,T];\mathcal{V})},
      \norma{\rho_o}_{\L1(\reali;\reali^k)},
      \norma{\rho_o}_{\L\infty(\reali;\reali^k)}, t
      \right)
      \hfill [\mbox{By~\ref{item:47}}]
    \\
    \norma{\partial_x \rho(t)}_{\L\infty(\reali; \reali^k)}
    & \leq
    & \mathcal{C} \!\left(\!
      \norma{\eta}_{\W3\infty(\reali; \reali^k)},
      \norma{V}_{\C0([0,T];\mathcal{V})},
      \norma{\rho_o}_{\L1(\reali; \reali^k)},
      \norma{\rho_o}_{\W1\infty(\reali; \reali^k)}, \!
      t \!
      \right)
      \hfill \! [\mbox{By~\ref{item:38}}]
  \end{array}
\end{displaymath}
Hence, there exists a positive constant
\begin{equation}
  \label{eq:85}
  \mathcal{C}_6
  \coloneqq
  \mathcal{C} \left(
    \norma{\eta}_{\W3\infty(\reali; \reali^k)},
    \norma{v_{NL}}_{\W3\infty(\reali; \reali^k)},
    \norma{\rho_o}_{\W21(\reali; \reali^k)},
    \norma{r_o}_{\L1(\reali; \reali)},
    T
  \right)
\end{equation}
such that
\begin{equation}
   \label{apriori3}
   \norma{\rho(t)}_{\L1(\reali;\reali^k)}
   =
   \norma{\rho_o}_{\L1(\reali;\reali^k)}
   \qquad \text{and} \qquad
   \norma{\rho(t)}_{\W21(\reali; \reali^k)}
   \leq \mathcal{C}_6\,.
\end{equation}
We are now ready to prove the estimate on the total variation of $r$
taking advantage of~\eqref{eq:tvest} in \Cref{lem:2}, with
$R \coloneqq R_L + \norma{\rho}_{\L\infty([0,T]\times \reali;
  \reali^k)}$.  For all $t \in [0,T]$, $r(t)$ is in
$\BV{}(\reali; \reali)$ and
\begin{align*}
    \tv\left(r(t)\right)
    \leq &
    \tv(r_o) e^{\kappa_o t} + R \int_0^t e^{\kappa_o (t-\tau)}\int_\reali \norma{\partial_{xx}^2 w(\tau, x, \cdot)}_{\L\infty(\reali; \reali)} \d{x} \d{\tau}
\end{align*}
where
  \begin{equation*}
    \kappa_o \leq
    3
    \left(
      \norma{\partial_x w}_{\L\infty([0,T]\times \reali\times\reali;\reali)}
      +
      R \,
      \norma{\partial_{xr}^2 w}_{\L\infty([0,T]\times\reali\times\reali;\reali)}
    \right).
  \end{equation*}
Computations analogous to the ones in the proof to \Cref{lem:6} and the use of~\eqref{apriori3} yield
\begin{displaymath}
  \kappa_o
  \leq \mathcal C
  \left(
    \begin{array}{c}
      \norma{\eta}_{\W3\infty(\reali;\reali^k)},
      \norma{v_{NL}}_{\W3\infty(\reali;\reali^k)},
      R_L, \norma{v_L'}_{\W1\infty(\reali;\reali)},
      \\

      \norma{\rho_o}_{\L1(\reali;\reali^k)},
      \norma{\rho_o}_{\W1\infty(\reali;\reali^k)},
      \norma{r_{o}}_{\L1(\reali;\reali)},
      T
    \end{array}
  \right)
\end{displaymath}
and
\begin{eqnarray*}
  &
  &  \int_0^t \int_\reali \norma{\partial_{xx}^2 w(\tau,x,\cdot)}
    _{\L\infty(\reali;\reali)} \d{x} \d{\tau}
  \\
  & \leq
  &
    \mathcal C \left(
    \norma{\eta}_{\W3\infty(\reali;\reali^k)},
    \norma{v_{NL}}_{\W3\infty(\reali;\reali^k)},
    \norma{v_L'}_{\W1\infty(\reali;\reali)},
    \norma{\rho_o}_{\W21(\reali;\reali^k)},
    \norma{r_o}_{\L1(\reali; \reali)},
    t
    \right).
\end{eqnarray*}
So, one gets that for all $t \in [0,T]$
\begin{displaymath}
  \tv\left(r(t)\right)
  \leq
  \mathcal{C}
  \left(
    \begin{array}{c}
      \norma{\eta}_{\W3\infty(\reali;\reali^k)},
      \norma{v_{NL}}_{\W3\infty(\reali;\reali^k)}, R_L,
      \norma{v_L'}_{\W1\infty(\reali;\reali)},
      \\
      \norma{\rho_o}_{\W21(\reali;\reali^k)},
      \norma{r_{o}}_{\cBV{}(\reali;\reali)},
      T
    \end{array}
  \right)
\end{displaymath}
and, thanks to \Cref{lem:lei}, $\norma{r(t)}_{\L\infty(\reali; \reali)}$ is  bounded similarly. Finally, there exist a positive
\begin{displaymath}
    \mathcal{C}_7
     \coloneqq
    \mathcal{C}
    \left(
      \begin{array}{c}
        \norma{\eta}_{\W3\infty(\reali;\reali^k)},
        \norma{v_{NL}}_{\W3\infty(\reali;\reali^k)},
        R_L, \norma{v_L'}_{\W1\infty(\reali;\reali)},
        \\
        \norma{\rho_o}_{\W21(\reali;\reali^k)},
        \norma{r_{o}}_{\cBV{}(\reali;\reali)},
        T
      \end{array}
    \right)
  \end{displaymath}
such that
for all $t \in [0,T]$ it holds
\begin{equation}
  \label{eq:69}
  \norma{r(t)}_{\cBV{}(\reali; \reali)}
  +
  \norma{r(t)}_{\L\infty(\reali; \reali)} \leq
  \mathcal{C}_7 \,.
\end{equation}
\begin{enumerate}[label=\bf{Step~5.}, ref=\textup{\textbf{Step~5}},
  align=left]
    \item  \label{cor:localexist}
    \textbf{
There exists $T^*>0$ such that the sequence
  $\{(\rho_n, r_n)\}$ constructed in~\ref{lem:benDef},
  restricted to the time interval $[0, T^*]$,
  converges to a limit point
  $(\rho, r)$ in
  $\C0([0,T^*]; \W11(\reali; \reali^k) \times \L1(\reali;\reali) )$.
    }
\end{enumerate}
Let $\mathcal{C}_6$ and $\mathcal{C}_7$ be as in~\ref{lem:9}, so
that~\eqref{apriori3} and~\eqref{eq:69} hold. Use $\mathcal{C}_4$
from~\ref{lem:7}, $\mathcal{C}_5$ from~\ref{lem:8} and recall that
$\mathcal{C} (\cdot)$ is non decreasing in each argument. Then, there
exists $\bar{\mathcal{C}}_5$ such that
$\mathcal{C}_5 \leq \bar{\mathcal{C}}_5$, where
\begin{eqnarray*}
  \mathcal{C}_5
  & =
  & \mathcal{C}
    \left(
    \norma{\eta}_{\W3\infty(\reali;\reali^k)},
    \norma{v_{NL}}_{\W3\infty(\reali;\reali^k)},
    R_L,
    \norma{v_L'}_{\W1\infty(\reali;\reali)},
    M, T
    \right)  \qquad \mbox{ and}
  \\
  \bar{\mathcal{C}}_5
  & \coloneqq
  & \mathcal{C}
    \left(
    \norma{\eta}_{\W3\infty(\reali;\reali^k)},
    \norma{v_{NL}}_{\W3\infty(\reali;\reali^k)},
    R_L,
    \norma{v_L'}_{\W1\infty(\reali;\reali)},
    \mathcal{C}_6 + \mathcal{C}_7,
    T
    \right)
\end{eqnarray*}
  so that, according to~\ref{lem:7} and ~\ref{lem:8}, for all
  $t \in [0, T]$ and
  $n \in \mathbb{N}\setminus \left\{0, 1\right\}$,
  \begin{eqnarray}
    \label{eq:47}
    \norma{{\rho_n}(t) - {\rho_{n-1}}(t)}_{ \W11(\reali; \reali^k)}
    & \leq
    & \, \mathcal{C}_4 \, t \,
      \norma{r_{n-1}-r_{n-2}}_{\C0([0,T^*]; \L1(\reali; \reali))} \,,
    \\
    \nonumber
    \norma{r_n(t) - r_{n-1}(t)}_{\L1(\reali; \reali)}
    & \leq
    & \bar{\mathcal{C}}_5 \, t^2 \,
      \norma{r_{n-1} - r_{n-2}}_{\C0([0,T^*]; \L1(\reali; \reali))} ,
  \end{eqnarray}
  The choice
  $T^* \coloneqq \min\left\{1 / \sqrt{2 \bar{\mathcal{C}}_5},
    T\right\}$ implies that $r_n$ is a Cauchy sequence in the space
  $\C0\left([0,T^*]; \L1(\reali; \reali)\right)$, hence it converges
  to a limit
  $r \in \C0\left([0,T^*]; \L1(\reali; \reali)\right)$.
  Then,~\eqref{eq:47} yields that $\rho_n$ is a Cauchy sequence in
  $\C0\left([0,T^*]; \W11(\reali; \reali^k)\right)$ and converges to
  $\rho \in \C0\left([0,T^*]; \W11(\reali; \reali^k)\right)$,
  completing this Step.

\begin{enumerate}[label=\bf{Step~6.}, ref=\textup{\textbf{Step~6}},
  align=left]
    \item  \label{lem:conv}
    \textbf{
    The limit function
  $(\rho,r) \in \C0\left([0,T^*]; \W11(\reali; \reali^k) \times
  \L1(\reali;\reali)\right)$ solves~\eqref{eq:1} in the sense of
  \Cref{def:sol}.
    }
  \end{enumerate}

  Apply \Cref{lem:GoodSol} with $\rho_{o,n} = \rho_o$,
  $r_{o,n} = r_o$, $e_n = r_{n-1}-r_n$ and $\epsilon_n = 0$.
\begin{enumerate}[label=\bf{Step~7.}, ref=\textup{\textbf{Step~7}},
  align=left]
    \item  \label{globalexist}
    \textbf{
    Existence of a global solution $(\rho, r) \in \C0\left([0,T]; \W11(\reali; \reali^k)\times \L1(\reali; \reali)\right)$.
    }
\end{enumerate}
Let $T^*$ be as in the proof of~\ref{cor:localexist} and
$(\rho, r) \in \C0\left([0,T^*]; \W11(\reali; \reali^k) \times
  \L1(\reali; \reali)\right)$ be the solution obtained as limit of the
Cauchy sequence $\{(\rho_n ,r_n)\}$.  We may apply \Cref{lem:5}, which
ensures that
$\rho(T^*) \in (\C2 \cap \W2\infty \cap \W21)(\reali; \reali^k)$ owing
to~\ref{item:32}--\ref{item:36} and the \textit{a priori} estimate
$\norma{\rho(T^*)}_{\W21(\reali; \reali^k)} \leq \, \mathcal{C}_6$
holds by~\ref{lem:9}.  Furthermore, \Cref{lem:0} and \Cref{lem:2}
yield that $r(T^*) \in \cBV{}(\reali; \reali)$ and
$\norma{r(T^*)}_{\cBV{}(\reali; \reali)} \leq \mathcal{C}_7$. Hence,
we can iterate the same construction in order to extend $(\rho, r)$ to
the time intervals $[T^*, 2T^*]$, $[2T^*, 3T^*]$, $\ldots$ until we
reach the final time $T$. Clearly, the extended $(\rho, r)$
solves~\eqref{eq:1} on the whole $[0,T]$ in the sense of
Definition~\ref{def:sol}.

\begin{enumerate}[label=\bf{Step~8.}, ref=\textup{\textbf{Step~8}},
  align=left]
    \item  \label{initial_dependence}
    \textbf{
Locally Lipschitz dependence on the initial datum.
    }
\end{enumerate}
 Fix
  $(\widehat\rho_o, \widehat r_o) \in (\C2 \cap \W21 \cap\W2\infty)
    (\reali;\reali^k) \times \cBV{} (\reali;
    \reali)$. Consider the problems
  \begin{displaymath}
    \begin{cases}
      \partial_t \rho^i
      +
      \partial_x \left(\rho^i \;
      v_{NL}^i\left((\Sigma \rho + r)  *\eta^i\right)\right)
      =0
      \\
      \partial_t r
      + \partial_x \left(r \; v_L\left(\Sigma \rho+r \right)\right)
      = 0
      \\
      \rho(0) = \rho_o
      \\
      r(0) = r_o
    \end{cases}
    \quad
    \begin{cases}
      \partial_t \widehat \rho^i
      +
      \partial_x \left(\widehat\rho^i \;
      v_{NL}^i\left(
      (\Sigma \widehat \rho + \widehat r)  *\eta^i\right)
      \right)
      =0
      \\
      \partial_t \widehat r
      + \partial_x \left(\widehat r \;
      v_L\left(\Sigma \widehat \rho+ \widehat r \right)\right)
      = 0
      \\
      \widehat\rho(0) = \widehat \rho_o
      \\
      \widehat r(0) = \widehat r_o \,.
    \end{cases}
  \end{displaymath}
To establish the distance between the solutions
$\rho$ and $\widehat \rho$ to the nonlocal problems, it is sufficient to apply~\eqref{eq:56} in
\Cref{lem:Linfty}.  Indeed, call
$\mathcal{R} \coloneqq \max\left\{ \norma{r_o}_{\L1(\reali; \reali)},
  \norma{\widehat r_o}_{\L1(\reali; \reali)} \right\}$. By~\eqref{eq:56}--\eqref{eq:57}, for all $t \in [0,T]$
\begin{align}
  \label{eq:61}
\norma{\rho(t) {-} \widehat\rho(t)}_{\W11(\reali; \reali^k)}
\leq
(1 {+} \mathcal{C}_{3}\,t)
\norma{\rho_o {-} \widehat \rho_o}_{\W11(\reali; \reali^k)}
+
\mathcal{C}_{3} \int_0^t \norma{r(\tau) {-} \widehat r(\tau)}_{\L1(\reali; \reali)} \d\tau .
\end{align}
On the other hand, in order to get the stability related the to local problem we take advantage of \Cref{distloc}, which yields that for all $t \in [0,T]$
\begin{align*}
   &
    \norma{r(t) - \widehat r(t)}_{\L1(\reali; \reali)}
    \\
    \leq &
    \norma{r_o - \widehat r_o}_{\L1(\reali; \reali)}
    \\
    &
    +
    \mathcal{C} \left(
      \begin{array}{@{}c@{}}
        R_L,
        \norma{v'_L}_{\W1\infty(\reali; \reali)},
        \norma{\rho}_{\C0([0,t]; \W1\infty(\reali; \reali^k)},
        \\
        \norma{\widehat \rho}_{\C0([0,t]; \W1\infty(\reali; \reali^k)},
        \norma{\partial_x \rho}_{\L1([0,t]; \W11(\reali; \reali^k))},
        \tv(\widehat r_o)
      \end{array}
      \right)
    t \, \norma{\rho - \widehat \rho}_{\C0([0,t]; \W11(\reali; \reali^k ))}.
\end{align*}
Owing to~\ref{lem:9}, there exists a constant
\begin{displaymath}
  \mathcal{C}_8
  \coloneqq
  \mathcal{C}
  \left(
    \begin{array}{c}
      \norma{\eta}_{\W3\infty(\reali; \reali^k)},
      \norma{v_{NL}}_{\W3\infty(\reali; \reali^k)},
      \norma{\rho_o}_{\W21(\reali; \reali^k)},
      \\
      \norma{\widehat \rho_o}_{\W21(\reali; \reali^k)},
      \norma{r_o}_{\L1(\reali; \reali)},
      \norma{\widehat r_o}_{\L1(\reali; \reali)},
      t
    \end{array}
  \right)
\end{displaymath}
such that for all $t \in [0,T]$ one has
\begin{displaymath}
  \norma{\rho}_{\C0([0,t]; \W1\infty(\reali; \reali^k)}
  +
  \norma{\partial_x \rho}_{\L1([0,t]; \W11(\reali; \reali^k))}
  +
  \norma{\widehat \rho}_{\C0([0,t]; \W1\infty(\reali; \reali^k)}
  \leq
  \mathcal{C}_8 \, .
\end{displaymath}
So, collecting the previous estimates and~\eqref{eq:61}, one obtains that there exists a constant
\begin{displaymath}
  \mathcal{C}_9 \coloneqq
  \mathcal{C} \left(
    \begin{array}{c}
      \norma{\eta}_{\W3\infty(\reali; \reali^k)},
      \norma{v_{NL}}_{\W3\infty(\reali; \reali^k)},
      R_L,
      \norma{v'_L}_{\W1\infty(\reali; \reali)},
      \\
      \norma{\rho_o}_{\W21(\reali; \reali^k)},
      \norma{\widehat \rho_o}_{\W21(\reali; \reali^k)},
      \norma{r_o}_{\L1(\reali; \reali)},
      \norma{\widehat r_o}_{\cBV{}(\reali; \reali)},
      t
    \end{array}
  \right)
\end{displaymath}
such that
\begin{eqnarray*}
  &
  & \norma{r(t) - \widehat r(t)}_{\L1(\reali; \reali)}
  \\
  & \leq
  & \norma{r_o - \widehat r_o}_{\L1(\reali; \reali)}
    +
    \mathcal{C}_9 \, t \left(
    \norma{\rho_o - \widehat \rho_o}_{\W11(\reali; \reali^k)}
    +
    \int_0^t \norma{r(\tau)- \widehat r(\tau)}_{\L1(\reali; \reali)} \d\tau
    \right).
\end{eqnarray*}
Gronwall Lemma allows to conclude that for all $t \in [0,T]$
\begin{displaymath}
  \norma{r(t)-\widehat r(t)}_{\L1(\reali; \reali)}
  \leq
    \left(
    \norma{r_o - \widehat r_o}_{\L1(\reali; \reali)}
    +
    \mathcal{C}_9 \, t \, \norma{\rho_o - \widehat\rho_o}_{\W11(\reali; \reali^k)}
    \right) e^{ \mathcal{C}_9 t^2}
\end{displaymath}
and, completing~\eqref{eq:61},
\begin{eqnarray*}
  \norma{\rho(t) - \widehat\rho(t)}_{\W11(\reali; \reali^k)}
  & \leq
  & (1 + \mathcal{C}_3 t)
    \norma{\rho_o - \widehat \rho_o}_{\W11(\reali; \reali^k)}
  \\
  &
  & +
    \mathcal{C}_3
    \left(
    \norma{r_o - \widehat r_o}_{\L1(\reali; \reali)} t
    +
    \mathcal{C}_9 t^2
    \norma{\rho_o - \widehat \rho_o}_{\W11(\reali; \reali^k)}
    \right)
    e^{ \mathcal{C}_9 t^2} \,.
\end{eqnarray*}
In conclusion, for all $t \in [0,T]$ there exists a constant
\begin{equation}
  \label{eq:78}
  \mathcal{C}_{10}
  \coloneqq
  \mathcal{C} \left(
    \begin{array}{c}
      \norma{\eta}_{\W3\infty(\reali; \reali^k)},
      \norma{v_{NL}}_{\W3\infty(\reali; \reali^k)},
      R_L,
      \norma{v_L'}_{\W1\infty(\reali; \reali)},
      \\
      \norma{\rho_o}_{\W21(\reali; \reali^k)},
      \norma{\widehat\rho_o}_{\W21(\reali; \reali^k)},
      \norma{r_o}_{\L1(\reali; \reali)},
      \norma{\widehat r_o}_{\cBV{}(\reali; \reali)},
      t
    \end{array}
  \right)
\end{equation}
  such that
  \begin{equation}
    \label{eq:81}
    \begin{array}{cl}
      & \norma{\rho(t)- \widehat \rho(t)}_{\W11(\reali; \reali^k)}
        +
        \norma{r(t)- \widehat r(t)}_{\L1(\reali; \reali)}
      \\
      \leq
      &
        (1+\mathcal{C}_{10} \, t)
        \left(
        \norma{\rho_o - \widehat\rho_o}_{\W11(\reali; \reali^k)}
        +
        \norma{r_o - \widehat r_o}_{\L1(\reali; \reali)}
        \right).
    \end{array}
  \end{equation}

  \begin{enumerate}[label=\bf{Step~9.}, ref=\textup{\textbf{Step~9}},
  align=left]
    \item
    \textbf{
Local Lipschitz continuity in time.
    }
\end{enumerate}
Apply~\ref{item:26} in
Lemma~\ref{lem:5} to obtain that for all $t, \widehat t \in [0,T]$,
\begin{flalign}
  \nonumber
  & \norma{\rho (t) - \rho (\widehat{t})}_{\W11 (\reali; \reali^k)}
  \\
  \nonumber
  \leq
  & \mathcal{C}
    \left(
    \norma{\eta}_{\W3\infty (\reali; \reali^k)},
    \norma{V}_{\C0([0,T];\mathcal{V})},
    \norma{\rho_o}_{\W21 (\reali;\reali^k)},
    T\right)
    \modulo{t-\widehat{t}}
    & [\mbox{Use~\eqref{eq:68}}]
  \\
  \label{eq:76}
  \leq
  & \mathcal{C}
    \left(
    \norma{\eta}_{\W3\infty (\reali; \reali^k)},
    \norma{v_{NL}}_{\W3\infty(\reali; \reali^k)},
    \norma{\rho_o}_{\W21 (\reali;\reali^k)},
    \norma{r_o}_{\L1(\reali; \reali)},
    T\right)
    \modulo{t-\widehat{t}} \,.
  \end{flalign}
  Moreover, as we aim at applying \Cref{lem:continuity}, call
  $w(t,x,r) \coloneq v_L\left(\Sigma \rho(t,x) + r\right)$, so that by
  the proof to \Cref{lem:6} one gets
\begin{eqnarray*}
  \sup_{q \in [-K,K]} \tv\left(w (t, \cdot, q)\right)
  & \leq
  & \norma{v'_L}_{\L\infty(\reali; \reali)} \,
    \norma{\partial_x \rho(t)}_{\L1(\reali; \reali^k)}
  \\
  \norma{\partial_r w}_{\L\infty ([0,T]\times\reali\times\reali; \reali)}
  & \leq
  & \norma{v'_L}_{\L\infty(\reali; \reali)}
  \\
  \norma{w}_{\L\infty ([0,T]\times\reali\times\reali;\reali)}
  & \leq
  & V_L + \norma{v'_L}_{\L\infty(\reali; \reali)}R_L.
\end{eqnarray*}
Hence, as~\ref{lem:9} implies $K= \mathcal{C}_7$ and $\norma{\partial_x \rho(t)}_{\L1(\reali; \reali^k)} \leq \mathcal{C}_6$, again by~\ref{lem:9}, one obtains
\begin{flalign}
  \nonumber
  &\norma{r (t) - r (\widehat t)}_{\L1 (\reali;\reali)}
  \\
  \nonumber
  \leq \;
  & \mathcal{C}_7
    \left(
    \norma{v'_L}_{\L\infty(\reali; \reali)}{\mathcal{C}_6}
    + \mathcal{C}_7 \norma{v'_L}_{\L\infty(\reali; \reali)}
    +  V_L + \norma{v'_L}_{\L\infty(\reali; \reali)}R_L
    \right)
    \modulo{t-\widehat t}
  &[\mbox{By Lemma~\ref{lem:continuity}}]
  \\
  \label{eq:77}
  \leq \;
  & \mathcal{C}
    \left(
    \begin{array}{c}
      \norma{\eta}_{\W3\infty (\reali;\reali^k)},
      \norma{v_{NL}}_{\W3\infty (\reali;\reali^k)},
      R_L, V_L,
      \\
      \norma{v'_L}_{\W1\infty (\reali;\reali)},
      \norma{\rho_o}_{\W21 (\reali;\reali^k)},
      \norma{r_o}_{\cBV{}(\reali;\reali)},
      T
    \end{array}
    \right)
    \, \modulo{t-\widehat t}
\end{flalign}
completing the proof to the local Lipschitz continuity in time.

\begin{enumerate}[label=\bf{Step~10.}, ref=\textup{\textbf{Step~10}},
  align=left]
    \item
    \textbf{Definition and properties of $\mathcal{S}$.}
  \end{enumerate}

  By the arbitrariness of $T$, the above steps allow to define the map
  \begin{displaymath}
    \mathcal{S}
    \colon
     \reali_+
      \times (\C2 \cap \W2\infty \cap \W21) (\reali;\reali^k)
      \times \cBV{} (\reali; \reali)
     \longrightarrow
      (\C2 \cap \W2\infty \cap \W21) (\reali;\reali^k)
      \times \cBV{} (\reali; \reali)
  \end{displaymath}
  by
  $\mathcal{S}_t (\rho_o,r_o) \coloneqq \left(\rho (t), r (t)\right)$,
  where $(\rho,r)$ is the solution constructed above.

  By Definition~\ref{def:sol} and the uniqueness of the solutions
  to~\eqref{eq:53} and to~\eqref{eq:29}, for any $t_1>0$, the map
  $t \mapsto (\rho,r) (t_1+t)$ solves
  \begin{displaymath}
    \begin{array}{@{}l@{\quad}r@{\,}c@{\,}l@{}}
      \left\{
      \begin{array}{@{}l@{}}
        \partial_t \tilde\rho^i +
        \partial_x \left(\tilde\rho^i \, \tilde v^i (t,x)\right) =0
        \\
        \tilde\rho^i (0,x) = \rho^i (t_1, x)
      \end{array}
      \right.
      &
        \tilde v^i (t,x)
      & \coloneq
      & v_{NL}^i\left(
        \left(\Sigma \tilde \rho  (t_1+t) * \eta^i\right) (x)
        + \left(\tilde r (t_1+t)*\eta^i\right) (x)\right)
      \\
      \left\{
      \begin{array}{@{}l@{}}
        \partial_t \tilde r +
        \partial_x \left(\tilde r \, \tilde w (t,x,\tilde r)\right) =0
        \\
        \tilde r (0,x) = r (t_1, x)
      \end{array}
      \right.
      & \tilde w (t,x,r)
      & \coloneq
      & v_L\left(\Sigma \tilde \rho (t_1+t,x)  + \tilde r\right)
    \end{array}
  \end{displaymath}
  which ensures that
  $\mathcal{S}_t \circ \mathcal{S}_{t_1} = \mathcal{S}_{t+t_1}$, showing
  that $\mathcal{S}$ is a semigroup.

  For fixed $M>0$, call $B_M$ the set of pairs $(\rho_o, r_o)$ satisfying~\eqref{eq:67}. The restriction of $\mathcal{S}$ to $[0,T] \times B_M$ is Lipschitz continuous with respect to the modulus in $t$ and the norm~\eqref{eq:66} in $(\rho_o,r_o)$ with a Lipschitz constant that, by~\eqref{eq:81}, \eqref{eq:76} and~\eqref{eq:77} results of the type
  \begin{displaymath}
    \mathcal{C}
    \left(
      \norma{\eta}_{\W3\infty (\reali; \reali^k)},
      \norma{v_{NL}}_{\W3\infty (\reali;\reali^k)},
      R_L,
      V_L,
      \norma{v'_L}_{\W1\infty (\reali;\reali)},
      M,
      T
    \right) \,.
  \end{displaymath}
  By uniform continuity, $\mathcal{S}$ can be extended to $[0,T] \times \overline{B}_M$, where $\overline{B}_M$ is the closure of $B_M$ with respect to the norm~\eqref{eq:66}, which is
  \begin{equation}
    \label{eq:80}
    \overline{B}_M =
    \left\{
      (\rho, r) \in \cBV1 (\reali;\reali^k) \times \cBV{} (\reali; \reali)
      \colon
      \norma{\rho}_{\cBV1(\reali; \reali^k)}
      +
      \norma{r}_{\cBV{} (\reali;\reali)} \leq M
    \right\}
  \end{equation}
by Lemma~\ref{lem:lui}. The arbitrariness of $T$ and $M$ allow to further extend $\mathcal{S}$ as in~\eqref{eq:79}.

To prove the properties of $\mathcal{S}$, fix $(\rho_o,r_o)$ in the
domain of $\mathcal{S}$ and choose a sequence $(\rho_{o,n}, r_{o,n})$ in
$(\C2 \cap \W2\infty \cap \W21) (\reali;\reali^k) \times (\L1 \cap
\BV{}) (\reali; \reali)$ converging to $(\rho_o,r_o)$ with respect to~\eqref{eq:66}.

For all
$t_1,t_2 \in \reali_+$ we have
$(\mathcal{S}_{t_1} \circ \mathcal{S}_{t_2}) (\rho_{o,n},r_{o,n}) =
\mathcal{S}_{t_1+t_2} (\rho_{o,n},r_{o,n})$, passing to the limit
$n \to +\infty$ we prove~\ref{item:1} using \Cref{lem:GoodSol}.

To prove~\ref{item:23} apply \Cref{lem:GoodSol} with $\epsilon_n = e_n = 0$.

Property~\ref{item:24} states the Lipschitz continuity of $\mathcal{S}$ on $[0,T] \times \overline{B}_M$ as defined in~\eqref{eq:80}. This holds true by the very construction of $\mathcal{S}$, while  the form of the Lipschitz constant follows from~\eqref{eq:81}, \eqref{eq:76} and~\eqref{eq:77}.
Similarly, \ref{item:3} follows from~\eqref{apriori1}, \eqref{apriori3}, \eqref{eq:69} and Lemma~\ref{lem:lui}.

The positivity of $\rho$ follows from~\ref{item:22} in
Lemma~\ref{lem:4}. Corollary~\ref{cor:pos} ensures both the invariance
of $[0,R_L]$ and the conservation of the $\L1$ norm, completing the
proof of~\ref{item:13}.

The proof of \Cref{thm:main} is completed.

\subsection{Proof of Theorem~\ref{thm:unique}}
\label{subsec:proof-theor-refthm}

Fix a positive $\epsilon$ and choose
$\rho_o^\epsilon \in \W21 (\reali; \reali^k)$ such that
\begin{equation}
  \label{eq:39}
  \norma{\rho_o^\epsilon - \rho_o}_{\W11 (\reali; \reali^k)} < \epsilon
  \quad \mbox{ and } \quad
  \norma{\rho_o^\epsilon}_{\W21 (\reali; \reali^k)} \leq
  \norma{\rho_o}_{\cBV1 (\reali; \reali^k)} \,.
\end{equation}
Call $(\rho^\epsilon, r^\epsilon)$ the solution to~\eqref{eq:1} with
initial datum $(\rho_o^\epsilon, r_o)$ constructed in
Theorem~\ref{thm:main}.

Definition~\ref{def:sol} allows to use Lemma~\ref{lem:Linfty} with
$\rho = \rho_1$, $\widehat\rho = \rho_\epsilon$ leading to the
estimate
\begin{equation}
  \label{eq:83}
  \norma{\rho_1 (t) - \rho^\epsilon (t)}_{\W11 (\reali; \reali^k)}
  \leq
  (1+\mathcal{C}_{3}\,t)
  \norma{\rho_o - \rho_o^\epsilon}_{\W11 (\reali;\reali^k)}
  +
  \mathcal{C}_{3} \int_0^t
  \norma{r_1 (\tau) - r^\epsilon (\tau)}_{\L1 (\reali;\reali)} \d\tau
\end{equation}
where $\mathcal{C}_3$ is defined in~\eqref{eq:57}. Similarly,
Definition~\ref{def:sol} also allows to use Lemma~\ref{distloc} with
$r = r^\epsilon$ and $\widehat{r} = r_1$, leading to
\begin{equation}
  \label{eq:84}
  \norma{r_1(t) - r^\epsilon(t)}_{\L1(\reali; \reali)}
  \leq
  \mathcal{C}_{11} \,
  t \, \norma{\rho_1 - \rho^\epsilon}_{\C0([0,t]; \W11(\reali; \reali^k ))}\,.
\end{equation}
where
\begin{displaymath}
  \mathcal{C}_{11} \coloneqq
  \mathcal{C} \left(
    \begin{array}{@{}c@{}}
      R_L,
      \norma{v'_L}_{\W1\infty(\reali; \reali)},
      \norma{\rho^\epsilon}_{\C0([0,t]; \W1\infty(\reali; \reali^k)},
      \\
      \norma{\rho_1}_{\C0([0,t]; \W1\infty(\reali; \reali^k)},
      \norma{\partial_x \rho^\epsilon}_{\L1([0,t]; \W11(\reali; \reali^k))},
      \tv(r_o), t
    \end{array}
  \right) \,.
\end{displaymath}
Remark that the above expression allows to cope with the different
regularities of $\rho_1$ and $\rho^\epsilon$. Proceed as
in~\ref{lem:9} in the proof of Theorem~\ref{thm:main} to estimate the
arguments in the right hand side, obtaining
\begin{displaymath}
  \norma{\rho^\epsilon}_{\C0([0,t]; \W1\infty(\reali; \reali^k)}
  \leq
  \mathcal{C}_6
  \quad \mbox{ and } \quad
  \norma{\partial_x \rho^\epsilon}_{\L1([0,t]; \W11(\reali; \reali^k))}
  \leq
  \mathcal{C}_6
\end{displaymath}
where $\mathcal{C}_6$, defined in~\eqref{eq:85}, is bounded uniformly
in $\epsilon$ and, hence, the same holds for $\mathcal{C}_{11}$.

Insert~\eqref{eq:83} in~\eqref{eq:84} and apply Gronwall Lemma to
obtain
\begin{eqnarray*}
  \norma{r_1(t) - r^\epsilon(t)}_{\L1(\reali; \reali)}
  & \leq
  & (1+\mathcal{C}_{3}\,t) \,
    \norma{\rho_o - \rho_o^\epsilon}_{\W11 (\reali;\reali^k)}
    \exp( \mathcal{C}_{3} \, \mathcal{C}_{11} \, t^2)
  \\
  & \leq
  & (1+\mathcal{C}_{3}\,t) \,
    \exp( \mathcal{C}_{3} \, \mathcal{C}_{11} \, t^2) \, \epsilon \,.
\end{eqnarray*}
Insert the latter estimate in~\eqref{eq:83} and use~\eqref{eq:39} to
get
\begin{displaymath}
  \norma{\rho_1 (t) - \rho^\epsilon (t)}_{\W11 (\reali; \reali^k)}
  \leq
  (1+\mathcal{C}_{3}\,t)
  \left(1
    +
    \mathcal{C}_{3} \, t
    \exp( \mathcal{C}_{3} \, \mathcal{C}_{11} \, t^2)\right)
  \epsilon \,.
\end{displaymath}
Repeat the same procedure replacing $\rho_1$ with $\rho_2$ so that, by
the triangle inequality, we get
\begin{displaymath}
  \norma{\rho_1 (t) - \rho_2 (t)}_{\W11 (\reali; \reali^k)}
  +
  \norma{r_1(t) - r_2(t)}_{\L1(\reali; \reali)}
  \leq
  2 (1+\mathcal{C}_{3}\,t)
  \left(1
    +
    2 \, \mathcal{C}_{3} \, t
    \exp( \mathcal{C}_{3} \, \mathcal{C}_{11} \, t^2)\right)
  \epsilon \,.
\end{displaymath}
Since $\epsilon$ is arbitrary, the proof is completed.

\section*{Acknowledgments}
The authors were partly supported by the GNAMPA~2025 project
\emph{Modellli di Traffico, di Biologia e di Dinamica dei Gas Basati
  su Sistemi di Equazioni Iperboliche}. They also acknowledge the
PRIN~2022 project \emph{Modeling, Control and Games through Partial
  Differential Equations} (CUP~D53D23005620006), funded by the
European Union - Next Generation EU.  MG was partly supported by the
European Union-NextGeneration EU (National Sustainable Mobility Center
CN00000023, Italian Ministry of University and Research Decree
n.~1033-17/06/2022, Spoke 8) and by the Project funded under the
National Recovery and Resilience Plan (NRRP) of Italian Ministry of
University and Research funded by the European Union-NextGeneration
EU. Award Number: ECS\_00000037, CUP: H43C22000510001,
MUSA-Multilayered Urban Sustainability Action. CN was supported by
the PNRR project \emph{Ricerca DM 118/2023}, CUP~F13C23000360007.

\section*{Conflict  of Interest}
The authors declare no conflicts of interest in this paper.

{\small

  \bibliographystyle{abbrv}

  \bibliography{ColomboGaravelloNocita}

}

\end{document}